\documentclass[11pt, twoside, leqno]{article}
\usepackage{mathrsfs}
\usepackage{amssymb}
\usepackage{amsmath}
\usepackage{mathrsfs}
\usepackage{amsthm}
\usepackage{amsfonts}
\usepackage{color}
\usepackage{latexsym}
\usepackage{txfonts}
\usepackage{indentfirst}
\usepackage{anysize}
\usepackage{tikz}

\allowdisplaybreaks

\newcounter{assum}

\newtheorem{theorem}{Theorem}[section]
\newtheorem{lemma}[theorem]{Lemma}
\newtheorem{corollary}[theorem]{Corollary}
\newtheorem{proposition}[theorem]{Proposition}
\theoremstyle{definition}
\newtheorem{remark}[theorem]{Remark}

\newtheorem{definition}[theorem]{Definition}
\renewcommand{\appendix}{\par
	\setcounter{section}{0}%
	\setcounter{subsection}{0}%
	\setcounter{subsubsection}{0}%
	\gdef\thesection{\@Alph\c@section}%
	\gdef\thesubsection{\@Alph\c@section.\@arabic\c@subsection}%
	\gdef\theHsection{\@Alph\c@section.}%
	\gdef\theHsubsection{\@Alph\c@section.\@arabic\c@subsection}%
	\csname appendixmore\endcsname
}

\numberwithin{equation}{section}

\begin{document}

\title{\bf\Large Parabolic Muckenhoupt Weights Characterized
by Parabolic Fractional Maximal and Integral Operators with Time Lag
\footnotetext{\hspace{-0.35cm} 2020
{\it Mathematics Subject Classification}.
Primary 42B20; Secondary 47A30, 42B25, 42B35, 42B37, 35K05. \endgraf
{\it Key words and phrases.}
parabolic Muckenhoupt weight, time lag,
parabolic fractional maximal operator,
parabolic fractional integral,
parabolic Welland inequality.
\endgraf
This project is partially supported by the National
Key Research and Development Program of China
(Grant No. 2020YFA0712900),
the National Natural Science Foundation of China
(Grant Nos. 12431006 and 12371093),
and the Fundamental Research Funds
for the Central Universities (Grant No. 2233300008).}}
\date{}
\author{Weiyi Kong, Dachun Yang\footnote{Corresponding author,
E-mail: \texttt{dcyang@bnu.edu.cn}/{\color{red}
October 6, 2024}/Final version.},\ \ Wen Yuan and Chenfeng Zhu}

\maketitle

\vspace{-0.7cm}

\begin{center}
\begin{minipage}{13cm}
{\small {\bf Abstract}\quad In this article, motivated by the
regularity theory of the solutions of
doubly nonlinear parabolic partial differential equations
the authors introduce the off-diagonal two-weight version of
the parabolic Muckenhoupt class with time lag.
Then the authors introduce the uncentered
parabolic fractional maximal operator with time lag and
characterize its two-weighted boundedness (including the
endpoint case) via these weights under an extra mild
assumption (which is not necessary for one-weight case).
The most novelty of this article
exists in that the authors further introduce a new parabolic shaped
domain and its corresponding parabolic fractional integral
with time lag and, moreover, applying the aforementioned
two-weighted boundedness of the uncentered
parabolic fractional maximal operator with time lag,
the authors characterize the (two-)weighted boundedness
(including the endpoint case) of these parabolic fractional
integrals in terms of the off-diagonal (two-weight) parabolic
Muckenhoupt class with time lag; as applications, the authors
further establish a parabolic weighted
Sobolev embedding and a priori estimate for
the solution of the heat equation. The key tools to
achieve these include the parabolic Calder\'on--Zygmund-type
decomposition, the chaining argument, and
the parabolic Welland inequality which is
obtained by making the utmost of the geometrical relation
between the parabolic shaped domain and the parabolic rectangle.
}
\end{minipage}
\end{center}

\vspace{0.2cm}

\tableofcontents

\vspace{0.2cm}

\section{Introduction}

The Muckenhoupt class, introduced by Muckenhoupt
\cite{m(tams-1972)}, is of fundamental importance in harmonic
analysis and partial differential equations. Let
$q\in(1,\infty)$ and $\omega$ be a \emph{weight} on $\mathbb{R}^n$,
namely a nonnegative locally integrable function.
It is well known that some classical operators (for instance,
the Hardy--Littlewood maximal operator and the
Calder\'on--Zygmund operators) are bounded on the
\emph{weighted Lebesgue space}
\begin{align*}
L^q(\mathbb{R}^n,\omega):=\left\{f:\
\|f\|_{L^q(\mathbb{R}^n,\omega)}:=\left[\int_{\mathbb{R}^n}
|f(x)|^q\omega(x)\,dx\right]^{\frac{1}{q}}<\infty\right\}
\end{align*}
if and only if $\omega$ belongs to the \emph{Muckenhoupt class
$A_q(\mathbb{R}^n)$}, that is,
\begin{align}\label{20241002.2236}
[\omega]_{A_q(\mathbb{R}^n)}:=\sup_{Q\subset\mathbb{R}^n}
\frac{1}{|Q|}\int_Q\omega(y)\,dy
\left\{\frac{1}{|Q|}\int_Q\left[\omega(y)\right]
^{\frac{1}{1-q}}\,dy\right\}^{q-1}<\infty,
\end{align}
where the supremum is taken over all cubes
$Q\subset\mathbb{R}^n$ whose edges are all parallel to the
coordinate axis. Furthermore, the Muckenhoupt weights have
deep connections with the elliptic partial differential
equation
\begin{align}\label{20241001.2203}
\mathrm{div}\left(|\nabla u|^{p-2}\nabla u\right)=0,
\end{align}
where $p\in(1,\infty)$. So far, the Muckenhoupt classes and
the theory of weighted function spaces have been developed in
a comprehensive manner; see, for instance, \cite{cm(sm-2021),
cmp(jfa-2004), cmp-Extrapolation, cmp(am-2012),
h(am-2012), h(plms-2018), hp(apde-2013), hpr(jfa-2012)}.
Moreover, there exists a well-established theory related to
the Muckenhoupt weights with applications in partial
differential equations; see, for instance,
\cite{fks(cpde-1982), hk(mams-2000), hs(aif-2001),
ll(jmaa-2022), lw(jmaa-2020), m(cpam-1961)}.

From the perspective of partial differential equations, in
addition to the Muckenhoupt classes related to
\eqref{20241001.2203}, there also exist parabolic Muckenhoupt
classes with time lag, introduced by Kinnunen and Saari
\cite{ks(apde-2016)}, tailored to the doubly nonlinear
parabolic partial differential equation
\begin{align}\label{20240903.1403}
\frac{\partial}{\partial t}\left(|u|^{p-2}u\right)-
\mathrm{div}\left(|\nabla u|^{p-2}\nabla u\right)=0.
\end{align}
Here and thereafter, we \emph{always fix} $p\in(1,\infty)$. The
definition of parabolic Muckenhoupt weights with time lag is
based on the following definition of parabolic rectangles.
For any $x\in\mathbb{R}^n$ and $L\in(0,\infty)$, let
$Q(x,L)$ be the \emph{cube} in $\mathbb{R}^n$ centered at $x$
with edge length $2L$.

\begin{definition}\label{parabolic rectangles}
Let $(x,t)\in\mathbb{R}^{n+1}$
and $L\in(0,\infty)$. A \emph{parabolic
rectangle} $R$ centered at
$(x,t)$ with edge length
$L$ is defined by setting
\begin{align*}
R:=R(x,t,L):=Q(x,L)\times\left(t-L^{p},t+L^{p}\right).
\end{align*}
Let $\gamma\in[0,1)$. The \emph{$\gamma$-upper part}
$R^{+}(\gamma)$ and the \emph{$\gamma$-lower part}
$R^{-}(\gamma)$ of $R$ are defined, respectively, by setting
\begin{align*}
R^+(\gamma):=Q(x,L)\times(t+\gamma L^p,t+L^{p})\ \
\mbox{and}\ \ R^-(\gamma):=Q(x,L)
\times(t-L^{p},t-\gamma L^p),
\end{align*}
where $\gamma$ is called the \emph{time lag}.
\end{definition}

Denote by $\mathcal{R}_{p}^{n+1}$ the set of all parabolic
rectangles in $\mathbb{R}^{n+1}$. For any locally integrable
function $f$ on $\mathbb{R}^{n+1}$ and for any measurable set
$A\subset\mathbb{R}^{n+1}$ with $|A|\in(0,\infty)$, let
\begin{align*}
\fint_{A} f:=\frac{1}{|A|}\int_{A}f.
\end{align*}
Here and thereafter, we \emph{always omit} the differential $dx\,dt$
in all integral representations to simplify the presentation if
there exists no confusion. The following is the definition of
parabolic Muckenhoupt classes with time lag; see also
\cite[Definition 3.2]{ks(apde-2016)}.

\begin{definition}
Let $\gamma\in[0,1)$ and $q\in(1,\infty)$. The \emph{parabolic
Muckenhoupt class $A_q^+(\gamma)$} is defined to be the set of
all nonnegative locally integrable functions $\omega$ on
$\mathbb{R}^{n+1}$ such that
\begin{align}\label{20241002.2227}
[\omega]_{A_q^+(\gamma)}:=\sup_{R\in\mathcal{R}_p^{n+1}}
\fint_{R^-(\gamma)}\omega\left[\fint_{R^+(\gamma)}
\omega^{\frac{1}{1-q}}\right]^{q-1}<\infty.
\end{align}
If the above condition is satisfied with the direction of the
time axis reversed, then $\omega\in A_q^-(\gamma)$ which is
also called the \emph{parabolic Muckenhoupt class}
and consists of all such $\omega$.
\end{definition}

Different from the classical case, in
\eqref{20241002.2227}, Euclidean cubes in
\eqref{20241002.2236} are substituted by parabolic rectangles,
which respects the natural geometry of \eqref{20240903.1403}.
Indeed, if $u(x,t)$ is a solution of \eqref{20240903.1403},
then so does $u(\lambda x,\lambda^pt)$ for any
$\lambda\in(0,\infty)$. It turns out in Moser
\cite{m(cpam-1964), m(cpam-1964)2} and Trudinger
\cite{t(cpam-1968)} that any nonnegative weak solution $u$ of
\eqref{20240903.1403} satisfies a scale and location invariant
Harnack inequality, that is, for any given $\gamma\in(0,1)$,
there is a positive constant $C$ such that, for any
$R\in\mathcal{R}_p^{n+1}$,
\begin{align*}
\mathop\mathrm{ess\,sup}_{(x,t)\in R^-(\gamma)}u(x,t)\leq
C\mathop\mathrm{ess\,inf}_{(x,t)\in R^+(\gamma)}u(x,t),
\end{align*}
where the time lag $\gamma$ appears naturally. The Harnack
inequality further implies that any nonnegative weak solution
of \eqref{20240903.1403} is a parabolic Muckenhoupt weight
with time lag. Kinnunen and Saari \cite{ks(apde-2016)} also
introduced the centered parabolic Hardy--Littlewood maximal
operators with time lag and showed that these operators are
bounded on the weighted Lebesgue space if and only if the
weight belongs to the corresponding parabolic Muckenhoupt
class with time lag. Their results in \cite{ks(apde-2016)}
were streamlined and complemented by Kinnunen and
Myyryl\"ainen \cite{km(am-2024)} in which they
replaced the centered maximal operator by the uncentered
version to include the endpoint case. On the other hand,
as proved in \cite[Lemma 7.4]{ks(apde-2016)},
the parabolic Muckenhoupt classes with time lag
give a Coifman--Rochberg type characterization of
the function space with parabolic bounded mean oscillation
which was explicitly defined by Fabes and Garofalo
\cite{fg(pams-1985)} and is essential in the regularity theory
for \eqref{20240903.1403}. We refer to \cite{bdks(jmpa-2020),
bdl(jfa-2021), bdls(rmi-2023), bds(pdea-2022),
bhss(cvpde-2021), gv(jafa-2006), kk(ma-2007), ksu(iumj-2012),
m(cpam-1971), v(mm-1992)} for more studies about
\eqref{20240903.1403}, to \cite{a(tams-1988), kmy(ma-2023),
kyyz-2024, my(mz-2024), s(rmi-2016), s(ampa-2018)} for more
studies of parabolic function spaces, and to
\cite{km(am-2024), 2310.00370, kmyz(pa-2023), ks(na-2016),
ks(apde-2016), mhy(fm-2023)} for recent studies of
the parabolic Muckenhoupt classes with time lag.

The other motivation to study the parabolic Muckenhoupt classes
with time lag is due to the theory of the one-sided
Muckenhoupt classes introduced by Sawyer \cite{s(tams-1986)}
in connection with ergodic theory. Recall that, for any
$q\in(1,\infty)$, the \emph{one-sided Muckenhoupt class
$A_q^+(\mathbb{R})$} is defined to be the set of all
nonnegative locally integrable functions $\omega$ on
$\mathbb{R}$ such that
\begin{align*}
[\omega]_{A_q^+(\mathbb{R})}:=
\sup_{x\in\mathbb{R},\,h\in(0,\infty)}\frac{1}{h}
\int_{x-h}^x\omega(y)\,dy\left\{\frac{1}{h}
\int_x^{x+h}\left[\omega(y)\right]
^{\frac{1}{1-q}}\,dy\right\}^{q-1}<\infty.
\end{align*}
Actually, the parabolic Muckenhoupt classes with time lag
are higher dimensional generalizations of the one-sided
Muckenhoupt classes in some sense. The one-sided weighted
theory has been extensively investigated; see, for instance,
\cite{afm(pams-1997), ck(sm-2018),
cno(sm-1993), m(pams-1993), md(pams-1993), md(jlms-1994),
md(cm-2015), mod(tams-1990), mpd(cjm-1993), rd(cj-2001)}.
There also exists several inspirational studies about
higher-dimensional extensions of the one-sided weights
and related topics; see, for instance, \cite{b(jmaa-2011),
b(cj-2022), fmo(tams-2011), gm(em-2022), gs(mia-2019),
lo(pm-2010), lm(ruma-2017), o(pams-2005)}.

On the other hand, both the fractional maximal operators and
the fractional integral operators occupy an important position
in potential theory, harmonic analysis, and partial
differential equations; see, for instance, \cite{adams-book,
bk(jia-2024), cjy(tjm-2024), g(p-2024), turesson-book}.
Recall that, for any given $\beta\in(0,n)$, the
\emph{fractional maximal operator $M_\beta$} and the
\emph{fractional integral operator $I_\beta$} are defined,
respectively, by setting, for any locally integrable function
$f$ on $\mathbb{R}^n$ and any $x\in\mathbb{R}^n$,
\begin{align*}
M_\beta(f)(x):=\sup_{L\in(0,\infty)}\frac 1{|Q(x,L)|
^{1-\frac{\beta}{n}}}\int_{Q(x,L)}|f(y)|\,dy\ \ \mbox{and}\ \
I_\beta(f)(x):=\int_{\mathbb{R}^n}\frac{f(y)}{|x-y|^{n-\beta}}
\,dy.
\end{align*}
Let $\beta\in(0,n)$, $1<r\leq q<\infty$ with
$\frac1r-\frac1q=\frac{\beta}{n}$, and $\omega$ be a weight on
$\mathbb{R}^n$. It is well known that the fractional
integral operator $I_\beta$ (or the fractional maximal operator $M_\beta$) is
bounded from $L^r(\mathbb{R}^n,\omega^r)$ to
$L^q(\mathbb{R}^n,\omega^q)$ if and only if $\omega$
belongs to the \emph{off-diagonal Muckenhoupt class
$A_{r,q}(\mathbb{R}^n)$}, that is,
\begin{align}\label{20241003.1040}
[\omega]_{A_{r,q}(\mathbb{R}^n)}:=
\sup_{Q\subset\mathbb{R}^n}\frac{1}{|Q|}\int_{Q}
[\omega(x)]^q\,dx\left\{\frac{1}{|Q|}\int_{Q}
[\omega(x)]^{-r'}\,dx\right\}^{\frac{q}{r'}}<\infty,
\end{align}
where the supremum is taken over all cubes
$Q\subset\mathbb{R}^n$ whose edges are all parallel to the
coordinate axis. We refer to \cite{cm(ijm-2013),
lushanzhen-book, mw(tams-1974), w(pams-1975)} for more studies
on the one-weight case and to \cite{cm(ijm-2013),
cs(jmaa-2024), mp(itsf-2018),kmz(bjma-2018), lmpt(jfa-2010),
p(iumj-1994), s(tams-1988), sw(ajm-1992)} for more
investigations on the two-weight case of the weighted
boundedness of the fractional integral operators and the
fractional maximal operators. As for the
one-sided situation, Andersen and Sawyer
\cite{as(tams-1988)} obtained the characterizations
of the weighted boundedness of one-sided fractional maximal
operators and the Weyl (and the Riemann--Liouville) fractional
integral operators in terms of the one-sided off-diagonal
Muckenhoupt classes. For more studies of the one-sided fractional maximal
operators and the one-sided fractional integral operators,
see, for instance, \cite{dd(cmuc-2004), l(cjm-1997),
mr(prsea-2000), 2406.11663v2, od(pm-2003)}. In the parabolic
setting, Ma et al. \cite{mhy(fm-2023)} introduced the
centered parabolic fractional maximal operator with time lag
and showed that it is bounded on the weighted Lebesgue spaces
if and only if the weight belongs to the corresponding
off-diagonal parabolic Muckenhoupt class with time lag.
Inspired by these, it is natural to ask what is the
most appropriate definition of parabolic
fractional integral operators with time lag and whether or not the
weighted boundedness of such operators can characterize the
parabolic off-diagonal Muckenhoupt class with time lag.
We give positive answers to these two questions in this
article (see Definition \ref{parabolic Riesz potentials} and Theorems
\ref{20241004.1030} and \ref{20241004.1125}).

The main goal of this article are twofold.
One is to generalize the parabolic Muckenhoupt class
with time lag
in \cite{km(am-2024),ks(apde-2016)}
and the off-diagonal parabolic Muckenhoupt class
with time lag in \cite{mhy(fm-2023)}
to the two-weight case. The other is to characterize such two-weight
parabolic Muckenhoupt class with time lag
in terms of the weighted boundedness of
some fractional operators, namely
the centered and the uncentered
parabolic fractional maximal operators with time lag
and the parabolic fractional integral operators with time lag.
More precisely, inspired by the regularity theory
of \eqref{20240903.1403}, we
introduce the off-diagonal two-weight version of
the parabolic Muckenhoupt class with time lag.
Then we introduce the uncentered
parabolic fractional maximal operator with time lag and
characterize its two-weighted boundedness (including the
endpoint case) via these weights under an extra mild
assumption (which is not necessary for one-weight case).
The most novelty of this article
exists in that we further introduce a new parabolic shaped
domain and its corresponding parabolic fractional integral
with time lag and, moreover, applying the aforementioned
two-weighted boundedness of the uncentered
parabolic fractional maximal operator with time lag,
we characterize the (two-)weighted boundedness
(including the endpoint case) of these parabolic fractional
integrals in terms of the off-diagonal (two-weight) parabolic
Muckenhoupt class with time lag; as applications, we
further establish a parabolic weighted
Sobolev embedding and a priori estimate for
the solution of the heat equation. The key tools to
achieve these include the parabolic Calder\'on--Zygmund-type
decomposition, the chaining argument, and
the parabolic Welland inequality which is
obtained by making the utmost of the geometrical relation
between the parabolic shaped domain and the parabolic rectangle.

The organization of the remainder of this article is as
follows.

In Section \ref{section2}, we introduce the concept
of parabolic Muckenhoupt two-weight classes
with time lag. Several elementary
properties of parabolic Muckenhoupt two weights, such as
the nested property, the duality
property, and the forward-in-time doubling property, are
presented. Moreover, we give a characterization of
the parabolic Muckenhoupt two-weight class with time lag in
the endpoint case via the uncentered parabolic maximal operator
with time lag; see Proposition \ref{M- and TA1+}.

In Section \ref{section3}, under an extra mild
assumption (which is not necessary for one-weight case),
applying the chaining argument we
show that the parabolic Muckenhoupt two-weight class is
independent of the choice of the time lag; see Theorem
\ref{independence of time lag 2}. As an application,
we obtain the self-improving property
of the parabolic Muckenhoupt two-weight; see Corollary
\ref{self-improving property}.

Section \ref{section4} is devoted to characterizing the
parabolic Muckenhoupt two-weight class with time lag via the
uncentered parabolic fractional maximal operator with time
lag; see Theorem \ref{weak type inequality}. To achieve this, we
utilize a covering argument in \cite{ks(apde-2016)} and
Theorem \ref{independence of time lag 2} to change
the time lag. As a corollary, we prove the strong-type parabolic
weighted norm inequality for the uncentered parabolic
fractional maximal operator with time lag; see Corollary
\ref{weighted inequality uncentered}. All these
results are both the fractional variants of the counterparts
in \cite{km(am-2024), ks(apde-2016)} and the generalization
of the counterpart of \cite{mhy(fm-2023)} from
the centered one to the uncentered one.
We also obtain the weak-type parabolic two-weighted
norm inequality for the centered parabolic fractional
maximal operator with time lag; see Theorem
\ref{weak type inequality centered}.
Notice that Theorems \ref{weak type inequality} and
\ref{weak type inequality centered}
are respectively the generalizations of
the counterparts of \cite{km(am-2024),mhy(fm-2023)}
from the one weight to two weights.

In Section \ref{section5},
based on a new parabolic shaped domain, we introduce the
parabolic forward in time and back in time fractional integral
operators with time lag; see Definition
\ref{parabolic Riesz potentials}. Then we establish the
pointwise relation between the centered parabolic fractional
maximal operator with time lag and the parabolic fractional
integral operator with time lag by showing a parabolic
Welland type inequality; see Lemmas \ref{pointwise control} and
\ref{Welland inequality}. Using this, we prove the weak-type
parabolic two-weighted inequality and the strong-type
parabolic weighted inequality for the parabolic fractional
integrals with time lag; see Theorems
\ref{weak type Riesz potential} and
\ref{weighted inequality of Riesz potential} and Corollary
\ref{weak type Riesz potential 2}.

The aims of Section \ref{section6} are two
aspects. One is to establish the parabolic weighted
boundedness of the parabolic Riesz potentials introduced in
\cite{j(jmm-1968/69)} for a special class of parabolic
Muckenhoupt weights with time lag; see Theorem
\ref{parabolic Riesz potantial weighted boundedness Gopala Rao}.
Applying this, we establish a parabolic weighted Sobolev
embedding theorem and obtain a priori estimate
for the solution of the heat equation; see, respectively,
Corollaries \ref{Sobolev embedding} and \ref{priori estimate}.

At the end of this introduction, we make some conventions on
notation. Throughout this article, let
$\mathbb{N}:=\{1,2,\ldots\}$, $\mathbb{Z}_{+}:=\mathbb{N}
\cup\{0\}$, and $\mathbb{R}^{n+1}_+:=\mathbb{R}^n
\times(0,\infty)$. For any $s\in\mathbb{R}$, the symbol
$\lceil s\rceil$ denotes the smallest integer not less than
$s$. For any $r\in[1,\infty]$, let $r'$ be the conjugate
number of $r$, that is, $\frac1r+\frac{1}{r'}=1$. For
any $x:=(x_1,\dots,x_n)\in\mathbb{R}^n$,
let $\|x\|_\infty:=\max\{|x_1|,\dots,|x_n|\}$ and
$|x|:=\sqrt{|x_1|^2+\dots+|x_n|^2}$. For any
$A\subset\mathbb{R}^{n+1}$, let
\begin{align*}
\text{pr}_{x}(A):=\{x\in
\mathbb{R}^{n}:\ \text{there exists}
\ t\in\mathbb{R}\ \text{such that}
\ (x,t)\in A\}
\end{align*}
and
\begin{align*}
\text{pr}_{t}(A):=\{t\in
\mathbb{R}:\ \text{there exists}
\ x\in\mathbb{R}^{n}\ \text{such that}
\ (x,t)\in A\}
\end{align*}
be the projections of $A$, respectively, onto the space
$\mathbb{R}^{n}$ and the time axis $\mathbb{R}$. Let
$\mathbf{0}$ denote the origin of $\mathbb{R}^{n}$.
For any $A\subset\mathbb{R}^{n+1}$ and $a\in\mathbb{R}$,
let
$A-(\mathbf{0},a):=\{(x,t-a):\ (x,t)\in A\}$.
For any measurable set $A\subset\mathbb{R}^{n+1}$, we denote
by $\vert A\vert$ its $(n+1)$-dimensional Lebesgue measure.
For any $R\in\mathcal{R}_{p}^{n+1}$, denote by $(x_R,t_R)$ its
center and by $l(R)$ its edge length. The \emph{top} of
$R\in\mathcal{R}_{p}^{n+1}$ means
$Q(x_R,l(R))\times\{t_R+[l(R)]^{p}\}$ and the \emph{bottom} of
$R$ means $Q(x_R,l(R))\times\{t_R-[l(R)]^{p}\}$.
Let $L_{\mathrm{loc}}^1(\mathbb{R}^{n+1})$
(resp. $L_{\mathrm{loc}}^1(\mathbb{R}^{n+1}_+)$) be the set of
all locally integrable functions on $\mathbb{R}^{n+1}$
(resp. $\mathbb{R}^{n+1}_+$).
The symbol $f\lesssim g$ means that there exists a positive
constant $C$ such that $f\le Cg$ and, if $f\lesssim g\lesssim
f$, we then write $f\sim g$. Finally, when we show a theorem
(and the like), in its proof we always use the same symbols as
those appearing in the statement itself of that theorem (and the like).

\section{Parabolic Muckenhoupt Two-Weight Classes with Time Lag}
\label{section2}

In this section, we introduce the concept of
parabolic Muckenhoupt two-weight classes with time lag and
present their several basic properties. We begin with
the following definition.

\begin{definition}\label{parabolic Muckenhoupt class}
Let $\gamma\in[0,1)$.
\begin{enumerate}
\item[\rm(i)] Let $1<r\leq q<\infty$. The
\emph{parabolic Muckenhoupt two-weight class
$TA_{r,q}^+(\gamma)$ with time lag} is defined to be the set
of all pairs $(u,v)$ of nonnegative functions on
$\mathbb{R}^{n+1}$ such that
\begin{align*}
[u,v]_{TA_{r,q}^+(\gamma)}:=
\sup_{R\in\mathcal{R}_{p}^{n+1}}
\fint_{R^-(\gamma)}u^q
\left[\fint_{R^+(\gamma)}v^
{-r'}\right]^{\frac{q}{r'}}<\infty.
\end{align*}
If the above condition holds with the time axis reversed,
that is,
\begin{align*}
[u,v]_{TA_{r,q}^-(\gamma)}:=
\sup_{R\in\mathcal{R}_{p}^{n+1}}\fint_{R^+(\gamma)}u^q
\left[\fint_{R^-(\gamma)}v^
{-r'}\right]^{\frac{q}{r'}}<\infty,
\end{align*}
then $(u,v)\in TA_{r,q}^-(\gamma)$ which is
also called the \emph{parabolic Muckenhoupt
two-weight class with time lag} and consists of all such
$(u,v)$.

\item[\rm(ii)] Let $q\in[1,\infty)$. The
\emph{parabolic Muckenhoupt two-weight class
$TA_{1,q}^+(\gamma)$ with time lag} is defined to be the set
of all pairs $(u,v)$ of nonnegative functions on
$\mathbb{R}^{n+1}$ such that
\begin{align*}
[u,v]_{TA_{1,q}^+(\gamma)}:=
\sup_{R\in\mathcal{R}_{p}^{n+1}}
\fint_{R^-(\gamma)}u^q
\left[\mathop\mathrm{ess\,inf}_{(x,t)\in
R^+(\gamma)}v(x,t)\right]^{-q}<\infty.
\end{align*}
If the above condition holds with the time axis reversed, that
is,
\begin{align*}
[u,v]_{TA_{1,q}^-(\gamma)}:=
\sup_{R\in\mathcal{R}_{p}^{n+1}}
\fint_{R^+(\gamma)}u^q
\left[\mathop\mathrm{ess\,inf}_{(x,t)\in
R^-(\gamma)}v(x,t)\right]^{-q}<\infty,
\end{align*}
then $(u,v)\in TA_{1,q}^-(\gamma)$ which is
also called the \emph{parabolic Muckenhoupt
two-weight class with time lag} and consists of all such $(u,v)$.
\end{enumerate}
\end{definition}

\begin{remark}\label{remark on parabolic weight}
\begin{enumerate}
\item[\rm(i)]
Let $\gamma\in[0,1)$ and
$1<r\leq q<\infty$.
Recall that the \emph{parabolic Muckenhoupt class
$A_{r,q}^+(\gamma)$ with time lag},
introduced in \cite{mhy(fm-2023)}, is defined to be the set of
all nonnegative functions $\omega$ on $\mathbb{R}^{n+1}$ such
that
\begin{align*}
\sup_{R\in\mathcal{R}_p^{n+1}}
\fint_{R^-(\gamma)}\omega^q
\left[\fint_{R^+(\gamma)}\omega
^{-r'}\right]^{\frac{q}{r'}}<\infty
\end{align*}
and the \emph{parabolic Muckenhoupt class $A_1^+(\gamma)$ with
time lag}, introduced in \cite{km(am-2024),ks(apde-2016)},
is defined to be the set of all nonnegative functions $\omega$
on $\mathbb{R}^{n+1}$ such that
\begin{align*}
\sup_{R\in\mathcal{R}_p^{n+1}}
\fint_{R^-(\gamma)}\omega\left[\mathop\mathrm{ess\,inf}
_{(x,t)\in R^+(\gamma)}\omega(x,t)\right]^{-1}<\infty.
\end{align*}
Then we are easy to prove that,
if $\omega\in A_{r,q}^+(\gamma)$,
then $(\omega,\omega)\in TA_{r,q}^+(\gamma)$
and, if $w\in A_1^+(\gamma)$, then, for any $s\in[1,\infty)$,
$(w^{\frac1s},w^{\frac1s})\in TA_{1,s}^+(\gamma)$.
Thus, the parabolic Muckenhoupt two-weight class with time lag
in Definition~\ref{parabolic Muckenhoupt class}
is indeed a generalization of both
$A_{r,q}^+(\gamma)$ and $A_1^+(\gamma)$.

\item[\rm(ii)]
Let $0\leq\gamma_1\leq\gamma_2<1$ and $1\leq r\leq q<\infty$.
Then $TA_{r,q}^+(\gamma_1)\subset TA_{r,q}^+(\gamma_2)$.
Moreover, for any $(u,v)\in TA_{r,q}^+(\gamma_1)$,
$[u,v]_{TA_{r,q}^+(\gamma_2)}\leq
(\frac{1-\gamma_1}{1-\gamma_2})^{1+\frac{q}{r'}}
[u,v]_{TA_{r,q}^+(\gamma_1)}$.
\end{enumerate}
\end{remark}

For the parabolic Muckenhoupt two-weight class with time lag,
we have the following nested property
which can be directly deduced from Definition~\ref{parabolic Muckenhoupt class}
and the H\"older inequality.

\begin{proposition}\label{nested property}
Let $\gamma\in[0,1)$.
\begin{enumerate}
\item[\rm(i)] Let $1\leq\widetilde{r}<r\leq q<\infty$.
Then $TA_{\widetilde{r},q}^+(\gamma)\subset
TA_{r,q}^+(\gamma)$. Moreover, for any
$(u,v)\in TA_{\widetilde{r},q}^+(\gamma)$,
$[u,v]_{TA_{r,q}^+(\gamma)}
\leq[u,v]_{TA_{\widetilde{r},q}^+(\gamma)}$.

\item[\rm(ii)] Let $1\leq r\leq\widetilde{q}<q<\infty$.
Then $TA_{r,q}^+(\gamma)\subset
TA_{r,\widetilde{q}}^+(\gamma)$. Moreover, for any $(u,v)\in
TA_{r,q}^+(\gamma)$, $[u,v]_{TA_{r,\widetilde{q}}^+(\gamma)}
\leq[u,v]_{TA_{r,q}^+(\gamma)}^{\frac{\widetilde{q}}{q}}$.
\end{enumerate}
\end{proposition}

Next, we introduce
the uncentered parabolic fractional maximal function with time lag.

\begin{definition}\label{parabolic maximal operators}
Let $\gamma,\beta\in[0,1)$. For any $f\in
L_{\mathrm{loc}}^{1}(\mathbb{R}^{n+1})$,
the \emph{uncentered forward in time parabolic fractional
maximal function $M^{\gamma+}_\beta(f)$ with time lag} and the \emph{uncentered back
in time parabolic fractional maximal function
$M^{\gamma-}_\beta(f)$ with time lag} of $f$
are defined, respectively, by setting,
for any $(x,t)\in\mathbb{R}^{n+1}$,
\begin{align*}
M^{\gamma+}_\beta(f)(x,t):=\sup_{\genfrac{}{}{0pt}{}{R\in
\mathcal{R}_{p}^{n+1}}{(x,t)\in R^-(\gamma)}}
\left|R^+(\gamma)\right|^\beta
\fint_{R^+(\gamma)}|f|
\end{align*}
and
\begin{align*}
M^{\gamma-}_\beta(f)(x,t):=\sup_{\genfrac{}{}{0pt}{}{R\in
\mathcal{R}_{p}^{n+1}}{(x,t)\in R^+(\gamma)}}
\left|R^-(\gamma)\right|^\beta
\fint_{R^-(\gamma)}|f|.
\end{align*}
\end{definition}

Notice that, for any $\gamma\in[0,1)$, $M_0^{\gamma+}$
coincides with the uncentered parabolic parabolic forward
in time maximal operator $M^{\gamma+}$ with time lag
introduced in \cite[Definition 2.2]{km(am-2024)}.
We present the following characterization of
the parabolic Muckenhoupt
two-weight class $TA_{1,q}^+(\gamma)$
in terms of the uncentered back in time parabolic
maximal operator $M^{\gamma-}_0$.

\begin{proposition}\label{M- and TA1+}
Let $\gamma\in[0,1)$, $q\in[1,\infty)$, and $(u,v)$ be a pair
of nonnegative functions on $\mathbb{R}^{n+1}$.
Then $(u,v)\in TA_{1,q}^+(\gamma)$ if and only if
there exists $C\in(0,\infty)$ such that,
for almost every $(x,t)\in\mathbb{R}^{n+1}$,
\begin{align}\label{1401}
M^{\gamma-}_0(u^q)(x,t)\leq C[v(x,t)]^q.
\end{align}
\end{proposition}

\begin{proof}
We first show the sufficiency.
Assume that $(u,v)$ satisfies \eqref{1401}.
By this and Definition
\ref{parabolic maximal operators},
we find that, for any $R\in\mathcal{R}_p^{n+1}$ and for
almost every $(x,t)\in R^+(\gamma)$,
\begin{align*}
\fint_{R^-(\gamma)}u^q\leq M^{\gamma-}_0(u^q)(x,t)\leq
C[v(x,t)]^q.
\end{align*}
Taking the essential infimum over all $(x,t)\in R^+(\gamma)$,
we then obtain
\begin{align*}
\fint_{R^-(\gamma)}u^q\leq C\left[\mathop\mathrm{ess\,inf}
_{(x,t)\in R^+(\gamma)}v(x,t)\right]^q
\end{align*}
and hence $(u,v)\in TA_{1,q}^+(\gamma)$ with
$[u,v]_{TA_{1,q}^+(\gamma)}\leq C$. This finishes the proof
of the sufficiency.

Now, we prove the necessity. Assume that $(u,v)\in
TA_{1,q}^+(\gamma)$ and let
\begin{align*}
\mathcal{N}:=\left\{(x,t)\in\mathbb{R}^{n+1}:\
M^{\gamma-}_0\left(u^q\right)(x,t)>[u,v]_{TA_{1,q}^+(\gamma)}
[v(x,t)]^q\right\}.
\end{align*}
To show that \eqref{1401} holds almost
everywhere for some $C\in(0,\infty)$,
it suffices to prove that $|\mathcal{N}|=0$.
Indeed, from Definition \ref{parabolic maximal operators},
we infer that, for any given $(x,t)\in\mathcal{N}$,
there exist $R_{(x,t)}\in\mathcal{R}_p^{n+1}$
and $\epsilon\in(0,\infty)$ such that $(x,t)\in
R_{(x,t)}^+(\gamma)$ and
\begin{align}\label{20240723.2145}
\fint_{R_{(x,t)}^-(\gamma)}u^q>\frac{1}{|R_{(x,t)}^-(\gamma)|
+\epsilon}\int_{R_{(x,t)}^-(\gamma)}u^q>[u,v]
_{TA_{1,q}^+(\gamma)}[v(x,t)]^q.
\end{align}
Since $\text{pr}_t(R_{(x,t)}^+(\gamma))$ is an open interval
and $(x,t)\in R_{(x,t)}^+(\gamma)$, it follows that there
exists $\widetilde{R}_{(x,t)}:=Q(x_0,L_0)\times
(t_0-L_0^p,t_0+L_0^p)\in\mathcal{R}_p^{n+1}$ with
$(x_0,t_0)\in\mathbb{R}^{n+1}$ and $L_0\in(0,\infty)$
such that the following statements hold:
\begin{enumerate}
\item[(i)] $(x,t)\in \widetilde{R}_{(x,t)}^+(\gamma)$;

\item[(ii)] $R_{(x,t)}^-(\gamma)\subset
\widetilde{R}_{(x,t)}^-(\gamma)$ and
$|\widetilde{R}_{(x,t)}^-(\gamma)\setminus
R_{(x,t)}^-(\gamma)|<\epsilon$;

\item[(iii)] all the vertices of $Q(x_0,L_0)$ belong
to $\mathbb{Q}^n$ and $t_0-L_0^p\in\mathbb{Q}$.
\end{enumerate}
Combining \eqref{20240723.2145}, (ii), and
$(u,v)\in TA_{1,q}^+(\gamma)$, we conclude that
\begin{align*}
[u,v]_{TA_{1,q}^+(\gamma)}
[v(x,t)]^q&<\frac{1}{|R_{(x,t)}
^-(\gamma)|+\epsilon}\int_{R_{(x,t)}^-(\gamma)}u^q\\
&<\fint_{\widetilde{R}_{(x,t)}^-(\gamma)}u^q\leq[u,v]
_{TA_{1,q}^+(\gamma)}\left[\mathop\mathrm{ess\,inf}_{(y,s)
\in\widetilde{R}_{(x,t)}^+(\gamma)}v(y,s)\right]^q,
\end{align*}
which further implies that
\begin{align}\label{20240723.2159}
v(x,t)<\mathop\mathrm{ess\,inf}_{(y,s)
\in\widetilde{R}_{(x,t)}^+(\gamma)}v(y,s).
\end{align}

Let $\{R_k\}_{k\in\mathbb{N}}$ be the sequence of all
$R:=Q(z,L)\times(r-L^p,r+L^p)\in\mathcal{R}_p^{n+1}$
with $(z,r)\in\mathbb{R}^{n+1}$ and $L\in(0,\infty)$
such that all the vertices of $Q(z,L)$ belong to
$\mathbb{Q}^n$ and $r-L^p\in\mathbb{Q}$ and, for any
$k\in\mathbb{N}$, let
\begin{align*}
\mathcal{N}_k:=\left\{(x,t)\in R_k^+(\gamma):\
v(x,t)<\mathop\mathrm{ess\,inf}_{(y,s)
\in R_k^+(\gamma)}v(y,s)\right\}.
\end{align*}
Then $|\mathcal{N}_k|=0$ for any $k\in\mathbb{N}$.
Moreover, from (i), (iii), and \eqref{20240723.2159},
it follows that
$\mathcal{N}\subset\bigcup_{k\in\mathbb{N}}
\mathcal{N}_k$,
which further implies that
$|\mathcal{N}|\leq\sum_{k\in\mathbb{N}}|\mathcal{N}_k|=0$.
This finishes the
proof of the necessity and hence Proposition
\ref{M- and TA1+}.
\end{proof}

Throughout this article,
we \emph{always omit} the variables $(x,t)$ in the notation if
there is no ambiguity and, for instance, for any
$A\subset\mathbb{R}^{n+1}$,
any function $f$ on $\mathbb{R}^{n+1}$,
and any $\lambda\in\mathbb{R}$, we simply write
\begin{align*}
A\cap\{f>\lambda\}:=\left\{(x,t)\in
A:\ f(x,t)>\lambda\right\}.
\end{align*}
The following proposition indicates that the parabolic
Muckenhoupt two-weight class is closed under taking
the maximum and the minimum.

\begin{proposition}\label{max and min}
Let $\gamma\in[0,1)$ and $1\leq r\leq q<\infty$.
Then, for any
$(u,v),(\widetilde{u},\widetilde{v})\in TA_{r,q}^+(\gamma)$,
\begin{enumerate}
\item[\rm(i)] $(\max\{u,\,\widetilde{u}\},\max\{v,\,
\widetilde{v}\})\in TA_{r,q}^+(\gamma)$ and
\begin{align*}
\left[\max\{u,\,\widetilde{u}\},
\max\{v,\,\widetilde{v}\}\right]_{TA_{r,q}^+(\gamma)}
\leq[u,v]_{TA_{r,q}^+(\gamma)}
+[\widetilde{u},\widetilde{v}]_{TA_{r,q}^+(\gamma)}.
\end{align*}

\item[\rm(ii)] $(\min\{u,\,\widetilde{u}\},
\min\{v,\,\widetilde{v}\})\in TA_{r,q}^+(\gamma)$.
Moreover, if $r=1$, then
\begin{align*}
\left[\min\{u,\,\widetilde{u}\},
\min\{v,\,\widetilde{v}\}\right]
_{TA_{1,q}^+(\gamma)}
\leq\max\left\{[u,v]
_{TA_{1,q}^+(\gamma)},\,[\widetilde{u},\widetilde{v}]
_{TA_{1,q}^+(\gamma)}\right\}
\end{align*}
and, if $r\in(1,\infty)$, then
$[\min\{u,\,\widetilde{u}\},
\min\{v,\,\widetilde{v}\}]_{TA_{r,q}^
+(\gamma)}
\leq[u,v]_{TA_{r,q}^+(\gamma)}
+[\widetilde{u},\widetilde{v}]_{TA_{r,q}^+(\gamma)}$.
\end{enumerate}
\end{proposition}

\begin{proof}
Let $(u,v),(\widetilde{u},\widetilde{v})\in
TA_{r,q}^+(\gamma)$. We first show (i).
Let $W_1:=\max\{u,\,\widetilde{u}\}$ and
$W_2:=\max\{v,\,\widetilde{v}\}$.
We consider the following two cases on $r$.

\emph{Case 1)} $r=1$. In this case,
by Definition \ref{parabolic Muckenhoupt class}(ii),
we find that, for any given $R\in\mathcal{R}_p^{n+1}$,
\begin{align*}
\fint_{R^-(\gamma)}W_1^q&\leq\fint_{R^-(\gamma)}u^q+
\fint_{R^-(\gamma)}\widetilde{u}^q\\
&\leq[u,v]_{TA_{1,q}^+(\gamma)}\left[\mathop\mathrm{ess\,inf}
_{(x,t)\in R^+(\gamma)}v(x,t)\right]^q
+[\widetilde{u},\widetilde{v}]
_{TA_{1,q}^+(\gamma)}\left[\mathop\mathrm{ess\,inf}
_{(x,t)\in R^+(\gamma)}\widetilde{v}(x,t)\right]^q\\
&\leq\left\{[u,v]_{TA_{1,q}^+(\gamma)}
+[\widetilde{u},
\widetilde{v}]_{TA_{1,q}(\gamma)}\right\}
\left[\mathop\mathrm{ess\,inf}_{(x,t)\in
R^+(\gamma)}W_2(x,t)\right]^q.
\end{align*}
Taking the supremum over all
$R\in\mathcal{R}_p^{n+1}$, we conclude that $(W_1,W_2)\in
TA_{1,q}^+(\gamma)$ and
\begin{align*}
[W_1,W_2]_{TA_{1,q}^+(\gamma)}\leq[u,v]
_{TA_{1,q}^+(\gamma)}+[\widetilde{u},\widetilde{v}]
_{TA_{1,q}^+(\gamma)}.
\end{align*}

\emph{Case 2)} $r\in(1,\infty)$. In this case,
from  Definition \ref{parabolic Muckenhoupt class}(i),
we deduce that, for any given
$R\in\mathcal{R}_p^{n+1}$,
\begin{align*}
\fint_{R^-(\gamma)}W_1^q\left[\fint_{R^+(\gamma)}
W_2^{-r'}\right]^{\frac{q}{r'}}
&\leq\left[\fint_{R^-(\gamma)}u^q
+\fint_{R^-(\gamma)}\widetilde{u}^q\right]
\left[\fint_{R^+(\gamma)}W_2^{-r'}\right]^{\frac{q}{r'}}\\
&\leq\fint_{R^-(\gamma)}u^q\left[\fint_{R^+(\gamma)}v^{-r'}
\right]^{\frac{q}{r'}}+\fint_{R^-(\gamma)}\widetilde{u}^q
\left[\fint_{R^+(\gamma)}\widetilde{v}^{-r'}\right]^{\frac{q}{r'}}\\
&\leq[u,v]_{TA_{r,q}^+(\gamma)}
+[\widetilde{u},\widetilde{v}]_{TA_{r,q}^+(\gamma)}.
\end{align*}
Taking the supremum over all
$R\in\mathcal{R}_p^{n+1}$, we obtain $(W_1,W_2)\in
TA_{r,q}^+(\gamma)$ and
\begin{align*}
[W_1,W_2]_{TA_{r,q}^+(\gamma)}\leq[u,v]
_{TA_{r,q}^+(\gamma)}+[\widetilde{u},\widetilde{v}]
_{TA_{r,q}^+(\gamma)}.
\end{align*}
Combining the above two cases, we then finish
the proof of (i).

Next, we prove (ii).
Let $w_1:=\min\{u,\,\widetilde{u}\}$ and
$w_2:=\min\{v,\,\widetilde{v}\}$.
Similarly to the proof of (i),
we consider the following two cases on $r$.

\emph{Case 1)} $r=1$. In this case, using Definition
\ref{parabolic Muckenhoupt class}(ii), we conclude
that, for any $R\in\mathcal{R}_p^{n+1}$,
\begin{align*}
\fint_{R^-(\gamma)}w_1^q
&\leq\min\left\{\fint_{R^-(\gamma)}u^q,
\,\fint_{R^-(\gamma)}\widetilde{u}^q\right\}\\
&\leq\min\left\{[u,v]_{TA_{1,q}^+(\gamma)}
\left[\mathop\mathrm{ess\,inf}_{(x,t)\in
R^+(\gamma)}v(x,t)\right]^q,\,
[\widetilde{u},\widetilde{v}]_{TA_{1,q}^+(\gamma)}
\left[\mathop\mathrm{ess\,inf}_{(x,t)\in
R^+(\gamma)}\widetilde{v}(x,t)\right]^q\right\}\\
&\leq\max\left\{[u,v]_{TA_{1,q}^+(\gamma)},\,
[\widetilde{u},\widetilde{v}]_{TA_{1,q}^+(\gamma)}\right\}
\left[\mathop\mathrm{ess\,inf}_{(x,t)\in
R^+(\gamma)}w_2(x,t)\right]^q.
\end{align*}
Taking the supremum over all $R\in\mathcal{R}_p^{n+1}$, we
find that $(w_1,w_2)\in TA_{1,q}^+(\gamma)$ and
\begin{align*}
[w_1,w_2]_{TA_{1,q}^+(\gamma)}\leq\max\left\{[u,v]
_{TA_{1,q}^+(\gamma)},\,[\widetilde{u},\widetilde{v}]
_{TA_{1,q}^+(\gamma)}\right\}.
\end{align*}

\emph{Case 2)} $r\in(1,\infty)$. In this case, from
Definition \ref{parabolic Muckenhoupt class}(i),
we infer that, for any $R\in\mathcal{R}_p^{n+1}$,
\begin{align*}
\fint_{R^-(\gamma)}w_1^q\left[\fint_{R^+(\gamma)}w_2^{-r'}
\right]^{\frac{q}{r'}}
&\leq\fint_{R^-(\gamma)}w_1^q\left[\frac{1}{|R^+(\gamma)|}
\int_{R^+(\gamma)\cap\{v>\widetilde{v}\}}\widetilde{v}
^{-r'}\right]^{\frac{q}{r'}}\\
&\quad+\fint_{R^-(\gamma)}w_1^q\left[\frac{1}{|R^+(\gamma)|}
\int_{R^+(\gamma)\cap\{v\leq\widetilde{v}\}}v
^{-r'}\right]^{\frac{q}{r'}}\\
&\leq\fint_{R^-(\gamma)}\widetilde{u}^q
\left[\fint_{R^+(\gamma)}\widetilde{v}^{-r'}\right]
^{\frac{q}{r'}}+\fint_{R^-(\gamma)}u^q
\left[\fint_{R^+(\gamma)}v^{-r'}\right]^{\frac{q}{r'}}\\
&\leq[\widetilde{u},\widetilde{v}]_{TA_{r,q}^+(\gamma)}
+[u,v]_{TA_{r,q}^+(\gamma)}.
\end{align*}
Taking the supremum over all
$R\in\mathcal{R}_p^{n+1}$, we obtain $(w_1,w_2)\in
TA_{r,q}^+(\gamma)$ and
\begin{align*}
[w_1,w_2]_{TA_{r,q}^+(\gamma)}\leq[u,v]
_{TA_{r,q}^+(\gamma)}+[\widetilde{u},\widetilde{v}]
_{TA_{r,q}^+(\gamma)}.
\end{align*}
Combining the above two cases  then completes the proof of (ii) and hence Proposition
\ref{max and min}.
\end{proof}

The following duality property
follows directly from Definition~\ref{parabolic Muckenhoupt class}(i);
we omit the details.

\begin{proposition}\label{duality property}
Let $\gamma\in[0,1)$ and $1<r\leq q<\infty$. Assume that $(u,v)$ is a
pair of positive functions on $\mathbb{R}^{n+1}$.
Then $(u,v)\in TA_{r,q}^+(\gamma)$ if and only if
$(v^{-1},u^{-1})\in TA_{q',r'}^-(\gamma)$.
\end{proposition}

At the end of this section, we give a
characterization of the diagonal parabolic Muckenhoupt
two-weight class via a forward in time doubling condition.
In what follows, for any
nonnegative function $f\in
L_{\mathrm{loc}}^1(\mathbb{R}^{n+1})$ and any
measurable set $E\subset\mathbb{R}^{n+1}$, we denote
$\int_Ef$ by $f(E)$.

\begin{proposition}\label{quantitative measure condition}
Let $\gamma\in[0,1)$ and $(u,v)$ be a pair of positive
functions on $\mathbb{R}^{n+1}$. Then the following statements
are equivalent.
\begin{enumerate}
\item[\rm(i)] There exists $q\in(1,\infty)$ such that
$(u,v)\in TA_{q,q}^+(\gamma)$.

\item[\rm(ii)] There exist $C\in(0,\infty)$, $\delta\in(0,1)$,
and $r\in(\frac{1}{\delta},\infty)$ such that, for any
$R\in\mathcal{R}_p^{n+1}$
and any measurable set $E\subset R^+(\gamma)$,
\begin{align}\label{20240724.2128}
\frac{|E|}{|R^+(\gamma)|}\leq
C\left[\frac{(v^r)(E)}{(u^r)(R^-(\gamma))}\right]^{\delta}.
\end{align}
\end{enumerate}
\end{proposition}

\begin{proof}
We first show $(\mathrm{i})\Longrightarrow(\mathrm{ii})$.
Assume that $(u,v)\in TA_{q,q}^+(\gamma)$ for
some $q\in(1,\infty)$. Then,
from the H\"older
inequality, we deduce that, for any
$R\in\mathcal{R}_p^{n+1}$
and any measurable set $E\subset R^+(\gamma)$,
\begin{align*}
\frac{|E|}{|R^+(\gamma)|}&=\fint_{R^+(\gamma)}
\boldsymbol{1}_{E}=\fint_{R^+(\gamma)}v^{-1}v\boldsymbol{1}_E
\leq\left[\fint_{R^+(\gamma)}v^q\boldsymbol{1}_E\right]
^{\frac{1}{q}}\left[\fint_{R^+(\gamma)}v^{-q'}\right]
^{\frac{1}{q'}}\\
&\leq\left[\fint_{R^+(\gamma)}v^q\boldsymbol{1}_E\right]
^{\frac{1}{q}}[u,v]_{TA_{q,q}^+(\gamma)}^{\frac{1}{q}}
\left[\fint_{R^-(\gamma)}u^q\right]^{-\frac{1}{q}}
=[u,v]_{TA_{q,q}^+(\gamma)}^{\frac{1}{q}}
\left[\frac{(v^q)(E)}{(u^q)(R^-(\gamma))}\right]^{\frac{1}{q}},
\end{align*}
which further implies that \eqref{20240724.2128}
with $r:=q$, $\delta\in(0,\frac{1}{q}]$,
and $C:=[u,v]_{TA_q^+(\gamma)}^{\delta}$ holds.
This finishes the proof of
$(\mathrm{i})\Longrightarrow(\mathrm{ii})$.

Now, we prove $(\mathrm{ii})\Longrightarrow(\mathrm{i})$.
Assume that (ii) holds. According to \eqref{20240724.2128},
we conclude that, for any $R\in\mathcal{R}_p^{n+1}$
and any measurable set $E\subset R^+(\gamma)$,
\begin{align}\label{20240724.2156}
C^{-\frac{1}{\delta}}\left(u^r\right)(R^-(\gamma))
\left[\frac{|E|}{|R^+(\gamma)|}\right]
^{\frac{1}{\delta}}\leq\left(v^r\right)(E).
\end{align}
Fix $R\in\mathcal{R}_p^{n+1}$ and, for any given
$\lambda\in(0,\infty)$, let
$E_\lambda:=R^+(\gamma)\cap\{v^{-r}>\lambda\}$.
Then
$(v^r)(E_\lambda)\leq|E_\lambda|/\lambda$.
Combining this and \eqref{20240724.2156}, we find that,
for any $\lambda\in(0,\infty)$,
\begin{align*}
C^{-\frac{1}{\delta}}\left(u^r\right)(R^-(\gamma))
\left[\frac{|E_\lambda|}{|R^+(\gamma)|}
\right]^{\frac{1}{\delta}}\leq\left(v^r\right)(E_\lambda)
\leq\frac{|E_\lambda|}{\lambda},
\end{align*}
which further implies that
\begin{align}\label{20240724.2209}
|E_\lambda|\leq\frac{C^{\frac{1}{1-\delta}}|R^-(\gamma)|
^{\frac{1}{1-\delta}}}{\lambda^{\frac{\delta}{1-\delta}}
[(u^r)(R^-(\gamma))]^{\frac{\delta}{1-\delta}}}.
\end{align}
Since $r\in(\frac{1}{\delta},\infty)$, it follows that
$\frac{r'}{r}\in(0,\frac{\delta}{1-\delta})$. From the
Cavalieri principle (see, for instance, \cite[Proposition
1.1.4]{Grafakos first volume}), the obvious fact that
$E_\lambda\subset R^+(\gamma)$ for any $\lambda\in(0,\infty)$,
and \eqref{20240724.2209}, we infer that
\begin{align*}
\int_{R^+(\gamma)}v^{-r'}
&=\frac{r'}{r}\int_0^\infty
\lambda^{\frac{r'}{r}-1}\left|R^+(\gamma)\cap\left\{v^{-r}>
\lambda\right\}\right|\,d\lambda\\
&=\frac{r'}{r}\left[\int_0^{\frac{1}{(u^r)_{R^-(\gamma)}}}+
\int_{\frac{1}{(u^r)_{R^-(\gamma)}}}^\infty\right]
\lambda^{\frac{r'}{r}-1}|E_\lambda|\,d\lambda\\
&\leq\frac{r'}{r}\left|R^+(\gamma)\right|
\int_0^{\frac{1}{(u^r)_{R^-(\gamma)}}}
\lambda^{\frac{r'}{r}-1}\,d\lambda
+\frac{r'}{r}
\frac{C^{\frac{1}{1-\delta}}|R^-(\gamma)|
^{\frac{1}{1-\delta}}}{
[(u^r)(R^-(\gamma))]^{\frac{\delta}{1-\delta}}}
\int_{\frac{1}{(u^r)_{R^-(\gamma)}}}^\infty
\lambda^{\frac{r'}{r}-\frac{1}{1-\delta}}
\,d\lambda\\
&\lesssim\left|R^+(\gamma)\right|
\left[\fint_{R^-(\gamma)}u^r\right]^{-\frac{r'}{r}},
\end{align*}
where the implicit positive constant depends only on $C$,
$r$, and $\delta$.
Taking the supremum over all $R\in\mathcal{R}_p^{n+1}$, we
conclude that $(u,v)\in TA_{r,r}^+(\gamma)$,
which completes the proof of the sufficiency and hence
Proposition \ref{quantitative measure condition}.
\end{proof}

\section{Independence of Choices of Time Lag \\ and
Self-improving Property of $TA_{r,q}^+(\gamma)$}
\label{section3}

In this section, under an extra mild assumption (which is
not necessary for one-weight case), we show that the
parabolic Muckenhoupt two-weight class is independent of
the time lag and the distance between the upper and the lower
parts of parabolic rectangles. As an application, we
prove that the parabolic Muckenhoupt two-weight class has
the self-improving property. Recall that, for any
$\gamma\in[0,1)$, $A_\infty^+(\gamma):=
\bigcup_{q\in(1,\infty)}A_q^+(\gamma)$
and, for any $A\subset\mathbb{R}^{n+1}$ and $a\in\mathbb{R}$,
$A-(\mathbf{0},a):=\{(x,t-a):\ (x,t)\in A\}$.

\begin{theorem}\label{independence of time lag 2}
Let $0<\gamma<\alpha<1$, $\tau\in[1,\infty)$, and $(u,v)$
be a pair of nonnegative functions on $\mathbb{R}^{n+1}$.
Assume that $u\in A_\infty^+(\gamma)$. Then the following
statements hold.
\begin{enumerate}
\item[\rm(i)] If $1<r\leq q<\infty$, then $(u,v)\in
TA_{r,q}^+(\gamma)$ if and only if there exists a positive
constant $C$ such that, for any $R\in\mathcal{R}_p^{n+1}$,
\begin{align}\label{20240815.2233}
\fint_{R^+(\alpha)-(\mathbf{0},\tau(1+\alpha)[l(R)]^p)}u^q
\left[\fint_{R^+(\alpha)}v^
{-r'}\right]^{\frac{q}{r'}}\leq C.
\end{align}

\item[\rm(ii)] If $q\in[1,\infty)$, then $(u,v)\in
TA_{1,q}^+(\gamma)$ if and only if there exists a positive
constant $C$ such that, for any $R\in\mathcal{R}_p^{n+1}$,
\begin{align}\label{20241006.0907}
\fint_{R^+(\alpha)-(\mathbf{0},\tau(1+\alpha)[l(R)]^p)}u^q
\left[\mathop\mathrm{ess\,inf}_{(x,t)\in
R^+(\alpha)}v(x,t)\right]^{-q}\leq C.
\end{align}
\end{enumerate}
\end{theorem}

\begin{proof}
We first show the necessity of (i) and (ii). Assume that
$(u,v)\in TA_{r,q}^+(\gamma)$. Fix
$R:=R(x,t,L)\in\mathcal{R}_p^{n+1}$ with
$(x,t)\in\mathbb{R}^{n+1}$ and $L\in(0,\infty)$. Let
\begin{align*}
P:=R\left(x,t-\frac{(\tau-1)(1+\alpha)L^p}{2},
\left[1+\frac{(\tau-1)(1+\alpha)}{2}\right]^\frac{1}{p}L\right).
\end{align*}
Then
\begin{align}\label{2137}
R^+(\alpha)\subset P^+(\widetilde{\alpha})
\ \ \mbox{and}\ \
R^+(\alpha)-(\mathbf{0},\tau(1+\alpha)L^p)\subset
P^-(\widetilde{\alpha}),
\end{align}
where
$\widetilde{\alpha}:=\frac{2\alpha+(\tau-1)(1+\alpha)}
{2+(\tau-1)(1+\alpha)}\in(\gamma,1)$.
Moreover, we have
\begin{align}\label{2135}
\left|P^\pm(\widetilde{\alpha})\right|=2^n
\left[1+\frac{(\tau-1)(1+\alpha)}{2}\right]^\frac{n}{p}
\left|R^\pm(\alpha)\right|.
\end{align}
From Remark \ref{remark on parabolic weight}(ii), we deduce
that $(u,v)\in TA_{r,q}^+(\widetilde{\alpha})$,
which, combined with \eqref{2137} and \eqref{2135},
further implies that
\begin{align}\label{20240816.1202}
&\fint_{R^+(\alpha)-(\mathbf{0},\tau(1+\alpha)L^p)}u^q
\left[\fint_{R^+(\alpha)}v^
{-r'}\right]^{\frac{q}{r'}}\\
&\quad\leq\left\{2^n
\left[1+\frac{(\tau-1)(1+\alpha)}{2}
\right]^\frac{n}{p}\right\}^{1+\frac{q}{r'}}
\fint_{P^-(\widetilde{\alpha})}u^q
\left[\fint_{P^+(\widetilde{\alpha})}v^
{-r'}\right]^{\frac{q}{r'}}\notag\\
&\quad\leq\left\{2^n
\left[1+\frac{(\tau-1)(1+\alpha)}{2}
\right]^\frac{n}{p}\right\}^{1+\frac{q}{r'}}
[u,v]_{TA_{r,q}^+(\widetilde{\alpha})}.\notag
\end{align}
This finishes the proof of the necessity of (i). Letting
$r\to1^+$ in the above argument, we find that the necessity
of (ii) holds. Here and thereafter, $r\to1^+$ means that $r\in(1,\infty)$
and $r\to1$.

Then we prove the sufficiency of (i) by the chaining
argument. Suppose that \eqref{20240815.2233} holds. Fix
$R\in\mathcal{R}_p^{n+1}$. Without loss of generality, we
may suppose that $R$ is centered at the origin
$(\mathbf{0},0)$. Let $m\in\mathbb{N}$ be such that there
exists $\epsilon\in[0,1)$ satisfying
\begin{align}\label{20240731.1610}
m=\log_2\left[\frac{\tau(1+\alpha)}{1-\alpha}\right]+
\frac{1}{p-1}\left\{1+\log_2\left[\frac{\tau(1+\alpha)}
{\gamma}\right]\right\}+2+\epsilon.
\end{align}
Partition each spatial edge of $R^+(\gamma)$ into $2^m$
equally long intervals and divide the temporal edge of
$R^+(\gamma)$ into $J:=\lceil(1-\gamma)2^{pm}/(1-\alpha)\rceil$
equally long intervals. Then we obtain $2^{nm}J$ subrectangles
of $R^+(\gamma)$ and denote them by
$\{V_{i,j}^+\}_{i\in\mathbb{N}\cap[1,2^{nm}],
j\in\mathbb{N}\cap[1,J]}$. Fix $i\in\mathbb{N}\cap[1,2^{nm}]$
and $j\in\mathbb{N}\cap[1,J]$. Notice that there exists
$R_{i,j}\in\mathcal{R}_p^{n+1}$ such that the tops of
both $R_{i,j}$ and $V_{i,j}^+$ coincide, $V_{i,j}^+\subset
R_{i,j}^+(\alpha)$, and
\begin{align}\label{20240915.0844}
\frac{|V_{i,j}^+|}{|R_{i,j}^+(\alpha)|}=
\frac{(1-\gamma)2^{pm}}{(1-\alpha)J}.
\end{align}
Divide $R^-(\gamma)$ in the same manner as we partition
$R^+(\gamma)$. Then we obtain $2^{nm}J$ subrectangles of
$R^-(\gamma)$ and denote them by $\{U_{k,\iota}^-\}
_{k\in\mathbb{N}\cap[1,2^{nm}],\iota\in\mathbb{N}\cap[1,J]}$.
Fix $k\in\mathbb{N}\cap[1,2^{nm}]$ and
$\iota\in\mathbb{N}\cap[1,J]$. Observe that there exists
$\widetilde{R}_{k,\iota}\in\mathcal{R}_p^{n+1}$ such that
the bottoms of both
$\widetilde{R}_{k,\iota}^+(\alpha)-(\mathbf{0},
\tau(1+\alpha)[l(\widetilde{R}_{k,\iota})]^p)$ and
$U_{k,\iota}^-$ coincide, $U_{k,\iota}^-\subset
\widetilde{R}_{k,\iota}^+(\alpha)-(\mathbf{0},
\tau(1+\alpha)[l(\widetilde{R}_{k,\iota})]^p)$, and
\begin{align}\label{20240915.0853}
\frac{|U_{k,\iota}^-|}{|\widetilde{R}_{k,\iota}^+(\alpha)
-(\mathbf{0},\tau(1+\alpha)[l(\widetilde{R}_{k,\iota})]^p)|}
=\frac{(1-\gamma)2^{pm}}{(1-\alpha)J}.
\end{align}

We claim that there exists
$\mathfrak{R}\in\mathcal{R}_p^{n+1}$ satisfying that, for any
$i,k\in\mathbb{N}\cap[1,2^{nm}]$ and
$j,\iota\in\mathbb{N}\cap[1,J]$, there exist
$N_j,\widetilde{N}_\iota\in\mathbb{N}$ and chains
$\{P_d^{i,j}\}_{d=0}^{N_j}$ and
$\{\widetilde{P}_d^{k,\iota}\}_{d=0}^{\widetilde{N}_\iota}$
consisting of congruent parabolic rectangles such that the
following statements hold.
\begin{enumerate}

\item[(i)] $\mathfrak{R}\subset R^+(0)$ and
$\mathfrak{R}^+(\alpha)-(\mathbf{0},\tau(1+\alpha)
[l(\mathfrak{R})]^p)\subset R^+(0)$.

\item[(ii)] For any $i\in\mathbb{N}\cap[1,2^{nm}]$ and
$j\in\mathbb{N}\cap[1,J]$, $P_0^{i,j}=R_{i,j}$ and
$P_{N_j}^{i,j}=\mathfrak{R}$. For any
$k\in\mathbb{N}\cap[1,2^{nm}]$ and
$\iota\in\mathbb{N}\cap[1,J]$,
$\widetilde{P}_0^{k,\iota}=\widetilde{R}_{k,\iota}$ and
$\widetilde{P}_{\widetilde{N}_\iota}^{k,\iota}=\mathfrak{R}$.

\item[(iii)] For any $j\in\mathbb{N}\cap[1,J]$,
\begin{align*}
N_j\leq2^{\frac{p}{p-1}+3p+1}\left[\frac{\tau(1+\alpha)}
{1-\alpha}\right]^p\left[\frac{\tau(1+\alpha)}{\gamma}\right]
^{\frac{p}{p-1}}=:C_1.
\end{align*}
For any $\iota\in\mathbb{N}\cap[1,J]$,
$\widetilde{N}_\iota\leq 2C_1$.

\item[(iv)] For any $i\in\mathbb{N}\cap[1,2^{nm}]$,
$j\in\mathbb{N}\cap[1,J]$, and
$d\in\mathbb{N}\cap[1,N_j]$,
\begin{align*}
\frac{|(P_d^{i,j})^+(\alpha)\cap(S_{d-1}^{i,j})^-(\alpha)|}
{|(P_d^{i,j})^+(\alpha)|}\in
\left[\frac{1}{2^{n+1}},1\right],
\end{align*}
where $(S_h^{i,j})^-(\alpha):=
(P_h^{i,j})^+(\alpha)-(\mathbf{0},
\tau(1+\alpha)[l(P_h^{i,j})]^p)$ for any $h\in\mathbb{Z}_+
\cap[0,N_j]$. Moreover,
for any $k\in\mathbb{N}\cap[1,2^{nm}]$,
$\iota\in\mathbb{N}\cap[1,J]$, and
$d\in\mathbb{N}\cap[1,\widetilde{N}_\iota]$,
\begin{align*}
\frac{|(\widetilde{S}_d^{k,\iota})
^-(\alpha)\cap(\widetilde{P}_{d-1}^{k,\iota})^+(\alpha)|}
{|(\widetilde{S}_d^{k,\iota})^-(\alpha)|}\in
\left[\frac{1}{2^{n+1}},1\right],
\end{align*}
where $(\widetilde{S}_h^{k,\iota})^-(\alpha):=
(\widetilde{P}_h^{k,\iota})^+(\alpha)-(\mathbf{0},
\tau(1+\alpha)[l(\widetilde{P}_h^{k,\iota})]^p)$ for any
$h\in\mathbb{Z}_+\cap[0,\widetilde{N}_\iota]$.
\end{enumerate}
Indeed, fix $i\in\mathbb{N}\cap[1,2^{nm}]$ and
$j\in\mathbb{N}\cap[1,J]$. We will specify $\mathfrak{R}$ and
$N_j$ and construct the chain $\{P_d^{i,j}\}_{d=0}^{N_j}$ in
the following two steps.

\emph{Step 1.} In this step, we construct the chain
corresponding to spatial variables. Assume that
$R_{i,j}=Q(x_i,l)\times(t_j-2l^p,t_j)$ with
$(x_i,t_j)\in\mathbb{R}^{n+1}$ and $l:=l(R)/2^m$. Let
$Q_0^i:=Q(x_i,l)$. For any $d\in\mathbb{N}\cap[1,M_i]$ with
$M_i\in\mathbb{N}$ determined later, let
\begin{align}\label{20240915.1007}
Q_d^i:=Q_{d-1}^i-\frac{x_i}
{\vert x_i\vert}\frac{\theta_il}{2},
\end{align}
where $\theta_i\in[1,\sqrt{n}]$, depending only on the angle
between $x_i$ and the spatial axes, is such that the center of
$Q_d^i$ belongs to the boundary of $Q_{d-1}^i$. Notice that
there exists $b_i\in\mathbb{N}\cap[1,2^m-1]$, depending only on
$\vert x_i\vert$, such that
\begin{align}\label{20240915.1008}
\vert x_{i}\vert=\frac{\theta_i}{2}[l(R)-b_il].
\end{align}
From this and \eqref{20240915.1007}, it follows that, to
ensure that $Q_{M_{i}}^i=Q(\mathbf{0},l)$,
we need to choose
\begin{align}\label{20240915.1009}
M_{i}:=\frac{2\vert x_i\vert}{\theta_il}
=\frac{l(R)}{l}-b_i=2^m-b_i.
\end{align}
Observe that, for any $d\in\mathbb{N}\cap[1,M_i]$,
\begin{align}\label{20240915.1010}
\frac{1}{2^n}\leq\frac{\vert Q_d^i\cap Q_{d-1}^i\vert}
{\vert Q_d^i\vert}\leq\frac{1}{2}.
\end{align}
Then $\{Q_d^i\}_{d=0}^{M_{i}}$ forms a chain in
$\mathbb{R}^n$ starting from $Q(x_i,l)$ and
ending with $Q(\mathbf{0},l)$.

\emph{Step 2.} In this step, we specify $\mathfrak{R}$ and
$N_j$ and construct the chain $\{P_d^{i,j}\}_{d=0}^{N_j}$
based on the chains in Step 1. For any
$d\in\mathbb{Z}_+\cap[0,M_i]$, let
\begin{align*}
P_d^{i,j}:=Q_d^i\times\left(t_j-d\tau(1+\alpha)l^p-2l^p,
t_j-d\tau(1+\alpha)l^p\right).
\end{align*}
Then $P_0^{i,j}=R_{i,j}$ and $\text{pr}_x(P_{M_i}^{i,j})
=Q(\mathbf{0},l)$. To determine $\mathfrak{R}$, we first
assume that $j=1$ and $i\in\mathbb{N}\cap[1,2^{nm}]$
such that $Q(x_i,l)$ intersects with the boundary of
$\text{pr}_x(R)=Q(\mathbf{0},l(R))$. From
\eqref{20240915.1008} and \eqref{20240915.1009}, we infer that
$b_i=1$ and hence $M_i=2^m-1$ in this case. Let $N:=2^m-1$ and
$\mathfrak{R}:=P_N^{i,1}$. We show that (i) holds. Indeed, on
the one hand, notice that
\begin{align}\label{QN subset Q(0,L) and t1<Lp}
Q_N=Q(\mathbf{0},l)\subset Q(\mathbf{0},l(R))\ \ \mbox{and}\ \
t_1-N\tau(1+\alpha)l^p<t_1<[l(R)]^p.
\end{align}
On the other hand, by \eqref{20240731.1610}, we obtain
\begin{align*}
m\geq\frac{1}{p-1}\left\{1+\log_2\left[\frac{\tau(1+\alpha)}
{\gamma}\right]\right\},
\end{align*}
which further implies that
$2^{(p-1)m}\geq\frac{2\tau(1+\alpha)}{\gamma}$ and hence
$\gamma-\frac{2\tau(1+\alpha)}{2^{(p-1)m}}\geq0$. From this,
the definitions of both $J$ and $l$, and the fact that $\lceil
s\rceil\leq 2s$ for any $s\in[1,\infty)$, we deduce that
\begin{align*}
&t_1-(N+1)\tau(1+\alpha)l^p-(1-\alpha)l^p\\
&\quad=\left[\gamma+\frac{1-\gamma}{\lceil(1-\gamma)2^{pm}
/(1-\alpha)\rceil}-\frac{(N+1)\tau(1+\alpha)}{2^{pm}}
-\frac{(1-\alpha)}{2^{pm}}\right][l(R)]^p\\
&\quad\geq\left[\gamma+\frac{1-\gamma}{(1-\gamma)2^{pm+1}
/(1-\alpha)}-\frac{(N+1)\tau(1+\alpha)}{2^{pm}}
-\frac{(1-\alpha)}{2^{pm}}\right][l(R)]^p\\
&\quad=\left[\gamma-\frac{1-\alpha}
{2^{pm+1}}-\frac{\tau(1+\alpha)}{2^{(p-1)m}}\right][l(R)]^p
\geq\left[\gamma-\frac{2\tau(1+\alpha)}{2^{(p-1)m}}\right]
[l(R)]^p\geq0.
\end{align*}
This, together with both the fact that the bottom of
$\mathfrak{R}^+(\alpha)-(\mathbf{0},\tau(1+\alpha)
[l(\mathfrak{R})]^p)$ is
\begin{align*}
Q(\mathbf{0},l)\times
\left\{t_1-(N+1)\tau(1+\alpha)l^p-(1-\alpha)l^p\right\}
\end{align*}
and \eqref{QN subset Q(0,L) and t1<Lp}, further implies that
(i) holds.

Then we suppose that $j=1$ and $i\in\mathbb{N}\cap[1,2^{nm}]$
such that $Q(x_i,l)$ does not intersect with the boundary of
$\text{pr}_x(R)=Q(\mathbf{0},l(R))$. In this case, $b_i\neq1$
and hence $M_{i}\in\mathbb{N}\cap[1,N)$. For any
$d\in\mathbb{N}\cap[M_i,N-1]$, let
$P_{d+1}^{i,j}:=P_d^{i,j}-(\mathbf{0},\tau(1+\alpha)l^{p})$.
Then we are easy to see that $P_{N}^{i,j}=\mathfrak{R}$.

Now, we consider the residual situation
$j\in\mathbb{N}\cap[2,J]$ and
$i\in\mathbb{N}\cap[1,2^{nm}]$.
In this case, we first extend the chain
$\{P_d^{i,j}\}_{d=0}^{M_i}$ to the chain
$\{P_d^{i,j}\}_{d=0}^N$ in the same way as we did in the
above case. Then we can easily verify that
$\mathrm{pr}_x(P_N^{i,j})=\mathrm{pr}_x(P_N^{i,1})$.
However, in the temporal variable,
the distance between the top of $P_N^{i,j}$ and that of
$P_N^{i,1}$ is $(j-1)(1-\gamma)[l(R)]^p/J$. We can shift every
$P_d^{i,j}$ for $d\in\mathbb{N}\cap[1,2^{m-1}]$ and add
$\widetilde{M}_j$ parabolic rectangles into the chain to
guarantee that the eventual parabolic rectangle is
$\mathfrak{R}$. To be specific, we can choose $\beta_j\in[0,1)$
and $\widetilde{M}_j\in\mathbb{Z}_{+}$ such that
\begin{align}\label{20240915.2115}
\frac{2^{m-1}\beta_j(1-\alpha)[l(R)]^p}{2^{pm}}
+\frac{\widetilde{M}_j\tau(1+\alpha)[l(R)]^{p}}{2^{pm}}
=\frac{(j-1)(1-\gamma)[l(R)]^p}{J}.
\end{align}
Indeed, notice that there exists $\eta\in[1,2)$ such that
\begin{align*}
J=\left\lceil\frac{(1-\gamma)2^{pm}}
{1-\alpha}\right\rceil
=\frac{\eta(1-\gamma)2^{pm}}{1-\alpha}
\end{align*}
and hence we can rewrite \eqref{20240915.2115} as
\begin{align}\label{20240915.2152}
2^{m-1}\beta_j(1-\alpha)+\widetilde{M}_j\tau(1+\alpha)
=\frac{(j-1)(1-\alpha)}{\eta}.
\end{align}
Choose $\widetilde{M}_j\in\mathbb{Z}_{+}$ such that
\begin{align*}
\widetilde{M}_j\tau(1+\alpha)\leq\frac{(j-1)(1-\alpha)}{\eta}
<(\widetilde{M}_j+1)\tau(1+\alpha).
\end{align*}
That is, let $\widetilde{M}_j\in\mathbb{Z}_{+}$ be such that
\begin{align}\label{20240915.2208}
\frac{(j-1)(1-\alpha)}{2\tau(1+\alpha)}-1&\leq
\frac{(j-1)(1-\alpha)}{\eta\tau(1+\alpha)}-1<\widetilde{M}_j\\
&\leq\frac{(j-1)(1-\alpha)}{\tau(1+\alpha)}
\leq\frac{(j-1)(1-\alpha)}{\tau(1+\alpha)}.\notag
\end{align}
Let $\xi_j:=\frac{(j-1)(1-\alpha)}{\eta}
-\widetilde{M}_j(1+\alpha)$. Using this and the
choice of $\widetilde{M}_j$, we find that $\xi_j
\in[0,\tau(1+\alpha))$. Select $\beta_j\in[0,\infty)$ such that
$\xi_j=2^{m-1}\beta_{j}(1-\alpha)$.
From this and \eqref{20240731.1610}, we infer that
\begin{align*}
\beta_j=2^{\frac{-1}{p-1}-1-\epsilon}\left(\frac{\gamma}
{1+\alpha}\right)^{\frac{1}{p-1}}\frac{\xi_j}{\tau(1+\alpha)}.
\end{align*}
By this and $\xi_j\in[0,\tau(1+\alpha))$, we conclude that
\begin{align*}
0\leq\beta_j<2^{-1}\left(\frac{\gamma}{1+\alpha}\right)
^{\frac{1}{p-1}}<\frac{1}{2},
\end{align*}
which, together with the choices of both $\widetilde{M}_j$ and
$\beta_j$, further implies that \eqref{20240915.2152}
holds. Finally, for any $d\in\mathbb{N}\cap[1,2^{m-1}]$,
we modify the definition of $P_d^{i,j}$ by setting
\begin{align*}
P_d^{i,j}:=Q_d^i\times\left(t_j-d\left[\tau(1+\alpha)+
\beta_j(1-\alpha)\right]l^p-2l^p,
t_j-d\left[\tau(1+\alpha)+\beta_j(1-\alpha)\right]l^p\right).
\end{align*}
If $\widetilde{M}_j\in\mathbb{N}$, then, for any
$d\in\mathbb{N}\cap[N,N+\widetilde{M}_j-1]$, let
$P_{d+1}^{i,j}:=P_d^{i,j}-\left(\mathbf{0},\tau(1+\alpha)
l^p\right)$.
Then $P_{N+\widetilde{M}_j}^{i,j}=\mathfrak{R}$.
For convenience, let $\widetilde{M}_1:=0$.

In conclusion, for any $i\in\mathbb{N}\cap[1,2^{nm}]$
and $j\in\mathbb{N}\cap[1,J]$, we have constructed a chain
$\{P_d^{i,j}\}_{d=0}^{N+\widetilde{M}_j}$ starting from
$R_{i,j}$ and ending with $\mathfrak{R}$. Let $N_j:=N+
\widetilde{M}_j$. Applying the definitions of both $N$ and $J$,
\eqref{20240915.2208}, and \eqref{20240731.1610}, we obtain
\begin{align*}
N_j&=2^m-1+\widetilde{M}_j\leq2^m+
\frac{(j-1)(1-\alpha)}{\tau(1+\alpha)}\\
&\leq2^m+\frac{(J-1)(1-\alpha)}{\tau(1+\alpha)}\leq
2^m+\frac{(1-\gamma)2^{pm}}{1-\alpha}\frac{1-\alpha}
{\tau(1+\alpha)}\\
&\leq2^{pm+1}\leq2^{\frac{p}{p-1}+3p+1}
\left[\frac{\tau(1+\alpha)}
{1-\alpha}\right]^p\left[\frac{\tau(1+\alpha)}{\gamma}\right]
^{\frac{p}{p-1}}.
\end{align*}
From \eqref{20240915.1010}, the fact that
$\beta_j\in[0,\frac{1}{2})$, and the construction of
$\{P_d^{i,j}\}_{d=0}^{N_j}$, we deduce that,
for any $i\in\mathbb{N}\cap[1,2^{nm}]$,
$j\in\mathbb{N}\cap[1,J]$, and
$d\in\mathbb{N}\cap[1,N_j]$,
\begin{align*}
\frac{|(P_d^{i,j})^+(\alpha)\cap(S_{d-1}^{i,j})^-(\alpha)|}
{|(P_d^{i,j})^+(\alpha)|}\in
\left[\frac{1}{2^{n+1}},1\right],
\end{align*}

Similarly, for any $k\in\mathbb{N}\cap[1,2^{nm}]$ and $\iota\in\mathbb{N}
\cap[1,J]$, we can construct a chain
$\{\widetilde{P}_d^{k,\iota}\}_{d=0}^{\widetilde{N_\iota}}$
such that $\widetilde{N}_\iota\in\mathbb{N}\cap[1,2C_1]$,
$\widetilde{P}_0^{k,\iota}=\widetilde{R}_{k,\iota}$,
$\widetilde{P}_{\widetilde{N}_\iota}^{k,\iota}=\mathfrak{R}$,
and
\begin{align*}
\frac{|(\widetilde{S}_d^{k,\iota})
^-(\alpha)\cap(\widetilde{P}_{d-1}^{k,\iota})^+(\alpha)|}
{|(\widetilde{S}_d^{k,\iota})^-(\alpha)|}\in
\left[\frac{1}{2^{n+1}},1\right]
\end{align*}
for any $d\in\mathbb{N}\cap[1,\widetilde{N}_\iota]$. This
finishes the proof of (ii)-(iv) and hence the above claim.

Next, based on (i)-(iv), we can build a chain connecting
$R_{i,j}$ and $\widetilde{R}_{k,\iota}$ for any
$i,k\in\mathbb{N}\cap[1,2^{nm}]$ and
$j,\iota\in\mathbb{N}\cap[1,J]$. More precisely, for any
$i,k\in\mathbb{N}\cap[1,2^{nm}]$,
$j,\iota\in\mathbb{N}\cap[1,J]$, and
$d\in\mathbb{Z}_+\cap[0,\widetilde{N}_\iota+N_j]$, let
\begin{align*}
P_{d}^{i,j,k,\iota}:=
\begin{cases}
\widetilde{P}_{d}^{k,\iota}&\text{if }d\in\mathbb{Z}_+
\cap\left[0,\widetilde{N}_\iota\right],\\
P_{\widetilde{N}_\iota+N_j-d}^{i,j}&\text{if }d\in\mathbb{Z}_+
\cap\left[\widetilde{N}_\iota+1,\widetilde{N}_\iota+N_j\right]
\end{cases}
\end{align*}
and $(S_d^{i,j,k,\iota})^-(\alpha):=
(P_d^{i,j,k,\iota})^+(\alpha)-(\mathbf{0},\tau(1+\alpha)
[l(P_d^{i,j,k,\iota})]^p)$.
Then, for any $i,k\in\mathbb{N}\cap[1,2^{nm}]$ and
$j,\iota\in\mathbb{N}\cap[1,J]$, we have constructed a chain
$\{P_{d}^{i,j,k,\iota}\}_{d=0}^{\widetilde{N}_\iota+N_j}$
consisting of congruent parabolic rectangles. From (i) through
(iv), it follows that the following statements hold.
\begin{enumerate}
\item[(v)] For any $i,k\in\mathbb{N}\cap[1,2^{nm}]$,
$j,\iota\in\mathbb{N}\cap[1,J]$, and
$d\in\mathbb{Z}_+\cap[0,\widetilde{N}_\iota+N_j]$,
$P_d^{i,j,k,\iota}\subset R$.

\item[(vi)] For any $i,k\in\mathbb{N}\cap[1,2^{nm}]$ and
$j,\iota\in\mathbb{N}\cap[1,J]$,
$P_0^{i,j,k,\iota}=\widetilde{R}_{k,\iota}$ and
$P_{\widetilde{N}_\iota+N_j}^{i,j,k,\iota}=R_{i,j}$.

\item[(vii)] For any $i,k\in\mathbb{N}\cap[1,2^{nm}]$,
$j,\iota\in\mathbb{N}\cap[1,J]$, and
$d\in\mathbb{N}\cap[1,\widetilde{N}_\iota+N_j]$,
\begin{align*}
\frac{|(S_d^{i,j,k,\iota})^-(\alpha)
\cap(P_{d-1}^{i,j,k,\iota})
^+(\alpha)|}{|(S_d^{i,j,k,\iota})^-(\alpha)|}
\in\left[\frac{1}{2^{n+1}},1\right].
\end{align*}
\end{enumerate}

Now, we prove that there exists a positive constant $C_2$ such
that, for any given $i,k\in\mathbb{N}\cap[1,2^{nm}]$
and $j,\iota\in\mathbb{N}\cap[1,J]$,
\begin{align}\label{20240731.1506}
\fint_{\widetilde{R}_{k,\iota}^+(\alpha)-(\mathbf{0},
\tau(1+\alpha)[l(\widetilde{R}_{k,\iota})]^p)}u^q\leq C_2
\fint_{R_{i,j}^+(\alpha)-(\mathbf{0},\tau(1+\alpha)
[l(R_{i,j})]^p)}u^q.
\end{align}
Indeed, by \cite[Lemma 7.4]{ks(apde-2016)}, \cite[Theorems
3.1 and 4.1]{km(am-2024)}, and the assumption that $u\in
A_\infty^+(\gamma)$, we find that there
exist $K,\delta\in(0,\infty)$ such that, for any
$R\in\mathcal{R}_p^{n+1}$ and any measurable set $E\subset
R^+(\alpha)$,
\begin{align*}
\left(u^q\right)\left(R^+(\alpha)-
(0,\tau(1+\alpha)l(R)^p)\right)\leq
K\left[\frac{|R^+(\alpha)|}{|E|}\right]^{\delta}
\left(u^q\right)(E).
\end{align*}
From this and (vii),
we infer that, for any $d\in\mathbb{N}
\cap[1,\widetilde{N}_\iota+N_j]$,
\begin{align*}
\fint_{(S_{d-1}^{i,j,k,\iota})^-(\alpha)}u^q&=\frac{1}
{|(S_{d-1}^{i,j,k,\iota})^-(\alpha)|}\left(u^q\right)
\left(\left(S_{d-1}^{i,j,k,\iota}\right)^-(\alpha)\right)\\
&\leq\frac{K}{|(S_{d-1}^{i,j,k,\iota})^-(\alpha)|}
\left[\frac{|(P_{d-1}^{i,j,k,\iota})^+(\alpha)|}
{|(S_{d}^{i,j,k,\iota})^-(\alpha)\cap(P_{d-1}^{i,j,k,\iota})
^+(\alpha)|}\right]^\delta\\
&\quad\times\left(u^q\right)\left(\left(S_d^{i,j,k,\iota}
\right)^-(\alpha)\cap\left(P_{d-1}^{i,j,k,\iota}\right)
^+(\alpha)\right)\\&\leq
2^{(n+1)\delta}K\fint_{(S_d^{i,j,k,\iota})^-(\alpha)}u^q.
\end{align*}
Iterating this inequality and using (vi) and (iii), we obtain
\begin{align*}
\fint_{\widetilde{R}_{k,\iota}^+(\alpha)-(\mathbf{0},
\tau(1+\alpha)[l(\widetilde{R}_{k,\iota})]^p)}u^q&=
\fint_{(S_0^{i,j,k,\iota})^-(\alpha)}u^q\leq
\left[2^{(n+1)\delta}K\right]^{\widetilde{N}_\iota+N_j}
\fint_{(S_{\widetilde{N}_\iota+N_j}^{i,j,k,\iota})
^-(\alpha)}u^q\\
&\leq\left[2^{(n+1)\delta}K\right]^{3C_1}
\fint_{R_{i,j}^+(\alpha)-(\mathbf{0},\tau(1+\alpha)
[l(R_{i,j})]^p)}u^q.
\end{align*}
Hence \eqref{20240731.1506} holds with
$C_2:=[2^{(n+1)\delta}K]^{3C_1}$.

To show $(u,v)\in TA_{r,q}^+(\gamma)$,
we still need the following estimate.
From the fact that $\lceil s\rceil\leq2s$ for any
$s\in[1,\infty)$, \eqref{20240731.1610}, and the definition
of $J$, we deduce that
\begin{align}\label{20240731.1619}
2^{nm}J&=2^{nm}\left\lceil\frac{(1-\gamma)2^{pm}}{1-\alpha}
\right\rceil\leq\frac{1-\gamma}{1-\alpha}2^{(n+p)m+1}\\
&\leq\frac{1-\gamma}{1-\alpha}\left(\frac{1+\alpha}{1-\alpha}
\right)^{n+p}\left(\frac{1+\alpha}{\gamma}\right)
^{n+p}2^{\frac{n+p}{p-1}+3(n+p)+1}=:C_3.\notag
\end{align}
Applying the obvious facts that
\begin{align*}
R^+(\gamma)=\bigcup_{i=1}^{2^{nm}}\bigcup_{j=1}^J V_{i,j}^+
\ \ \mbox{and}\ \ R^-(\gamma)=\bigcup_{k=1}^{2^{nm}}
\bigcup_{\iota=1}^JU_{k,\iota}^-,
\end{align*}
$V_{i,j}^+\subset R_{i,j}^+(\alpha)$ for any
$i\in\mathbb{N}\cap[1,2^{nm}]$ and
$j\in\mathbb{N}\cap[1,J]$, \eqref{20240915.0844},
$U_{k,\iota}^-\subset\widetilde{R}_{k,\iota}^+(\alpha)
-(\mathbf{0},\tau(1+\alpha)l(\widetilde{R}_{k,\iota})^p)$ for
any $k\in\mathbb{N}\cap[1,2^{nm}]$ and
$\iota\in\mathbb{N}\cap[1,J]$, \eqref{20240915.0853},
\eqref{20240731.1506}, \eqref{20240731.1619},
\eqref{20240815.2233}, the definition of $J$, and the fact
that $\lceil s\rceil\leq2s$ for any $s\in[1,\infty)$, we
conclude that
\begin{align}\label{20240731.1628}
&\fint_{R^-(\gamma)}u^q\left[\fint_{R^+(\gamma)}
v^{-r'}\right]^{\frac{q}{r'}}\\
&\quad\leq\left[\sum_{k=1}^{2^{nm}}
\sum_{\iota=1}^{J}\frac{|U_{k,\iota}^-|}{|R^-(\gamma)|}
\fint_{U_{k,\iota}^-}u^q\right]
\left[\sum_{i=1}^{2^{nm}}\sum_{j=1}^{J}
\frac{|V_{i,j}^+|}{|R^+(\gamma)|}
\fint_{V_{i,j}^+}v^{-r'}\right]^{\frac{q}{r'}}\notag\\
&\quad\leq\left(\frac{1}{2^{nm}J}\right)^{1+\frac{q}{r'}}
\left[\frac{(1-\alpha)J}{2^{pm}(1-\gamma)}\right]
^{1+\frac{q}{r'}}\notag\\
&\quad\quad\times\left[\sum_{k=1}^{2^{nm}}\sum_{\iota=1}^{J}
\fint_{\widetilde{R}_{k,\iota}^+(\alpha)-(\mathbf{0},
\tau(1+\alpha)[l(\widetilde{R}_{k,\iota})]^p)}u^q\right]
\left[\sum_{i=1}^{2^{nm}}\sum_{j=1}^{J}
\fint_{R_{i,j}^+(\alpha)}v^{-r'}\right]^{\frac{q}{r'}}\notag\\
&\quad\leq\left(\frac{1}{2^{nm}J}\right)^{1+\frac{q}{r'}}
\left[\frac{(1-\alpha)J}{2^{pm}(1-\gamma)}\right]
^{1+\frac{q}{r'}}\max\left\{1,\
\left(2^{nm}J\right)^{\frac{q}{r'}-1}\right\}\notag\\
&\quad\quad\times\sum_{i,k=1}^{2^{nm}}\sum_{j,\iota=1}^J
\fint_{\widetilde{R}_{k,\iota}^+(\alpha)-(\mathbf{0},
\tau(1+\alpha)[l(\widetilde{R}_{k,\iota})]^p)}u^q
\left[\fint_{R_{i,j}^+(\alpha)}v^{-r'}\right]^{\frac{q}{r'}}
\notag\\
&\quad\leq\left(\frac{1}{2^{nm}J}\right)^{1+\frac{q}{r'}}
\left[\frac{(1-\alpha)J}{2^{pm}(1-\gamma)}\right]
^{1+\frac{q}{r'}}\max\left\{1,\
\left(2^{nm}J\right)^{\frac{q}{r'}-1}\right\}C_2\notag\\
&\quad\quad\times\sum_{i,k=1}^{2^{nm}}\sum_{j,\iota=1}^J
\fint_{R_{i,j}^+(\alpha)-(\mathbf{0},\tau(1+\alpha)
[l(R_{i,j})]^p)}u^q
\left[\fint_{R_{i,j}^+(\alpha)}v^{-r'}
\right]^{\frac{q}{r'}}\notag\\
&\quad\leq2^{1+\frac{q}{r'}}\max\left\{1,\,
C_3^{1-\frac{q}{r'}}\right\}C_2C.\notag
\end{align}
Taking the supremum over all $R\in\mathcal{R}_p^{n+1}$,
we obtain $(u,v)\in TA_{r,q}^+(\gamma)$ and
\begin{align*}
[u,v]_{TA_{r,q}^+(\gamma)}\leq2^{1+\frac{q}{r'}}\max\left\{1,\,
C_3^{1-\frac{q}{r'}}\right\}C_2C,
\end{align*}
which completes the proof of the sufficiency of (i).

Finally, to prove the present theorem, it remains to show
the sufficiency of (ii). Indeed, using
\cite[Exercises 1.1.3(a)]{Grafakos first volume}
and letting $r\to1^+$ in \eqref{20240731.1628} with the
assumption therein replaced by \eqref{20241006.0907}, we find
that the sufficiency of (ii) holds, which then completes
the proof of Theorem \ref{independence of time lag 2}.
\end{proof}

\begin{remark}
\begin{enumerate}
\item[(i)]
The assumption that $u\in A_\infty^+(\gamma)$
is only used to prove the sufficiency of
Theorem~\ref{independence of time lag 2}.

\item[(ii)]
Theorem~\ref{independence of time lag 2} when both $r=q$ and
$u=v$ coincides with \cite[Theorem 3.1]{km(am-2024)}.
\end{enumerate}
\end{remark}

The following is a direct consequence of Theorem
\ref{independence of time lag 2}.

\begin{corollary}\label{independence of time lag}
Let $\gamma\in(0,1)$, $1\leq r\leq q<\infty$, and $(u,v)$
be a pair of nonnegative functions on $\mathbb{R}^{n+1}$.
Assume that $u\in A_\infty^+(\gamma)$. Then the following
statements hold.
\begin{enumerate}
\item[\rm(i)] For any $\alpha\in(0,1)$, $(u,v)\in TA_{r,q}^+
(\gamma)$ if and only if $(u,v)\in TA_{r,q}^+(\alpha)$.

\item[\rm(ii)] $(u,v)\in TA_{r,q}^+(\gamma)$ if and only if
there exists a positive constant $C$ such that, for any
$R\in\mathcal{R}_p^{n+1}$,
\begin{align*}
\fint_{R^{--}(\gamma)}u^q\left[\fint_{R^+(\gamma)}v^
{-r'}\right]^{\frac{q}{r'}}\leq C,
\end{align*}
where
$R^{--}(\gamma):=R^-(\gamma)-(\mathbf{0},(1+\gamma)[l(R)]^p)$.
\end{enumerate}
\end{corollary}

\begin{remark}
Corollary \ref{independence of time lag} when $u=v$ coincides
with \cite[Lemma 2.1(3)]{mhy(fm-2023)}.
\end{remark}

Using Corollary
\ref{independence of time lag},
we obtain the following self-improving property of the parabolic Muckenhoupt
two-weight class with time lag.

\begin{corollary}\label{self-improving property}
Let $\gamma\in(0,1)$, $1<r\leq q<\infty$, and $(u,v)\in
TA_{r,q}^+(\gamma)$. If $u\in A_\infty^+(\gamma)$, then
there exists $\delta_0\in(0,\infty)$, depending only on $n$,
$p$, $\gamma$, $r$, $q$, and $[u,v]_{TA_{r,q}^+(\gamma)}$,
such that, for any $\delta\in(0,\delta_0)$, $(u,v)\in
TA_{r,q+\delta}^+(\gamma)$.
\end{corollary}

\begin{proof}
From the assumption that $u\in A_\infty^+(\gamma)$ and
\cite[Lemma 7.4]{ks(apde-2016)},
it follows that there exists $q_u\in(1,\infty)$ such that
$u^q\in A_{q_u}^+(\gamma)$.
By this and \cite[Corollary 5.3]{km(am-2024)},
we conclude that there exist
positive constants $C$ and $\delta_0$, depending only on $n$,
$p$, $\gamma$, $q$, and the weight constant of $u^q$, such
that, for any $R\in\mathcal{R}_p^{n+1}$,
\begin{align*}
\left[\fint_{R^-(\gamma)}u^{q(1+\delta_0)}\right]
^{\frac{1}{1+\delta_0}}\leq C\fint_{R^+(\gamma)}u^q,
\end{align*}
which further implies that
\begin{align*}
&\left[\fint_{R^{--}(\gamma)}u^{q(1+\delta_0)}\right]
^{\frac{1}{q(1+\delta_0)}}\left[\fint_{R^+(\gamma)}v^
{-r'}\right]^{\frac{1}{r'}}\\
&\quad\leq C\left[\fint_{R^-(\gamma)}u^q\right]^{\frac{1}{q}}
\left[\fint_{R^+(\gamma)}v^{-r'}\right]^{\frac{1}{r'}}
\leq C[u,v]_{TA_{r,q}^+(\gamma)}^{\frac{1}{q}}.
\end{align*}
From this and Corollary \ref{independence of time lag}(ii),
we infer that $(u,v)\in TA_{r,q+\delta_0}^+(\gamma)$. This,
together with Proposition \ref{nested property}(ii), further
implies that, for any $\delta\in(0,\delta_0)$, $(u,v)\in
TA_{r,q+\delta}^+(\gamma)$. This finishes the proof of
Corollary \ref{self-improving property}.
\end{proof}

\section{Characterizations of Weighted Boundedness of\\
Parabolic Fractional Maximal Operators with Time Lag}
\label{section4}

In this section, we characterize the parabolic Muckenhoupt
(two-weight) class with time lag via the (weak-type) weighted
boundedness of the parabolic fractional maximal operator.
Recall that, for any given $q\in[1,\infty)$ and any nonnegative
locally integrable function $\omega$ on $\mathbb{R}^{n+1}$,
the \emph{weighted weak Lebesgue space
$L^{q,\infty}(\mathbb{R}^{n+1},\omega)$} is defined to be the
set of all measurable functions $f$ on $\mathbb{R}^{n+1}$ such
that
\begin{align*}
\|f\|_{L^{q,\infty}(\mathbb{R}^{n+1},\omega)}:=
\sup_{\lambda\in(0,\infty)}\lambda\left[\omega(\{|f|>\lambda\})
\right]^\frac1q<\infty.
\end{align*}
Specifically, if $\omega\equiv1$, then
$L^{q,\infty}(\mathbb{R}^{n+1},1)$ is exactly the
\emph{weak Lebesgue space} and we simply write
$L^{q,\infty}(\mathbb{R}^{n+1}):=
L^{q,\infty}(\mathbb{R}^{n+1},1)$. The following theorem is
the main result of this section. We borrow some ideas and
techniques from \cite{fmo(tams-2011),km(am-2024)}.

\begin{theorem}\label{weak type inequality}
Let $\gamma\in(0,1)$, $\beta\in[0,1)$, $1\leq r\leq
q<\infty$, $\beta=\frac{1}{r}-\frac{1}{q}$, and $(u,v)$ be a
pair of nonnegative functions on $\mathbb{R}^{n+1}$. If $u\in
A_\infty^+(\gamma)$, then $(u,v)\in TA_{r,q}^+(\gamma)$ if
and only if there exists a positive constant $C$
such that, for any $f\in L^r(\mathbb{R}^{n+1},v^r)$,
\begin{align}\label{20240725.1724}
\left\|M^{\gamma+}_\beta(f)\right\|
_{L^{q,\infty}(\mathbb{R}^{n+1},u^q)}\leq C
\|f\|_{L^r(\mathbb{R}^{n+1},v^r)}.
\end{align}
\end{theorem}

\begin{proof}
We first show the sufficiency. Assume that
\eqref{20240725.1724} holds for any $f\in
L^r(\mathbb{R}^{n+1},v^r)$. Fix
$R\in\mathcal{R}_p^{n+1}$. From Definition
\ref{parabolic maximal operators}, we deduce
that, for any $\lambda\in(0,|R^+(\gamma)|^\beta
(v^{-r'})_{R^+(\gamma)})$ and for almost every $(x,t)\in
R^-(\gamma)$,
\begin{align*}
\lambda<\left|R^+(\gamma)\right|^\beta
\left(v^{-r'}\right)_{R^+(\gamma)}=
\left|R^+(\gamma)\right|^\beta
\fint_{R^+(\gamma)}v^{-r'}\leq
M^{\gamma+}_\beta\left(v^{-r'}\boldsymbol{1}
_{R^+(\gamma)}\right)(x,t),
\end{align*}
which further implies that
$R^-(\gamma)\subset\{M^{\gamma+}_\beta(v^{-r'}
\boldsymbol{1}_{R^+(\gamma)})>\lambda\}$
up to a set of measure zero. Combining this and
\eqref{20240725.1724}, we obtain
\begin{align*}
\int_{R^-(\gamma)}u^q&\leq
\left(u^q\right)\left(\left\{M^{\gamma+}_\beta\left(v^{-r'}
\boldsymbol{1}_{R^+(\gamma)}\right)>\lambda\right\}\right)
\leq\frac{C}{\lambda^q}\left[\int
_{R^+(\gamma)}v^{-r'}\right]^{\frac{q}{r}}.
\end{align*}
Letting $\lambda\to|R^+(\gamma)|^\beta
(v^{-r'})_{R^+(\gamma)}$, dividing both
sides by $|R^+(\gamma)|$, and using $\frac{1}{r}-\frac{1}{q}=\beta$, we find that
\begin{align*}
\fint_{R^-(\gamma)}u^q\left[\fint_{R^+(\gamma)}v^{-r'}\right]
^{\frac{q}{r'}}\leq C.
\end{align*}
Taking the supremum over all $R\in\mathcal{R}_p^{n+1}$,
we conclude that $(u,v)\in TA_{r,q}^+(\gamma)$ and
$[u,v]_{TA_{r,q}^+(\gamma)}\leq C$, which completes the
proof of the sufficiency.

Next, we prove the necessity. Let $(u,v)\in
TA_{r,q}^+(\gamma)$, $f\in L^r(\mathbb{R}^{n+1},v^r)$, and
$\lambda\in(0,\infty)$. We need to show that
\eqref{20240725.1724} holds. To this end, we divide the
proof into the following five steps.

\emph{Step 1.} In this step, we make the following
simplification.
\begin{enumerate}
\item[(i)] We may assume that $f$ is bounded and has compact
support. Indeed, for any $k\in\mathbb{N}$, let
\begin{align*}
f_k:=\max\left\{|f|,\,k\right\}\mathbf{1}_{R(\mathbf{0},0,k)}.
\end{align*}
Then, for any $k\in\mathbb{N}$, $f_k$ is bounded and has
compact support and $f_k\to f$ almost
everywhere on $\mathbb{R}^{n+1}$
as $k\to\infty$. From this and the
monotone convergence theorem, it follows that
\begin{align}\label{20240725.1904}
\int_{\mathbb{R}^{n+1}}|f_k|^rv^r\to\int_{\mathbb{R}^{n+1}}
|f|^rv^r
\end{align}
as $k\to\infty$. In addition, by an argument similar to that
used in the proof of \cite[Lemma 3.30]{cf-Variable Lebesgue Spaces}
and the monotone convergence theorem again, we find that
\begin{align*}
\left(u^q\right)\left(\left\{M^{\gamma+}_\beta(f_k)>\lambda
\right\}\right)\to\left(u^q\right)\left(\left\{M^{\gamma+}
_\beta(f)>\lambda\right\}\right)
\end{align*}
as $k\to\infty$. From this and \eqref{20240725.1904},
we deduce that,
to prove \eqref{20240725.1724} for $f$, it suffices to show
that \eqref{20240725.1724} holds for any $f_k$ with
$k\in\mathbb{N}$ and the positive constant $C$
independent of $k$, $\lambda$, and $f$.

\item[(ii)] We may assume that both $u^q$
and $v^r$ have a lower bound $A$ for some $A\in(0,\infty)$.
Indeed, applying
Proposition \ref{max and min}, we conclude that $(\max\{u,\,
A^\frac{1}{q}\},\max\{v,\,A^\frac{1}{r}\})\in
TA_{r,q}^+(\gamma)$. If we have
\begin{align*}
\left(\max\{u^q,\,A\}\right)
\left(\left\{M^{\gamma+}_\beta(f)>\lambda\right\}\right)
\lesssim\frac{1}{\lambda^q}\left[\int_{\mathbb{R}^{n+1}}|f|^r
\max\left\{v^r,\,A\right\}\right]^{\frac{q}{r}},
\end{align*}
where the implicit positive constant is independent of
$A$, $\lambda$, and $f$, then we obtain
\eqref{20240725.1724} via letting $A\to0$ and taking the
supremum over $\lambda\in(0,\infty)$.

\item[(iii)] Fix $a\in(0,1)$. Let $M_{\beta,a}^{\gamma+}$
denote the \emph{truncated uncentered parabolic forward in
time maximal operator}, that is,
for any $g\in L_{\mathrm{loc}}^1(\mathbb{R}^{n+1})$ and
$(x,t)\in\mathbb{R}^{n+1}$,
\def\gfz{\genfrac{}{}{0pt}{}}
\begin{align*}
M_{\beta,a}^{\gamma+}(g)(x,t):=\sup_{\genfrac{}{}{0pt}{}{R\in\mathcal{R}
_p^{n+1}}{(x,t)\in R^-(\gamma),\,l(R)\in[a,\infty)}}
\left|R^+(\gamma)\right|^\beta\fint_{R^+(\gamma)}|g|.
\end{align*}
To prove \eqref{20240725.1724},
we only need to show that
\begin{align}\label{20240725.2114}
\left(u^q\right)\left(\left\{M_{\beta,a}^{\gamma+}(f)>\lambda
\right\}\right)\lesssim\frac{1}{\lambda^q}
\left(\int_{\mathbb{R}^{n+1}}|f|^rv^r\right)^{\frac{q}{r}},
\end{align}
where the implicit positive constant is independent of
$a$, $\lambda$, and $f$. Indeed, in \eqref{20240725.2114},
letting $a\to0$ and taking the supremum over
$\lambda\in(0,\infty)$, we then obtain \eqref{20240725.1724}.

\item[(iv)] To prove \eqref{20240725.2114},
it suffices to show that
\begin{align}\label{20240725.2120}
\left(u^q\right)\left(\left\{\lambda<M_{\beta,a}^{\gamma+}(f)
\leq2\lambda\right\}\right)
\leq\frac{1}{\lambda^q}\left(\int_{\mathbb{R}^{n+1}}|f|^rv^r
\right)^{\frac{q}{r}},
\end{align}
where the implicit positive constant is independent of
$a$, $\lambda$, and $f$. Indeed, from \eqref{20240725.2120},
we infer that
\begin{align*}
\left(u^q\right)\left(\left\{M_{\beta,a}^{\gamma+}(f)>\lambda
\right\}\right)&\leq\sum_{k\in\mathbb{Z}_+}\left(u^q\right)
\left(\left\{2^k\lambda<M_{\beta,a}^{\gamma+}(f)\leq
2^{k+1}\lambda\right\}\right)\\
&\lesssim\sum_{k\in\mathbb{Z}_+}\frac{1}{2^{kq}\lambda^q}
\left(\int_{\mathbb{R}^{n+1}}|f|^rv^r\right)^{\frac{q}{r}}
\sim\frac{1}{\lambda^q}
\left(\int_{\mathbb{R}^{n+1}}|f|^rv^r\right)^{\frac{q}{r}}
\end{align*}
and hence \eqref{20240725.2114} holds.

\item[(v)] Fix any compact set
$K\subset\{\lambda<M_{\beta,a}^{\gamma+}(f)
\leq2\lambda\}$. Then, to prove \eqref{20240725.2120},
it suffices to show
\begin{align}\label{20240725.2122}
\left(u^q\right)(K)\lesssim\frac{1}{\lambda^q}
\left(\int_{\mathbb{R}^{n+1}}|f|^rv\right)^{\frac{q}{r}},
\end{align}
where the implicit positive constant is independent of
$K$, $a$, $\lambda$, and $f$. Indeed, using the inner
regularity of the Lebesgue measure (see, for instance,
\cite[Theorem~2.14(d)]{rudin}) and taking the supremum
over all compact subsets of
$\{\lambda<M_{\beta,a}^{\gamma+}(f)\leq2\lambda\}$, we
then obtain \eqref{20240725.2120}.
\end{enumerate}

\emph{Step 2.} In this step, we aim to prove that
there exist $N\in\mathbb{N}$ and
a sequence $\{P_i\}_{i=1}^N$ of
parabolic rectangles such that
\begin{align}\label{20240725.2143}
\left(u^q\right)(K)\leq2\sum_{i=1}^{N}\left(u^q\right)
\left(P_i^-(\alpha)\right),
\end{align}
where $\alpha:=\frac{\gamma}{5^p}$. To
begin with, from the definitions of both $K$ and
$M_{\beta,a}^{\gamma+}$, we deduce that, for any given
$(x,t)\in K$, there exists $R_{(x,t)}\in\mathcal{R}_p^{n+1}$
such that $(x,t)\in R_{(x,t)}^-(\gamma)$,
$l(R_{(x,t)})\in[a,\infty)$, and
\begin{align}\label{20240725.2302}
\lambda<\left|R^+_{(x,t)}(\gamma)\right|^\beta
\fint_{R_{(x,t)}^+(\gamma)}|f|\leq2\lambda.
\end{align}
Using the assumption on $f$ in (i) of Step 1,
we find that $f\in L^1(\mathbb{R}^{n+1})$.
This, together with \eqref{20240725.2302}, further implies that
\begin{align*}
2^n(1-\gamma)\left[l(R_{(x,t)})\right]^{n+p}=
\left|R_{(x,t)}^+(\gamma)\right|
<\left[\frac{1}{\lambda}\int_{R_{(x,t)}^+(\gamma)}|f|\right]
^{\frac{1}{1-\beta}}
\leq\left[\frac{\|f\|_{L^1(\mathbb{R}^{n+1})}}{\lambda}\right]
^{\frac{1}{1-\beta}}<\infty.
\end{align*}
Combining this and the fact that $l(R_{(x,t)})\in[a,\infty)$,
we conclude that
\begin{align}\label{20240726.1137}
a\leq l\left(R_{(x,t)}\right)
\leq\left[\frac{\|f\|_{L^1(\mathbb{R}^{n+1})}}
{\lambda}\right]^{\frac{1}{(n+p)(1-\beta)}}
\left[\frac{1}{2^n(1-\gamma)}\right]^{\frac{1}{n+p}}.
\end{align}
Let $P_{(x,t)}:=5R_{(x,t)}$. Here and thereafter, for any
$R:=\prod_{i=1}^{n}[y_i-L_i,y_i+L_i]\times(s-L^p,s+L^p)\subset
\mathbb{R}^{n+1}$
with $\{y_i\}_{i=1}^{n}\subset\mathbb{R}$,
$\{L_i\}_{i=1}^{n}\subset(0,\infty)$, $s\in\mathbb{R}$, and
$L\in(0,\infty)$ and for any $\Lambda\in(0,\infty)$, define
$\Lambda R:=\prod_{i=1}^{n}[y_i-\Lambda L_i,y_i+\Lambda
L_i]\times(s-(\Lambda L)^p,s+(\Lambda L)^p)$. Then it is easy
to verify that $R_{(x,t)}^+(\gamma)\subset
P_{(x,t)}^+(\alpha)$ and $R_{(x,t)}^-(\gamma)\subset
P_{(x,t)}^-(\alpha)$. In addition, from the fact
that $K$ is compact and \eqref{20240726.1137}, it follows that
$\bigcup_{(x,t)\in K}P_{(x,t)}$ is bounded and hence $u^q$ is
integrable on $\bigcup_{(x,t)\in K}P_{(x,t)}$. Applying this,
the absolute continuity of the Lebesgue integral,
\eqref{20240726.1137}, and (ii) in Step 1, we find that there
exists $\epsilon\in(0,1)$ such that, for any $(x,t)\in K$,
\begin{align*}
&\left(u^q\right)\left((1+\epsilon)P_{(x,t)}^-(\alpha)\setminus
P_{(x,t)}^-(\alpha)\right)\\
&\quad\leq A2^n5^{n+p}a^{n+p}(1-\alpha)
\leq A\left|P_{(x,t)}^-(\alpha)\right|
\leq(u^q)\left(P_{(x,t)}^-(\alpha)\right),
\end{align*}
which further implies that
\begin{align}\label{20240726.1148}
\left(u^q\right)\left((1+\epsilon)P_{(x,t)}^-(\alpha)\right)
\leq2\left(u^q\right)\left(P_{(x,t)}^-(\alpha)\right).
\end{align}

On the other hand, let
\begin{align}\label{20240731.2156}
\widetilde{\epsilon}:=\min\left\{
\left(\frac{1-\alpha}{2}\right)^{\frac{1}{p}}
\left[(1+\epsilon)^p-1\right]^{\frac{1}{p}},\,\epsilon\right\}.
\end{align}
From the finite covering theorem and the fact that
\begin{align*}
K\subset\bigcup_{(x,t)\in K}
\mathrm{int\,}(R(x,t,5a\widetilde{\epsilon})),
\end{align*}
we deduce that there exist
$N_1\in\mathbb{N}$ and $\{(x_k,t_k)\}_{k=1}^{N_1}\subset K$
such that
\begin{align}\label{20240726.2120}
K\subset\bigcup_{k=1}^{N_1}\mathrm{int\,}(R(x_k,t_k,
5a\widetilde{\epsilon}))\subset\bigcup_{k=1}^{N_1}
R(x_k,t_k,5a\widetilde{\epsilon}),
\end{align}
where, for any $E\subset\mathbb{R}^{n+1}$, $\mathrm{int\,}(E)$
denotes the interior of $E$.

Now, we select a subsequence of $\{R_{(x_k,t_k)}\}
_{k=1}^{N_1}$ via two steps. Assume that, for any
$k\in\mathbb{N}\cap[1,N_1]$, $R_{(x_k,t_k)}:=R(y_k,s_k,L_k)$
with $(y_k,s_k)\in\mathbb{R}^{n+1}$ and $L_k\in(0,\infty)$.
Without loss of generality, we may assume that
the tops of $\{R_{(x_k,t_k)}\}_{k=1}^{N_1}$ are monotonically
descending, that is, for any $k,j\in\mathbb{N}\cap[1,N_1]$
with $k\leq j$, $s_k+L_k^p\geq s_j+L_j^p$; otherwise, we can
rearrange $\{R_{(x_k,t_k)}\}_{k=1}^{N_1}$ in terms of the
$t$-coordinates of their tops.

Then we make the
first selection inductively. Select $R_{(x_1,t_1)}$ and denote
it by $R_{(x_{k_1},t_{k_1})}$. Suppose that we have selected
the subsequence $\{R_{(x_{k_i},t_{k_i})}\}_{i=1}^{m}$
of $\{R_{(x_k,t_k)}\}_{k=1}^{N_1}$, where $m\in\mathbb{N}
\cap[1,N_1)$ and $k_m<N_1$. Let
\begin{align*}
j_m:=\min\left\{j\in\mathbb{N}:\ k_m+j\leq N_1\mbox{ and }
\left(x_{k_m+j},t_{k_m+j}\right)\notin\bigcup_{i=1}^{m}
P_{(x_{k_i},t_{k_i})}^-(\alpha)\right\}
\end{align*}
with the convention $\inf\emptyset=\infty$.
If $j_m\in\mathbb{N}$, then select
$R_{(x_{k_m+j_m},t_{k_m+j_m})}$ and denote it by
$R_{(x_{k_{m+1}},t_{k_{m+1}})}$; otherwise, we terminate the
selection process. In this manner, we have obtained a
subsequence $\{R_{(x_{k_i},t_{k_i})}\}_{i=1}^{N_2}$ of
$\{R_{(x_k,t_k)}\}_{k=1}^{N_1}$, where $N_2\in\mathbb{N}
\cap[1,N_1]$. This finishes the first step of selection
process.

Without loss of generality, we may assume that the
edge lengths of $\{R_{(x_{k_i},t_{k_i})}\}_{i=1}^{N_2}$ are
monotonically decreasing; otherwise, we can rearrange
$\{R_{(x_{k_i},t_{k_i})}\}_{i=1}^{N_2}$ in terms of their edge
lengths. Then we make the second selection inductively.
Select $R_{(x_{k_1},t_{k_1})}$ and denote it by
$R_{(x_{k_{r_1}},t_{k_{r_1}})}$. Assume that we have selected
the subsequence $\{R_{(x_{k_{r_i}},t_{k_{r_i}})}\}_{i=1}^m$
of $\{R_{(x_{k_i},t_{k_i})}\}_{i=1}^{N_2}$, where
$m\in\mathbb{N}\cap[1,N_2)$ and $k_{r_m}<N_2$. Let
\begin{align*}
j_{r_m}:=\min\left\{j\in\mathbb{N}:\ k_{r_m}+j\leq N_2
\mbox{ and }P_{(x_{k_{r_m}+j},t_{k_{r_m}+j})}
^-(\alpha)\nsubset\bigcup_{i=1}^m
P_{(x_{k_{r_i}},t_{k_{r_i}})}^-(\alpha)\right\}
\end{align*}
with the convention $\inf\emptyset=\infty$.
If $j_{r_m}\in\mathbb{N}$, then select
$R_{(x_{k_{r_m+}+j_{r_m}},t_{k_{r_m}+j_{r_m}})}$ and denote
it by $R_{(x_{k_{r_{m+1}}},t_{k_{r_{m+1}}})}$; otherwise, we
stop the selection process. By this way, we have obtained a
subsequence $\{R_{(x_{k_{r_i}},t_{k_{r_i}})}\}_{i=1}^N$ of
$\{R_{(x_{k_i},t_{k_i})}\}_{i=1}^{N_2}$, where
$N\in\mathbb{N}\cap[1,N_2]$, which completes the second step
of selection process. For convenience, for any $i\in\mathbb{N}
\cap[1,N]$, we simply write $R_{(x_{k_{r_i}},t_{k_{r_i}})}$ as
$R_i$. In conclusion, we have selected a subsequence
$\{R_i\}_{i=1}^N$ of $\{R_{(x_k,t_k)}\}_{k=1}^{N_1}$.

According to the above selection, we conclude that the
following statements hold.
\begin{enumerate}
\item[(i)] For any $i,j\in\mathbb{N}\cap[1,N]$ with $i\neq j$,
$R_i^-(\gamma)\nsubset R_j^-(\gamma)$, which is a
direct consequence of the first selection.

\item[(ii)] For any $k\in\mathbb{N}\cap[1,N_1]$, there exists
$i\in\mathbb{N}\cap[1,N]$ such that $(x_k,t_k)\in
P_i^-(\alpha)$. Indeed, if
$R_{(x_k,t_k)}\notin\{R_{(x_{k_i},t_{k_i})}\}
_{i=1}^{N_2}$, then, based on the first
selection, there exists $i\in\mathbb{N}\cap[1,N_2]$ such that
$(x_k,t_k)\in P_{(x_{k_i},t_{k_i})}^-(\alpha)$. If
$R_{(x_{k_i},t_{k_i})}\notin\{R_i\}_{i=1}^N$, then, in light
of the second selection, we have
$P_{(x_{k_i},t_{k_i})}^-(\alpha)\subset\bigcup_{i=1}^N
P_i^-(\alpha)$.

\item[(iii)] For any $k\in\mathbb{N}\cap[1,N_2]$, there
exists $i\in\mathbb{N}\cap[1,N]$ such that
\begin{align}\label{1957}
R(x_k,t_k,5a\widetilde{\epsilon})\subset(1+\epsilon)
P_i^-(\alpha).
\end{align}
Indeed, let $i\in\mathbb{N}\cap[1,N]$ be such that
\begin{align*}
(x_k,t_k)\in
P_i^-(\alpha)=P_{(x_{k_{r_i}},t_{k_{r_i}})}^-(\alpha)=
Q\left(y_{k_{r_i}},5L_{k_{r_i}}\right)\times
\left(s_{k_{r_i}}-\left(5L_{k_{r_i}}\right)^p,
s_{k_{r_i}}-\alpha\left(5L_{k_{r_i}}\right)^p\right).
\end{align*}
Then $\|x_k-y_{k_{r_i}}\|_{\infty}\in[0,5L_{k_{r_i}}]$
and $s_{k_{r_i}}-t_k\in(\alpha(5L_{k_{r_i}})^p,
(5L_{k_{r_i}})^p)$. From this, \eqref{20240726.1137},
and \eqref{20240731.2156}, we infer that, for any $(y,s)\in
R(x_k,t_k,5a\widetilde{\epsilon})$,
\begin{align*}
s&<t_k+(5a)^p\widetilde{\epsilon}^p
<s_{k_{r_i}}-\alpha\left(5L_{k_{r_i}}\right)^p
+(5a)^p\widetilde{\epsilon}^p\\
&\leq s_{k_{r_i}}+(\widetilde{\epsilon}^p-\alpha)
\left(5L_{k_{r_i}}\right)^p
\leq s_{k_{r_i}}-\frac{1+\alpha}{2}\left(5L_{k_{r_i}}\right)^p
+\frac{1-\alpha}{2}(1+\epsilon)^p\left(5L_{k_{r_i}}\right)^p,
\end{align*}
\begin{align*}
s&>t_k-(5a)^p\widetilde{\epsilon}^p
>s_{k_{r_i}}-\left(5L_{k_{r_i}}\right)^p
-(5a)^p\widetilde{\epsilon}^p\\
&\geq s_{k_{r_i}}-\left(5L_{k_{r_i}}\right)^p
(1+\widetilde{\epsilon}^p)
\geq s_{k_{r_i}}-\frac{1+\alpha}{2}\left(5L_{k_{r_i}}\right)^p
-\frac{1-\alpha}{2}(1+\epsilon)^p\left(5L_{k_{r_i}}\right)^p,
\end{align*}
and
\begin{align*}
\left\|y-y_{k_{r_i}}\right\|_\infty\leq\|y-x_k\|_\infty+
\left\|x_k-y_{k_{r_i}}\right\|_\infty\leq5a\widetilde{\epsilon}
+5L_{k_{r_i}}\leq5L_{k_{r_i}}(1+\epsilon),
\end{align*}
which further implies that
$(y,s)\in(1+\epsilon)P_i^-(\alpha)$. By the arbitrariness
of $(y,s)$, we obtain
$R(x_k,t_k,5a\widetilde{\epsilon})\subset(1+\epsilon)
P_i^-(\alpha)$.

\item[(iv)] For any given $k\in\mathbb{Z}$ and $R_i,R_j\in
\{R_i\}_{i=1}^N$ with $l(R_i),l(R_j)\in
(\frac{1}{2^{k+1}},\frac{1}{2^k}]$, we have
\begin{align}\label{1955}
R_i^-(\gamma)\cap
R_j^-(\gamma)=\emptyset.
\end{align}
We show this by
contradiction. Indeed, from the first selection, we deduce
that $R_i^-(\gamma)\nsubset P_j^-(\alpha)$ or
$R_j^-(\gamma)\nsubset P_i^-(\alpha)$.
Without loss of generality, we may assume that
$R_i^-(\gamma)\nsubset P_j^-(\alpha)$. Suppose that there
exist $k_0\in\mathbb{Z}$ and $R_i,R_j\in\{R_i\}_{i=1}^N$
satisfying $l(R_i),l(R_j)\in(\frac{1}{2^{k_0+1}},
\frac{1}{2^{k_0}}]$ and $R_i^-(\gamma)\cap
R_j^-(\gamma)\neq\emptyset$. Fix $(y_0,s_0)\in
R_i^-(\gamma)\cap R_j^-(\gamma)$. Then, for any $(y,s)\in
R_i^-(\gamma)=R_{(x_{k_{r_i}},t_{k_{r_i}})}^-(\gamma)$,
\begin{align*}
s_{k_{r_j}}-s\geq\gamma\left(L_{k_{r_j}}\right)^p=\alpha
\left(5L_{k_{r_j}}\right)^p,
\end{align*}
\begin{align*}
s-s_{k_{r_j}}&=s-s_0+s_0-s_{k_{r_j}}>s-s_0-
\left(L_{k_{r_j}}\right)^p\\
&>-(1-\gamma)\left(L_{k_{r_i}}\right)^p
-\left(L_{k_{r_j}}\right)^p
\geq-\left(2^p(1-\gamma)+1\right)\left(L_{k_{r_j}}\right)^p
>-\left(5L_{k_{r_j}}\right)^p,
\end{align*}
and
\begin{align*}
\left\|y-y_{k_{r_j}}\right\|_\infty\leq\|y-y_0\|_\infty+
\left\|y_0-y_{k_{r_j}}\right\|_\infty\leq2L_{k_{r_i}}
+L_{k_{r_j}}\leq5L_{k_{r_j}},
\end{align*}
which implies that $R_i^-(\gamma)\subset
P_j^-(\alpha)$. This contradicts
$R_i^-(\gamma)\nsubset P_j^-(\alpha)$.
Thus, \eqref{1955} holds.
\end{enumerate}
Using \eqref{20240726.2120} and \eqref{1957},
we conclude that
\begin{align*}
K\subset\bigcup_{k=1}^{N_1}R(x_k,t_k,5a\widetilde{\epsilon})
\subset\bigcup_{i=1}^N(1+\epsilon)P_i^-(\alpha),
\end{align*}
which, combined with \eqref{20240726.1148},
completes the proof of \eqref{20240725.2143} and hence Step 2.

\emph{Step 3.}
In this step, we prove that there exists a positive constant
$C$ such that, for any $i\in\mathbb{N}\cap[1,N]$,
\begin{align}\label{20240727.2207}
\sum_{j\in\Gamma_i}\int_{R_j^+(\gamma)}|f|\leq
C_1\int_{R_i^+(\gamma)}|f|,
\end{align}
where
$\Gamma_i:=\{j\in\mathbb{N}\cap[1,N]:\
R_i^+(\gamma)\cap R_j^+(\gamma)\neq\emptyset,\
l(R_j)<l(R_i)\}$.
Fix $i\in\mathbb{N}\cap[1,N]$ and let
\begin{align*}
\Gamma_{i,1}:=\left\{j\in \Gamma_i:\ R_j^+(\gamma)\nsubset
R_i^+(\gamma)\right\}
\ \ \mbox{and}\ \
\Gamma_{i,2}:=\left\{j\in \Gamma_i:\ R_j^+(\gamma)\subset
R_i^+(\gamma)\right\}.
\end{align*}
Then $\Gamma_{i,1}\cup\Gamma_{i,2}=\Gamma_{i}$ and
$\Gamma_{i,1}\cap\Gamma_{i,2}=\emptyset$.
To show \eqref{20240727.2207}, it suffices to prove that
there exist positive constants $C_2$ and $C_3$, depending only
on $n$, $p$, and $\gamma$, such that, for any $h\in\{1,2\}$,
\begin{align}\label{20240727.2217}
\sum_{j\in\Gamma_{i,h}}\left|R_j^+(\gamma)\right|\leq
C_{h+1}\left|R_i^+(\gamma)\right|.
\end{align}
Indeed, from this and \eqref{20240725.2302}, we infer that
\begin{align*}
\sum_{j\in\Gamma_i}\int_{R_j^+(\gamma)}|f|
&=\left(\sum_{j\in\Gamma_{i,1}}+\sum_{j\in\Gamma_{i,2}}\right)
\int_{R_j^+(\gamma)}|f|
\leq2\lambda\left(\sum_{j\in\Gamma_{i,1}}
+\sum_{j\in\Gamma_{i,2}}\right)\left|R_j^+(\gamma)\right|\\
&\leq2(C_2+C_3)\lambda\left|R_i^+(\gamma)\right|
\leq2(C_2+C_3)\int_{R_i^+(\gamma)}|f|
\end{align*}
and hence \eqref{20240727.2207} holds with $C_1:=2(C_2+C_3)$.

Next, we turn to show \eqref{20240727.2217}. Let
$k_0\in\mathbb{Z}$ be such that $l(R_i)\in(\frac{1}{2^{k_0+1}},
\frac{1}{2^{k_0}}]$. We first consider the case $h=1$.
We claim that, for any $k\in\mathbb{Z}\cap[k_0,\infty)$, there
exists a measurable set $E_k\subset\mathbb{R}^{n+1}$ such that
\begin{align*}
\bigcup_{j\in\Gamma_{i,1},\,l(R_j)
\in(\frac{1}{2^{k+1}},\frac{1}{2^k}]}R_j\subset E_k
\end{align*}
and
\begin{align*}
|E_k|\leq\frac{2\times3^n(3-\gamma)}{2^{kp}}[l(R_i)]^n
+\frac{2^33^{n-1}(2-\gamma)n}{2^k}[l(R_i)]^{n-1+p}.
\end{align*}
Indeed, fix $k\in\mathbb{Z}\cap[k_0,\infty)$. Denote the
$2n+2$ faces of $R_i^+(\gamma)$ by $\{S_m\}_{m=1}^{2n+2}$,
where $S_1$ and $S_2$ denote, respectively, the top and the
bottom of $R_i^+(\gamma)$. For any given $j\in\Gamma_{i,1}$
such that $l(R_j)\in(\frac{1}{2^{k+1}},\frac{1}{2^k}]$, since
$R_j^+(\gamma)\nsubset R_i^+(\gamma)$ and $R_j^+(\gamma)\cap
R_i^+(\gamma)\neq\emptyset$, it follows that $R_j^+(\gamma)$
intersects with the boundary of $R_i^+(\gamma)$, that is,
\begin{align}\label{20240728.1611}
R_j^+(\gamma)\cap\left(\bigcup_{m=1}^{2n+2}S_m\right)
\neq\emptyset.
\end{align}
For any $S_m$ with $m\in\mathbb{N}\cap[1,2n+2]$, there exists
a rectangle $E_{k,m}\subset\mathbb{R}^{n+1}$ such that
\begin{align*}
\bigcup_{\genfrac{}{}{0pt}{}{j\in\Gamma_{i,1}}{l(R_j)
\in(\frac{1}{2^{k+1}},\frac{1}{2^k}],\,R_j^+(\gamma)\cap
S_m\neq\emptyset}}R_j\subset E_{k,m}
\end{align*}
and
\begin{align*}
\left|E_{k,m}\right|=
\begin{cases}
(3-\gamma)\left[l(R_i)+2l(R_j)\right]^n
[l(R_j)]^p&\text{if }m\in\{1,2\},\\
2l(R_j)\left[l(R_i)+2l(R_j)\right]
^{n-1}\\
\quad\times\left\{(1-\gamma)[l(R_i)]^p
+(3-\gamma)[l(R_j)]^p\right\}&\text{if}\
m\in\mathbb{N}\cap[3,2n+2].
\end{cases}
\end{align*}
Applying this, \eqref{20240728.1611},
$l(R_j)<l(R_i)$, and $l(R_j)\in(\frac{1}{2^{k+1}},
\frac{1}{2^k}]$, we find that
\begin{align*}
\bigcup_{\genfrac{}{}{0pt}{}{j\in\Gamma_{i,1}}{l(R_j)
\in(\frac{1}{2^{k+1}},\frac{1}{2^k}]}}R_j=
\bigcup_{m=1}^{2n+2}\bigcup_{\genfrac{}{}{0pt}{}{j\in\Gamma_{i,1}}{l(R_j)
\in(\frac{1}{2^{k+1}},\frac{1}{2^k}],\,R_j^+(\gamma)\cap
S_m\neq\emptyset}}R_j\subset\bigcup_{m=1}^{2n+2}E_{k,m}=:E_k
\end{align*}
and
\begin{align*}
|E_k|&\leq\sum_{m=1}^{2n+2}\left|E_{k,m}\right|\leq
2[3l(R_i)]^n\frac{3-\gamma}{2^{kp}}+
2n\frac{2}{2^k}[3l(R_i)]^{n-1}(4-2\gamma)[l(R_i)]^p\\
&=\frac{2\times3^n(3-\gamma)}{2^{kp}}[l(R_i)]^n
+\frac{2^33^{n-1}(2-\gamma)n}{2^k}[l(R_i)]^{n-1+p}.
\end{align*}
This finishes the proof of the above claim.

Observe that,
for any $j\in\Gamma_{i,1}$, $l(R_j)<l(R_i)$ and hence
\begin{align*}
\Gamma_{i,1}=\bigcup_{k=k_0}^{\infty}\left\{j\in\Gamma_{i,1}:
\ l\left(R_j\right)\in\left(\frac{1}{2^{k+1}},
\frac{1}{2^k}\right]\right\}.
\end{align*}
From this, (iv) in Step 2, the above claim, and
$l(R_i)\in(\frac{1}{2^{k_0+1}},\frac{1}{2^{k_0}}]$,
we deduce that
\begin{align}\label{20240729.1539}
\sum_{j\in\Gamma_{i,1}}\left|R_j^+(\gamma)\right|&
=\sum_{k=k_0}^\infty\sum_{\{j\in\Gamma_{i,1}:\
l(R_j)\in(\frac{1}{2^{k+1}},
\frac{1}{2^k}]\}}\left|R_j^+(\gamma)\right|\\
&=\sum_{k=k_0}^\infty
\sum_{\{j\in\Gamma_{i,1}:\
l(R_j)\in(\frac{1}{2^{k+1}},
\frac{1}{2^k}]\}}\left|R_j^-(\gamma)\right|
\leq\sum_{k=k_0}^\infty|E_k|\notag\\
&\leq\sum_{k=k_0}^\infty
\left\{\frac{2\times3^n(3-\gamma)}{2^{kp}}[l(R_i)]^n
+\frac{2^33^{n-1}(2-\gamma)n}{2^k}[l(R_i)]^{n-1+p}
\right\}\notag\\
&=\frac{2^{2p+1}3^n(3-\gamma)}{2^p-1}[l(R_i)]^n
\frac{1}{2^{(k_0+1)p}}\notag\\
&\quad+2^53^{n-1}(2-\gamma)n[l(R_i)]^{n-1+p}
\frac{1}{2^{k_0+1}}\notag\\
&\leq\left[\frac{2^{2p+1}3^n(3-\gamma)}{2^p-1}
+2^53^{n-1}(2-\gamma)n\right][l(R_i)]^{n+p}\notag\\
&=\frac{2^{2p+1}3^n(3-\gamma)+2^53^{n-1}(2-\gamma)
(2^p-1)n}{2^n(1-\gamma)(2^p-1)}\left|R_i^+(\gamma)\right|\nonumber\\
&=:C_2\left|R_i^+(\gamma)\right|,\notag
\end{align}
which completes the proof of
\eqref{20240727.2217}
when $h=1$.

Now, we prove \eqref{20240727.2217}
in the case $h=2$. Let
\begin{align*}
\Omega_{k_0}:=\left\{j\in\Gamma_{i,2}:\
l(R_j)\in\left(\frac{1}{2^{k_0+1}},
\frac{1}{2^{k_0}}\right]\right\}.
\end{align*}
For any $k\in\mathbb{N}\cap[k_0+1,\infty)$, we define
$\Omega_k$ inductively by setting
\begin{align*}
\Omega_k:=\left\{j\in\Gamma_{i,2}:\
l(R_j)\in\left(\frac{1}{2^{k+1}},
\frac{1}{2^k}\right]\mbox{ and }
R_j^-(\gamma)\cap\bigcup_{d=k_0}^{k-1}\bigcup_{m\in\Omega_d}
R_m^-(\gamma)=\emptyset\right\}.
\end{align*}
Let $\Omega:=\bigcup_{k=k_0}^\infty\Omega_k$. Notice that
$\{\Omega_k\}_{k=k_0}^\infty$ are pairwise disjoint.
In addition, for any $j\in\Omega$,
we have $R_j^+(\gamma)\subset R_i^+(\gamma)$ and hence
$R_j^-(\gamma)\subset R_i$.
From these two facts and (iv) of Step 2, it follows that
\begin{align}\label{20240728.2258}
\sum_{j\in\Omega}\left|R_j^+(\gamma)\right|
&=\sum_{k=k_0}^\infty\sum_{j\in\Omega_k}
\left|R_j^+(\gamma)\right|
=\sum_{k=k_0}^\infty\sum_{j\in\Omega_k}
\left|R_j^-(\gamma)\right|\\
&=\sum_{k=k_0}^\infty
\left|\bigcup_{j\in\Omega_k}R_j^-(\gamma)\right|
=\left|\bigcup_{k=k_0}^\infty\bigcup_{j\in\Omega_k}
R_j^-(\gamma)\right|\notag\\
&\leq\left|\bigcup_{j\in\Gamma_{i,2}}R_j^-(\gamma)\right|\leq
\left|R_i^-(\gamma)\right|=\left|R_i^+(\gamma)\right|.\notag
\end{align}

Next, we show that, for any $m\in\Gamma_{i,2}\setminus\Omega$,
there exists $j\in\Omega$ such that $R_{j}^-(\gamma)\cap
R_m^-(\gamma)\neq\emptyset$, $R_m^-(\gamma)\nsubset
R_j^-(\gamma)$, and $l(R_{j})\in(l(R_m),\infty)$. Indeed,
let $k_{\mathrm{min}}\in[k_0,\infty)$ be the
smallest integer such that there exists
$j\in\Omega_{k_{\mathrm{min}}}$ satisfying $R_m^-(\gamma)\cap
R_j^-(\gamma)\neq\emptyset$. From (i) and (iv) in
Step 2, we deduce that $k_{\mathrm{min}}>k_0$ and
$R_m^-(\gamma)\nsubset R_j^-(\gamma)$. Moreover,
$l(R_{j})\in(l(R_m),\infty)$. Otherwise, by the definition
of $k_{\mathrm{min}}$ and (iv) in Step 2, we conclude that
there exists $k'\in\mathbb{N}\cap[k_0,k_{\mathrm{min}})$ such
that $m\in\Omega_{k'}$, which contradicts the assumption
that $m\in\Gamma_{i,2}\setminus\Omega$. Thus,
$l(R_{j})\in(l(R_m),\infty)$.

For any $j\in\Omega$, define
\begin{align*}
\widetilde{\Omega}_j:=\left\{m\in\Gamma_{i,2}\setminus\Omega:\
R_m^-(\gamma)\cap R_j^-(\gamma)\neq\emptyset,\
R_m^-(\gamma)\nsubset R_j^-(\gamma),
\mbox{ and }l(R_m)<l(R_j)\right\}.
\end{align*}
Let $j\in\Omega$.
Then there exists $\widetilde{k}\in\mathbb{Z}
\cap[k_0,\infty)$ such that $l(R_j)\in
(\frac{1}{2^{\widetilde{k}+1}},\frac{1}{2^{\widetilde{k}}}]$.
We claim that, for any $k\in\mathbb{Z}\cap
[\widetilde{k},\infty)$, there exists a
measurable set $\widetilde{E}_k\subset\mathbb{R}^{n+1}$ such
that
\begin{align}\label{2046}
\bigcup_{m\in\widetilde{\Omega}_j,\,l(R_m)
\in(\frac{1}{2^{k+1}},\frac{1}{2^k}]}R_m^-(\gamma)
\subset\widetilde{E}_k
\end{align}
and
\begin{align}\label{2047}
\left|\widetilde{E}_k\right|\leq\frac{2^23^n
(1-\gamma)}{2^{kp}}[l(R_j)]^n+\frac{2^23^n
(1-\gamma)n}{2^k}[l(R_j)]^{n-1+p}.
\end{align}
Indeed, fix $k\in\mathbb{Z}\cap[\widetilde{k},\infty)$. Denote
the $2n+2$ faces of $R_j^-(\gamma)$ by $\{\widetilde{S}_d\}
_{d=1}^{2n+2}$, where $\widetilde{S}_1$
and $\widetilde{S}_2$ denote, respectively, the top and the
bottom of $R_j^-(\gamma)$. For any given
$m\in\widetilde{\Omega}_j$ such that
$l(R_m)\in(\frac{1}{2^{k+1}},\frac{1}{2^k}]$, since
$R_m^-(\gamma)\nsubset R_j^-(\gamma)$ and $R_m^-(\gamma)\cap
R_j^-(\gamma)\neq\emptyset$, it follows that $R_m^-$
intersects with the boundary of $R_j^-(\gamma)$, that is,
\begin{align}\label{20240729.1507}
R_m^-(\gamma)\cap\left(\bigcup_{d=1}^{2n+2}
\widetilde{S}_d\right)\neq\emptyset.
\end{align}
For any $\widetilde{S}_d$ with $d\in\mathbb{N}\cap[1,2n+2]$,
there exists a rectangle $\widetilde{E}_{k,d}\subset
\mathbb{R}^{n+1}$ such that
\begin{align*}
\bigcup_{\genfrac{}{}{0pt}{}{m\in\widetilde{\Omega}_j}{l(R_m)
\in(\frac{1}{2^{k+1}},\frac{1}{2^k}],\,R_m^-(\gamma)\cap
S_d\neq\emptyset}}R_m^-(\gamma)\subset\widetilde{E}_{k,d}
\end{align*}
and
\begin{align*}
\left|\widetilde{E}_{k,d}\right|=
\begin{cases}
\left[l(R_j)+2l(R_m)\right]^n
2(1-\gamma)[l(R_m)]^p&\text{if }d\in\{1,2\},\\
2l(R_m)\left[l(R_j)+2l(R_m)\right]
^{n-1}\\
\quad\times\left\{(1-\gamma)[l(R_j)]^p
+2(1-\gamma)[l(R_m)]^p\right\}&\text{if }
d\in\mathbb{N}\cap[3,2n+2].
\end{cases}
\end{align*}
Applying this, \eqref{20240729.1507},
$l(R_m)<l(R_j)$, and $l(R_m)\in(\frac{1}{2^{k+1}},
\frac{1}{2^k}]$, we obtain
\begin{align*}
\bigcup_{\genfrac{}{}{0pt}{}{m\in\widetilde{\Omega}_j}{l(R_m)
\in(\frac{1}{2^{k+1}},\frac{1}{2^k}]}}R_m^-(\gamma)=
\bigcup_{d=1}^{2n+2}\bigcup_{\genfrac{}{}{0pt}{}{m\in
\widetilde{\Omega}_j}{l(R_m)
\in(\frac{1}{2^{k+1}},\frac{1}{2^k}],\,R_m^-(\gamma)\cap
S_d\neq\emptyset}}R_m^-(\gamma)\subset\bigcup_{d=1}^{2n+2}
\widetilde{E}_{k,d}=:\widetilde{E}_k
\end{align*}
and
\begin{align*}
\left|\widetilde{E}_k\right|&\leq\sum_{d=1}^{2n+2}
\left|\widetilde{E}_{k,d}\right|\leq
2\left[3l(R_j)\right]^n
\frac{2(1-\gamma)}{2^{kp}}+
2n\frac{2}{2^k}\left[3l(R_j)\right]
^{n-1}3(1-\gamma)\left[l(R_j)\right]^p\\
&=\frac{2^23^n(1-\gamma)}{2^{kp}}[l(R_j)]^n
+\frac{2^23^n(1-\gamma)n}{2^k}
[l(R_j)]^{n-1+p}.
\end{align*}
This finishes the proofs of \eqref{2046} and \eqref{2047}. From this,
the fact that
\begin{align*}
\widetilde{\Omega}_k=\bigcup_{k=\widetilde{k}}^\infty
\left\{m\in\widetilde{\Omega}_k:\
l(R_m)\in\left(\frac{1}{2^{k+1}},\frac{1}{2^k}\right]\right\},
\end{align*}
(iv) in Step 2, and $l(R_j)\in
(\frac{1}{2^{\widetilde{k}+1}},\frac{1}{2^{\widetilde{k}}}]$,
we infer that
\begin{align*}
\sum_{m\in\widetilde{\Omega}_j}\left|R_m^+(\gamma)\right|
&=\sum_{k=\widetilde{k}}^\infty
\sum_{\{m\in\widetilde{\Omega}
_j:\ l(R_m)\in(\frac{1}{2^{k+1}},\frac{1}{2^k}]\}}
\left|R_m^+(\gamma)\right|\\
&=\sum_{k=\widetilde{k}}^\infty
\sum_{\{m\in\widetilde{\Omega}
_j:\ l(R_m)\in(\frac{1}{2^{k+1}},\frac{1}{2^k}]\}}
\left|R_m^-(\gamma)\right|
\leq\sum_{k=\widetilde{k}}^\infty\left|\widetilde{E}_k\right|\\
&\leq\sum_{k=\widetilde{k}}^\infty
\left\{\frac{2^23^n(1-\gamma)}{2^{kp}}[l(R_j)]^n+\frac{2^2
3^n(1-\gamma)n}{2^k}
[l(R_j)]^{n-1+p}\right\}\\
&=\frac{2^{2p+2}3^n}{2^p-1}[l(R_j)]^n
\frac{1}{2^{(\widetilde{k}+1)p}}+2^43^nn
[l(R_j)]^{n-1+p}\frac{1}{2^{\widetilde{k}+1}}\\
&\leq\left[\frac{2^{2p+2}3^n(1-\gamma)}{2^p-1}
+2^43^n(1-\gamma)n\right][l(R_j)]^{n+p}\\
&=\frac{2^{2p+2}3^n(1-\gamma)
+2^43^n(1-\gamma)(2^p-1)n}{2^n(1-\gamma)(2^p-1)}
\left|R_j^+(\gamma)\right|\\
&=:\widetilde{C}_3\left|R_j^+(\gamma)\right|.
\end{align*}
Combining this and \eqref{20240728.2258}, we conclude that
\begin{align*}
\sum_{j\in\Gamma_{i,2}}\left|R_j^+(\gamma)\right|
&=\left(\sum_{j\in\Omega}+\sum_{m\in\Gamma_{i,2}
\setminus\Omega}\right)\left|R_j^+(\gamma)\right|
\leq\sum_{j\in\Omega}\left|R_j^+(\gamma)\right|
+\sum_{j\in\Omega}\sum_{m\in\widetilde{\Omega}_j}
\left|R_m^+(\gamma)\right|\\
&\leq\sum_{j\in\Omega}\left|R_j^+(\gamma)\right|
+\widetilde{C}_3\sum_{j\in\Omega}\left|R_j^+(\gamma)\right|
=\left(\widetilde{C}_3+1\right)\sum_{j\in\Omega}
\left|R_j^+(\gamma)\right|\\
&\leq\left(\widetilde{C}_3+1\right)
\left|R_i^+(\gamma)\right|=:C_3\left|R_i^+(\gamma)\right|,
\end{align*}
which completes the proof of
\eqref{20240727.2217} when
$h=2$. This, together with
\eqref{20240729.1539}, further implies that
\eqref{20240727.2217} holds and hence \eqref{20240727.2207}.

\emph{Step 4.} In this step, we
prove that there exists a positive constant $C_4$,
depending only on $n$ and $p$, such that, for any
$i\in\mathbb{N}\cap[1,N]$, there exists a measurable set
$F_i\subset R_i^+(\gamma)$ such that
\begin{align}\label{20240729.1551}
\frac{1}{|R_i^+(\gamma)|^{1-\beta}}\int_{F_i}|f|
\in\left(\frac{\lambda}{2},\infty\right)
\end{align}
and
\begin{align}\label{20240729.1552}
\sum_{i=1}^N\boldsymbol{1}_{F_i}\leq C_4.
\end{align}
In what follows, for any finite subset
$A\subset\mathbb{N}$, we denote the cardinality of $A$ by
$\sharp A$. For any given $i\in\mathbb{N}\cap[1,N]$,
we define $F_i$ by considering the following two cases on
$\sharp\Gamma_{i}$.

\emph{Case 1)}
$\sharp\Gamma_{i}\leq2\lceil C_1\rceil$.
In this case, let
$F_i:=R_i^+(\gamma)$.

\emph{Case 2)}
$\sharp\Gamma_{i}>2\lceil C_1\rceil$.
In this case, for any $k\in\mathbb{N}$, let
\begin{align*}
E_i^k:=R_i^+(\gamma)\cap\left\{\sum_{j\in\Gamma_i}
\boldsymbol{1}_{R_j^+(\gamma)}\in[k,\infty)\right\}.
\end{align*}
Notice that, for any $k,k'\in\mathbb{N}$
with $k<k'$, we have
\begin{align}\label{2113}
E_i^{k'}\subset E_i^k
\end{align}
and, for any $(x,t)\in R_i^+(\gamma)$,
\begin{align}\label{20240729.2142}
\sum_{k=1}^{2\lceil C_1\rceil}
\boldsymbol{1}_{E_i^k}(x,t)=\sum_{j\in\Gamma_i}
\boldsymbol{1}_{R_j^+(\gamma)}(x,t).
\end{align}
Then we define $F_i:=R_i^+(\gamma)\setminus E_i^{2\lceil
C_1\rceil}$.

We show that, for any $i\in\mathbb{N}\cap[1,N]$,
$F_i$ satisfies \eqref{20240729.1551}. We
consider the following two cases on $\sharp\Gamma_i$.

\emph{Case 1)}
$\sharp\Gamma_i\leq2\lceil C_1\rceil$.
In this case, from the
definition of $F_i=R_i^+(\gamma)$ and \eqref{20240725.2302},
it follows that
\begin{align*}
\frac{1}{|R_i^+(\gamma)|^{1-\beta}}\int_{F_i}|f|
=2\left|R^+(\gamma)\right|^\beta\fint_{R_i^+(\gamma)}|f|
\in(\lambda,\infty).
\end{align*}

\emph{Case 2)} $\sharp\Gamma_i>2\lceil C_1\rceil$.
In this case, using \eqref{2113}, $E_i^k\subset R_i^+(\gamma)$ for
any $k\in\mathbb{N}\cap[1,2\lceil C_1\rceil]$,
\eqref{20240729.2142}, and \eqref{20240727.2207}, we obtain
\begin{align*}
2\lceil C_1\rceil\int_{E_i^{2\lceil C_1\rceil}}|f|
&\leq\sharp\Gamma_i\int_{E_i^{2\lceil C_1\rceil}}|f|
\leq\sum_{k=1}^{2\lceil C_1\rceil}\int_{E_i^k}|f|\\
&=\int_{R_i^+(\gamma)}|f|\sum_{k=1}^{2\lceil C_1\rceil}
\boldsymbol{1}_{E_i^k}=\int_{R_i^+(\gamma)}|f|
\sum_{j\in\Gamma_i}\boldsymbol{1}_{R_j^+(\gamma)}\\
&\leq\sum_{j\in\Gamma_i}\int_{R_j^+(\gamma)}|f|\leq
C_1\int_{R_i^+(\gamma)}|f|
\leq2\lceil C_1\rceil\int_{R_i^+(\gamma)}|f|,
\end{align*}
which, together with \eqref{20240725.2302} and the definition of
$F_i$,
further implies that
\begin{align*}
\int_{F_i}|f|=\int_{R_i^+(\gamma)}|f|
-\int_{E_i^{2\lceil C_1\rceil}}|f|
\geq\frac{1}{2}\int_{R_i^+(\gamma)}|f|>\frac{\lambda}{2}
\left|R_i^+(\gamma)\right|^{1-\beta}.
\end{align*}
Combing the above two cases, we complete the proof of
\eqref{20240729.1551}.

Now, we turn to prove \eqref{20240729.1552}.
To this end, we first
show that, for any $(x,t)\in F_i$,
\begin{align}\label{20240729.2215}
\sum_{j\in\Gamma_i}\boldsymbol{1}_{R_j^+(\gamma)}(x,t)\leq
2\lceil C_1\rceil.
\end{align}
Indeed, if $\sharp\Gamma_i\leq2\lceil C_1\rceil$, then,
for any $(x,t)\in F_i$,
\begin{align}\label{20240729.2217}
\sum_{j\in\Gamma_i}\boldsymbol{1}_{R_j^+(\gamma)}(x,t)\leq
\sharp\Gamma_i\leq2\lceil C_1\rceil.
\end{align}
If $\sharp\Gamma_i>2\lceil C_1\rceil$, then, from
\eqref{20240729.2142}, we deduce that, for any $(x,t)\in F_i$,
\begin{align*}
\sum_{j\in\Gamma_i}\boldsymbol{1}_{R_j^+}(x,t)=
\sum_{k=1}^{2\lceil C_1\rceil}\boldsymbol{1}_{E_i^k}(x,t)
\leq2\lceil C_1\rceil,
\end{align*}
which, together with \eqref{20240729.2217},
completes the proof of
\eqref{20240729.2215}.

Then we prove that there exists a positive constant
$\widetilde{C}_4$, depending only on $n$, $p$, and $\gamma$,
such that, for any given $(x,t)\in\mathbb{R}^{n+1}$ and
$k\in\mathbb{Z}$,
\begin{align}\label{20240729.2257}
\sum_{i\in\mathbb{N}\cap[1,N],\,
l(R_i)\in(\frac{1}{2^{k+1}},
\frac{1}{2^k}]}\boldsymbol{1}_{R_i^+(\gamma)}(x,t)
\leq\widetilde{C}_4.
\end{align}
To do this, let
\begin{align*}
I_{(x,t)}^k:=\left\{i\in\mathbb{N}\cap[1,N]:\
l(R_i)\in\left(\frac{1}{2^{k+1}},\frac{1}{2^k}\right]
\mbox{ and }(x,t)\in R_i^+(\gamma)\right\}
\end{align*}
and
\begin{align*}
S_{(x,t)}:=Q\left(x,\frac{1}{2^{k-1}}\right)\times
\left(t-\frac{1}{2^{kp-1}},t+\frac{1}{2^{kp-1}}\right).
\end{align*}
Then, for any $(x,t)\in\mathbb{R}^{n+1}$,
\begin{align}\label{20240729.2234}
\bigcup_{i\in I_{(x,t)}^k}R_i^-(\gamma)\subset S_{(x,t)}.
\end{align}
Indeed, by the definition of $I_{(x,t)}^k$, we find that, for
any $i\in I_{(x,t)}^k$ and $(y,s)\in R_i^-(\gamma)$,
\begin{align*}
\|y-x\|_\infty\leq2l(R_i)\leq\frac{1}{2^{k-1}}
\ \ \mbox{and}\ \
|t-s|\leq2[l(R_i)]^p\leq\frac{1}{2^{kp-1}}.
\end{align*}
Therefore, $(y,s)\in S_{(x,t)}$ and hence \eqref{20240729.2234}
holds. From the definitions of $I_{(x,t)}^k$ and
$S_{(x,t)}$, (iv) in Step 2, and
\eqref{20240729.2234}, we infer that
\begin{align*}
\sum_{\genfrac{}{}{0pt}{}{i\in\mathbb{N}\cap[1,N]}{l(R_i)\in(\frac{1}{2^{k+1}},
\frac{1}{2^k}]}}\boldsymbol{1}_{R_i^+(\gamma)}(x,t)
&=\sum_{i\in I_{(x,t)}^k}1=\sum_{i\in I_{(x,t)}^k}
\frac{|R_i^-(\gamma)|}{2^n(1-\gamma)[l(R_i)]^{n+p}}\\
&\leq\frac{2^{(k+1)(n+p)-n}}{1-\gamma}\left|\bigcup_{i\in
I_{(x,t)}^k}R_i^-(\gamma)\right|
\leq\frac{2^{(k+1)(n+p)-n}}{1-\gamma}\left|S_{(x,t)}\right|\\
&=\frac{2^{(k+1)(n+p)-n}}{1-\gamma}
\left(\frac{2}{2^{k-1}}\right)^n
\frac{2}{2^{kp-1}}=\frac{2^{2n+p+2}}{1-\gamma}
=:\widetilde{C}_4,
\end{align*}
which completes the proof of \eqref{20240729.2257}.

Finally, we show \eqref{20240729.1552}.
For any $(x,t)\in\bigcup_{i=1}^N F_i$, there exists
$i_0\in\mathbb{N}\cap[1,N]$ such that $(x,t)\in F_{i_0}$ and
$R_{i_0}$ is the largest parabolic rectangle in the following
sense: for any $i\in\mathbb{N}\cap[1,N]$ satisfying $(x,t)\in
F_i$, $|R_i|\leq |R_{i_0}|$.
Let $k_0\in\mathbb{Z}$ be
such that $l(R_{i_0})\in(\frac{1}{2^{k_0+1}},
\frac{1}{2^{k_0}}]$. Applying the construction of $i_0$,
$F_i\subset R_i^+(\gamma)$ for any
$i\in\mathbb{N}\cap[1,N]$, the definition of $\Gamma_{i_0}$,
\eqref{20240729.2215}, and \eqref{20240729.2257}, we obtain
\begin{align*}
\sum_{i=1}^N\boldsymbol{1}_{F_i}
&=\left\{\sum_{\genfrac{}{}{0pt}{}{i\in\mathbb{N}\cap[1,N]}{l(R_i)\in
(\frac{1}{2^{k_0+1}},\frac{1}{2^{k_0}}]}}+
\sum_{\genfrac{}{}{0pt}{}{i\in\mathbb{N}\cap[1,N]}{l(R_i)\in
(\frac{1}{2^{k+1}},\frac{1}{2^k}]}}\right\}
\boldsymbol{1}_{F_i}\\
&\leq\left\{\sum_{\genfrac{}{}{0pt}{}{i\in\mathbb{N}\cap[1,N]}{l(R_i)\in
(\frac{1}{2^{k_0+1}},\frac{1}{2^{k_0}}]}}
+\sum_{i\in\Gamma_{i_0}}\right\}\boldsymbol{1}
_{R_i^+(\gamma)}
\leq\widetilde{C}_4+2\lceil C_1\rceil=:C_4.
\end{align*}
This finishes the proof of \eqref{20240729.1552} and hence Step 4.

\emph{Step 5.} In this step, we prove \eqref{20240725.2122} by
considering the following two cases on $r$.

\emph{Case 1)} $r\in(1,\infty)$. In this case, from
\eqref{20240725.2143}, \eqref{20240729.1551}, the H\"older
inequality, $F_i\subset R_i^+(\gamma)$,
$\alpha=\frac{\gamma}{5^p}$,
$P_i=5R_i$ for any $i\in\mathbb{N}\cap[1,N]$,
$\frac{1}{r}-\frac{1}{q}=\beta$, Definition
\ref{parabolic Muckenhoupt class}(i), and
\eqref{20240729.1552}, we deduce that
\begin{align}\label{20240729.1557}
\left(u^q\right)(K)&\leq2\sum_{i=1}^{N}\left(u^q\right)
\left(P_i^-(\alpha)\right)\leq\frac{2^{q+1}}{\lambda^q}
\sum_{i=1}^{N}\left(u^q\right)\left(P_i^-(\alpha)\right)
\left[\frac{1}{|R_i^+(\gamma)|^{1-\beta}}
\int_{F_i}|f|\right]^q\\
&\leq\frac{2^{q+1}}{\lambda^q}\sum_{i=1}^{N}
\frac{1}{|R_i^+(\gamma)|^{q-\beta q}}
\int_{P_i^-(\alpha)}u^q
\left(\int_{F_i}v^{-r'}\right)^{\frac{q}{r'}}
\left(\int_{F_i}|f|^rv^r\right)^{\frac{q}{r}}\notag\\
&\leq\frac{2^{q+1}5^{(n+p)(1+\frac{q}{r'})}
(1-\alpha)^{1+\frac{q}{r'}}}{(1-\gamma)^{1+\frac{q}{r'}}
\lambda^q}\sum_{i=1}^{N}
\fint_{P_i^-(\alpha)}u^q
\left[\fint_{P_i^+(\alpha)}v^{-r'}\right]
^{\frac{q}{r'}}\left(\int_{F_i}|f|^rv^r\right)
^{\frac{q}{r}}\notag\\
&\leq\frac{2^{q+1}5^{(n+p)(1+\frac{q}{r'})}(1-\alpha)
^{1+\frac{q}{r'}}[u,v]_{TA_{r,q}^+(\alpha)}}{(1-\gamma)
^{1+\frac{q}{r'}}\lambda^q}
\sum_{i=1}^{N}\left(\int_{F_i}|f|^rv^r\right)
^{\frac{q}{r}}\notag\\
&\leq\frac{2^{q+1}5^{(n+p)(1+\frac{q}{r'})}
(5^p-\gamma)^{1+\frac{q}{r'}}
[u,v]_{TA_{r,q}^+(\alpha)}C_4^{\frac{q}{r}}}{(1-\gamma)
^{1+\frac{q}{r'}}5^{p(1+\frac{q}{r'})}\lambda^q}
\left(\int_{\mathbb{R}^{n+1}}|f|^rv^r\right)^{\frac{q}{r}}.
\notag
\end{align}
This, together with Corollary \ref{independence of time lag},
finishes the proof of \eqref{20240725.2122}
in this case.

\emph{Case 2)} $r=1$. In this case, by
\eqref{20240725.2143}, \eqref{20240729.1551},
$\alpha=\frac{\gamma}{5^p}$, $P_i=5R_i$ for any
$i\in\mathbb{N}\cap[1,N]$,
$\frac{1}{r}-\frac{1}{q}=\beta$, Definition
\ref{parabolic Muckenhoupt class}(ii), and
\eqref{20240729.1552}, we conclude that
\begin{align*}
\left(u^q\right)(K)&\leq2\sum_{i=1}^{N}\left(u^q\right)
\left(P_i^-(\alpha)\right)\leq
\frac{2^{q+1}}{\lambda^q}\sum_{i=1}^{N}\left(u^q\right)
\left(P_i^-(\alpha)\right)\left[\frac{1}{|R_i^+(\gamma)|
^{1-\beta}}\int_{F_i}|f|\right]^q\\
&\leq\frac{2^{q+1}5^{n+p}(1-\alpha)}{(1-\gamma)\lambda^q}
\sum_{i=1}^{N}\left(u^q\right)_{P_i^-(\alpha)}
\left(\int_{F_i}|f|\right)^q\\
&\leq\frac{2^{q+1}5^{n+p}(1-\alpha)
[u,v]_{TA_{1,q}^+(\alpha)}}
{(1-\gamma)\lambda^q}\sum_{i=1}^{N}
\left(\int_{F_i}|f|v\right)^q\\
&\leq\frac{2^{q+1}5^{n+p}(5^p-\gamma)
[u,v]_{TA_{1,q}^+(\alpha)}C_4^q}
{5^p\lambda^q}\left(\int_{\mathbb{R}^{n+1}}|f|v\right)^q.
\end{align*}
This, together with Corollary \ref{independence of time lag},
finishes the proof of \eqref{20240725.2122} in this case.
Combining this and \eqref{20240729.1557}, we obtain
\eqref{20240725.2122}. This finishes the proof of the
necessity and hence Theorem \ref{weak type inequality}.
\end{proof}

\begin{remark}
Theorem~\ref{weak type inequality} when both $r=q$
and $u=v$
coincides with \cite[Theorem~6.1]{km(am-2024)}.
\end{remark}

The proof of Theorem \ref{weak type inequality} also works
for the case $\gamma=0$ and we present this result as follows.

\begin{theorem}\label{weak type inequality gamma=0}
Let $\beta\in[0,1)$, $1\leq r\leq q<\infty$,
$\beta=\frac{1}{r}-\frac{1}{q}$,
and $(u,v)$ be a pair of nonnegative functions on
$\mathbb{R}^{n+1}$. Then $(u,v)\in TA_{r,q}^+(0)$ if and only
$M^{0+}_\beta$ is bounded from $L^r(\mathbb{R}^{n+1},v^r)$ to
$L^{q,\infty}(\mathbb{R}^{n+1},u^q)$.
\end{theorem}

Notice that the condition $u\in
A_\infty^+(\gamma)$ is only used to show the necessity of
Theorem~\ref{weak type inequality}. Moreover, using
\cite[Lemma 2.1 (1)]{mhy(fm-2023)} and \cite[Lemma
7.4]{ks(na-2016)}, we find that, for any $1\leq r\leq
q<\infty$ and any nonnegative function $\omega$ on
$\mathbb{R}^{n+1}$, $\omega\in A_{r,q}^+(\gamma)$ implies that
$\omega\in A_\infty^+(\gamma)$. These, combined with Theorem
\ref{weak type inequality} with $u=v$, the self-improving
property of $A_{r,q}^+(\gamma)$ with $1<r\leq q<\infty$ (see
\cite[Lemma 2.2]{mhy(fm-2023)}), and the Stein--Weiss
interpolation theorem (see, for instance,
\cite[Corollary 5.5.2]{bl(intepolation-1976)}), further
implies the following corollary. We omit the details.

\begin{corollary}\label{weighted inequality uncentered}
Let $\gamma\in(0,1)$, $\beta\in[0,1)$, $1<r\leq q<\infty$,
$\beta=\frac{1}{r}-\frac{1}{q}$, and $\omega$ be a
weight on $\mathbb{R}^{n+1}$. Then $\omega\in
A_{r,q}^+(\gamma)$ if and only if $M^{\gamma+}_\beta$
is bounded from $L^r(\mathbb{R}^{n+1},\omega^r)$ to
$L^q(\mathbb{R}^{n+1},\omega^q)$.
\end{corollary}

\begin{remark}
Corollary~\ref{weighted inequality uncentered} when both $r=q$
and $u=v$
coincides with \cite[Theorem 6.2]{km(am-2024)}.
\end{remark}

The following definition of centered parabolic fractional
maximal operators can be found in
\cite[p.\,187]{mhy(fm-2023)}.

\begin{definition}\label{centered parabolic maximal operator}
Let $\gamma,\beta\in[0,1)$. For any $f\in
L_{\mathrm{loc}}^{1}(\mathbb{R}^{n+1})$,
the \emph{centered forward in time parabolic fractional
maximal function $\mathcal{M}^{\gamma+}_\beta(f)$
with time lag} and the \emph{centered back
in time parabolic fractional maximal function
$\mathcal{M}^{\gamma-}_\beta(f)$ with time lag} of
$f$ are defined, respectively, by setting,
for any $(x,t)\in\mathbb{R}^{n+1}$,
\begin{align*}
\mathcal{M}^{\gamma+}_\beta(f)(x,t):=\sup_{L\in(0,\infty)}
\left|R(x,t,L)^+(\gamma)\right|^\beta
\fint_{R(x,t,L)^+(\gamma)}|f|
\end{align*}
and
\begin{align*}
\mathcal{M}^{\gamma-}_\beta(f)(x,t):=\sup_{L\in(0,\infty)}
\left|R(x,t,L)^-(\gamma)\right|^\beta
\fint_{R(x,t,L)^-(\gamma)}|f|.
\end{align*}
\end{definition}

At the end of this section,
we prove the following weak-type parabolic two-weighted
boundedness of centered
parabolic fractional maximal operators
with time lag.

\begin{theorem}\label{weak type inequality centered}
Let $\gamma,\beta\in[0,1)$, $1<r\leq q<\infty$,
$\beta=\frac{1}{r}-\frac{1}{q}$, and $(u,v)$ be a pair of
nonnegative functions on $\mathbb{R}^{n+1}$.
\begin{enumerate}
\item[\rm(i)] If $\gamma\in(0,1)$ and $u\in
A_\infty^+(\gamma)$, then $(u,v)\in TA_{r,q}^+(\gamma)$ if
and only if $\mathcal{M}^{\gamma+}_\beta$ is bounded from
$L^r(\mathbb{R}^{n+1},v^r)$ to
$L^{q,\infty}(\mathbb{R}^{n+1},u^q)$.

\item[\rm(ii)] $(u,v)\in TA_{r,q}^+(0)$ if and only
if $\mathcal{M}^{0+}_\beta$ is bounded from
$L^r(\mathbb{R}^{n+1},v^r)$ to
$L^{q,\infty}(\mathbb{R}^{n+1},u^q)$.
\end{enumerate}
\end{theorem}

\begin{proof}
We first show the sufficiency of both (i) and (ii).
Let $\gamma\in[0,1)$.
Assume $\mathcal{M}^{\gamma+}_\beta$ is bounded from
$L^r(\mathbb{R}^{n+1},v^r)$ to
$L^{q,\infty}(\mathbb{R}^{n+1},u^q)$. Fix
$R:=R(x,t,L)\in\mathcal{R}_p^{n+1}$ with
$(x,t)\in\mathbb{R}^{n+1}$ and $L\in(0,\infty)$.
Define $S^+(\gamma):=R^-(\gamma)+(1-\gamma)L^p+2^p\gamma L^p$
and, for any $\epsilon\in(0,\infty)$, let
$f_\epsilon:=(v+\epsilon)^{-r'}\boldsymbol{1}_{S^+(\gamma)}$.
Then, for any $(y,s)\in R^-(\gamma)$, $S^+(\gamma)\subset
R(y,s,2L)^+(\gamma)$ and
$|R(y,s,2L)^+(\gamma)|=2^{n+p}|S^+(\gamma)|$. From this and
Definition \ref{parabolic maximal operators}, it follows
that, for any $\lambda\in(0,2^{(n+p)(\beta-1)}
|S^+(\gamma)|^\beta(f_\epsilon)_{S^+(\gamma)})$ and
$(y,s)\in R^-(\gamma)$,
\begin{align*}
\lambda<2^{(n+p)(\beta-1)}\left|S^+(\gamma)\right|^\beta
\left(f_\epsilon\right)_{S^+(\gamma)}
\leq\left|R(y,s,2L)^+(\gamma)\right|^\beta
\fint_{R(y,s,2L)^+(\gamma)}f_\epsilon\leq
M^{\gamma+}_\beta(f_\epsilon)(y,s),
\end{align*}
which further implies that
$R^-(\gamma)\subset\{M^{\gamma+}_\beta(f_\epsilon)
>\lambda\}$.
Combining this and the assumption that
$\mathcal{M}^{\gamma+}_\beta$ is bounded from
$L^r(\mathbb{R}^{n+1},v^r)$ to
$L^{q,\infty}(\mathbb{R}^{n+1},u^q)$, we obtain
\begin{align*}
\int_{R^-(\gamma)}u^q&=(u^q)(R^-(\gamma))\leq
(u^q)\left(\left\{M^{\gamma+}_\beta(f_\epsilon)
>\lambda\right\}\right)\\
&\leq\frac{C}{\lambda^q}\left(\int_{\mathbb{R}^{n+1}}
f_\epsilon^rv^r\right)^{\frac{q}{r}}
=\frac{C}{\lambda^q}\left[\int
_{S^+(\gamma)}(v+\epsilon)^{-r'r}v^r\right]^{\frac{q}{r}}.
\end{align*}
Letting $\lambda\to2^{(n+p)(\beta-1)}
|S^+(\gamma)|^\beta(f_\epsilon)_{S^+(\gamma)}$ and
$\epsilon\to0$, dividing both sides of the above inequality by
$|R^+(\gamma)|$, using $\frac{1}{r}-\frac{1}{q}=\beta$, and
taking the supremum over all $R\in\mathcal{R}_p^{n+1}$, we
find that
\begin{align*}
\sup_{R\in\mathcal{R}_p^{n+1}}
\fint_{R^-(\gamma)}u^q\left[\fint_{S^+(\gamma)}v^{-r'}
\right]^{\frac{q}{r'}}\leq\frac{C}{2^{(n+p)(\beta-1)q}}.
\end{align*}
Thus, if $\gamma=0$, this and Definition \ref{parabolic Muckenhoupt class}(i)
imply $(u,v)\in TA_{r,q}^+(0)$ and, if
$\gamma\in(0,1)$, then applying this and Theorem
\ref{independence of time lag 2}
[here we use the assumption that $u\in
A_\infty^+(\gamma)$] we also conclude $(u,v)\in
TA_{r,q}^+(\gamma)$, which then completes
the proof of the sufficiency of both (i) and (ii).

Next, we prove the necessity of (i). Let $f\in
L_{\mathrm{loc}}^1(\mathbb{R}^{n+1})$. From Corollary
\ref{independence of time lag}(i) and
Theorem \ref{weak type inequality}, we infer that, to show
that $\mathcal{M}^{\gamma+}_\beta$ is bounded from
$L^r(\mathbb{R}^{n+1},v^r)$ to
$L^{q,\infty}(\mathbb{R}^{n+1},u^q)$, it remains to prove that
there exists a positive constant $K$, depending only on $n$,
$p$, $\gamma$, and $\beta$, such that, for any
$(x,t)\in\mathbb{R}^{n+1}$,
\begin{align}\label{20240819.1730}
\mathcal{M}^{\gamma+}_\beta(f)(x,t)\leq
KM^{\frac{\gamma}{4}+}_\beta(f)(x,t).
\end{align}
Indeed, fix $L\in(0,\infty)$ and let $P:=R(x,t+\frac{\gamma
L^p}{2},L)$. Then we are easy to show that
$(x,t)\in P^-(\frac{\gamma}{4})$, $R(x,t,L)^+(\gamma)\subset
P^+(\frac{\gamma}{4})$, and $|P^+(\frac{\gamma}{4})|=
\frac{1-\frac{\gamma}{4}}{1-\gamma}|R(x,t,L)^+(\gamma)|$. This
further implies that
\begin{align*}
&\left|R(x,t,L)^+(\gamma)\right|^\beta
\fint_{R(x,t,L)^+(\gamma)}|f|\\
&\quad\leq
\left(\frac{1-\gamma}{1-\frac{\gamma}{4}}\right)^{\beta-1}
\left|P^+\left(\frac{\gamma}{4}\right)\right|^\beta
\fint_{P^+(\frac{\gamma}{4})}|f|\leq
\left(\frac{1-\gamma}{1-\frac{\gamma}{4}}\right)^{\beta-1}
M^{\frac{\gamma}{4}+}_\beta(f)(x,t).
\end{align*}
Taking the supremum over all $L\in(0,\infty)$, we obtain
\eqref{20240819.1730}, which completes the proof of
necessity of (i).

The proof of the necessity of (ii)
is a slight modification of that of (i) with Theorem
\ref{weak type inequality} replaced by
Theorem
\ref{weak type inequality gamma=0}; we omit the details.
This finishes the
proof of Theorem \ref{weak type inequality centered}.
\end{proof}

\begin{remark}
In the case $u=v$, Theorem
\ref{weak type inequality centered} when $r=q$ coincides with
\cite[Lemma 4.2]{ks(apde-2016)} and  when $r<q$ coincides with
\cite[Theorem 1.1]{mhy(fm-2023)}.
\end{remark}

The following is a simple corollary of Theorem
\ref{weak type inequality centered} with $u=v$, the
self-improving property of $A_{r,q}^+(\gamma)$ with $1<r\leq
q<\infty$, and the Stein--Weiss interpolation theorem,
which has been obtained in \cite[Theorem 5.4]{ks(apde-2016)}
when $\beta=0$ and in \cite[Theorem 1.3]{mhy(fm-2023)}
when $\beta\in(0,1)$.

\begin{corollary}\label{centered strong type inequality}
Let $\gamma,\rho\in(0,1)$, $\beta\in[0,1)$, $1<r\leq
q<\infty$, $\beta=\frac{1}{r}-\frac{1}{q}$, and $\omega$ be a
weight on $\mathbb{R}^{n+1}$. Then $\omega\in
A_{r,q}^+(\gamma)$ if and only if
$\mathcal{M}^{\rho+}_\beta$
is bounded from $L^r(\mathbb{R}^{n+1},\omega^r)$ to
$L^q(\mathbb{R}^{n+1},\omega^q)$.
\end{corollary}

\section{Characterizations of Weighted Boundedness of \\
Parabolic Fractional Integrals with Time Lag}
\label{section5}

In this section, we introduce the
parabolic forward in time and back in time fractional integral
operators with time lag. Then we give the weak-type
parabolic two-weighted inequality and the strong-type
parabolic weighted inequality for such operators.
Recall that the \emph{parabolic distance} $d_p$ on
$\mathbb{R}^{n+1}\times\mathbb{R}^{n+1}$
is defined by setting, for any $(x,t),(y,s)
\in\mathbb{R}^{n+1}$,
\begin{align*}
d_p((x,t),(y,s)):=\max\left\{\|x-y\|_\infty,\,
|t-s|^{\frac{1}{p}}\right\}.
\end{align*}
We then introduce the definitions of parabolic forward in time
fractional integrals with time lag
and parabolic back in time fractional
integrals with time lag as follows.

\begin{definition}\label{parabolic Riesz potentials}
Let $\gamma,\beta\in[0,1)$. For any $f\in
L_{\mathrm{loc}}^{1}(\mathbb{R}^{n+1})$,
the \emph{parabolic forward in time fractional integral
$I^{\gamma+}_\beta(f)$ with time lag} and the \emph{parabolic
back in time fractional integral $I^{\gamma-}_\beta(f)$ with
time lag} of $f$ are defined, respectively, by setting,
for any $(x,t)\in\mathbb{R}^{n+1}$,
\begin{align*}
I^{\gamma+}_\beta(f)(x,t):=\int_{\bigcup_{L\in(0,\infty)}
R(\mathbf{0},0,L)^+(\gamma)}\frac{f(x-y,t-s)}
{[d_p((y,s),(\mathbf{0},0))]^{(n+p)(1-\beta)}}\,dy\,ds
\end{align*}
and
\begin{align*}
I^{\gamma-}_\beta(f)(x,t):=\int_{\bigcup_{L\in(0,\infty)}
R(\mathbf{0},0,L)^-(\gamma)}\frac{f(x-y,t-s)}
{[d_p((y,s),(\mathbf{0},0))]^{(n+p)(1-\beta)}}\,dy\,ds.
\end{align*}
\end{definition}

\begin{remark}\label{range of integral}
The integral domain $\bigcup_{L\in(0,\infty)}
R(\mathbf{0},0,L)^+(\gamma)$ in Definition
\ref{parabolic Riesz potentials} is the area above the surface
$t=\gamma\|x\|_\infty^p$. In
particular, when $n=1$, the integral domain is
exactly the area above the parabola $t=\gamma|x|^p$;
see the following figure.
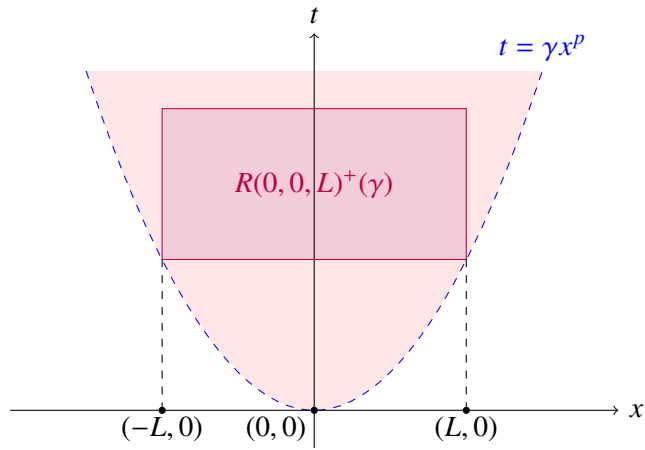
\begin{center}
\begin{tikzpicture}
\draw[dashed,color=blue,fill=red!10] (-3,4.5) parabola bend
(0,0) (3,4.5) node [above] {$t=\gamma x^p$};
\draw[purple,fill=purple!20] (-2,2) rectangle (2,4);
\node[purple] at (0,3) {$R(0,0,L)^+(\gamma)$};
\draw[dashed,-] (2,0)--(2,2);
\draw[fill] (2,0) circle (0.04);
\node at (2,-0.25) {$(L,0)$};
\draw[dashed,-] (-2,0)--(-2,2);
\draw[fill] (-2,0) circle (0.04);
\node at (-2,-0.25) {$(-L,0)$};
\draw[->] (-4,0)--(4,0) node [right] {$x$};
\draw[->] (0,-0.5)--(0,5) node [above] {$t$};
\draw[fill] (0,0) circle (0.04);
\node at (-0.5,-0.25) {$(0,0)$};
\end{tikzpicture}
\end{center}
\end{remark}

We have the following relation between
the parabolic fractional integral with time lag and the
centered parabolic fractional maximal operator with time lag.

\begin{lemma}\label{pointwise control}
Let $\gamma,\beta\in[0,1)$. Then there exists
a positive constant $C$, depending only on $n$, $\gamma$,
and $\beta$, such that, for any
$f\in L_{\mathrm{loc}}^{1}(\mathbb{R}^{n+1})$
and $(x,t)\in\mathbb{R}^{n+1}$,
\begin{align}\label{20240812.1543}
\mathcal{M}^{\gamma+}_\beta(f)(x,t)\leq
CI^{\gamma+}_\beta(|f|)(x,t).
\end{align}
\end{lemma}

\begin{proof}
Let $(x,t)\in\mathbb{R}^{n+1}$ and $L_0\in(0,\infty)$.
Simply denote $R(x,t,L_0)$ by $R$. Observe that, for any
$(y,s)\in R^+(\gamma)$, $d_p((x,t),(y,s))\leq L$.
From this, the fact that
$|R^+(\gamma)|=2^n(1-\gamma)L^{n+p}$,
and Definition \ref{parabolic Riesz potentials}, we deduce that
\begin{align*}
\left|R^+(\gamma)\right|^{\beta}
\fint_{R^+(\gamma)}|f|&\leq2^{n(\beta-1)}(1-\gamma)^{\beta-1}
\int_{R^+(\gamma)}\frac{|f(y,s)|}{[d_p((x,t),(y,s))]
^{(n+p)(1-\beta)}}\,dy\,ds\\
&\leq2^{n(\beta-1)}(1-\gamma)^{\beta-1}
I^{\gamma+}_\beta(|f|)(x,t).
\end{align*}
Taking the supremum over all $L_0\in(0,\infty)$, we conclude that
\eqref{20240812.1543} with
$C:=2^{n(\beta-1)}(1-\gamma)^{\beta-1}$ holds,
which completes the proof of Lemma
\ref{pointwise control}.
\end{proof}

The following lemma is a Welland type inequality in the
parabolic setting. For the Welland
inequality in the elliptic setting, see
\cite[(2.3)]{w(pams-1975)} and \cite[(1.2)]{as(tams-1988)}.

\begin{lemma}\label{Welland inequality}
Let $\gamma,\beta\in(0,1)$ and $\epsilon\in(0,\min\{\beta,\,
1-\beta\})$. Then there exists a positive constant
$C$, depending only on $n$, $p$, $\gamma$, $\beta$, and $\epsilon$,
such that, for any $f\in
L_{\mathrm{loc}}^{1}(\mathbb{R}^{n+1})$
and $(x,t)\in\mathbb{R}^{n+1}$,
\begin{align}\label{20240812.1627}
I^{\gamma+}_\beta(|f|)(x,t)\leq
C\left[\mathcal{M}^{\gamma^2+}_{\beta-\epsilon}(f)(x,t)
\mathcal{M}^{\gamma^2+}_{\beta+\epsilon}(f)(x,t)\right]^
\frac{1}{2}.
\end{align}
\end{lemma}

\begin{proof}
Let $f\in L_{\mathrm{loc}}^{1}(\mathbb{R}^{n+1})$,
$(x,t)\in\mathbb{R}^{n+1}$, and
$\Omega_{(x,t)}^{\gamma+}:=\bigcup_{L\in(0,\infty)}R(x,t,L)
^+(\gamma)$. Without loss of generality, we may assume that
$\mathcal{M}^{\gamma^2+}_{\beta-\epsilon}(f)(x,t)
\mathcal{M}^{\gamma^2+}_{\beta+\epsilon}(f)(x,t)\in(0,\infty)$;
otherwise, \eqref{20240812.1627} holds automatically.
Fix $\delta\in(0,\infty)$ (which will be
determined later) and define
\begin{align*}
\Omega_1:=
\left\{(y,s)\in\Omega_{(x,t)}^{\gamma+}:\
d_p((x,t),(y,s))\in[0,\delta)\right\}
\end{align*}
and $\Omega_2:=\Omega_{(x,t)}^{\gamma+}\setminus\Omega_1$. Then
\begin{align}\label{20240812.1634}
I^{\gamma+}_\beta(|f|)(x,t)=
\int_{\Omega_1}\frac{|f(y,s)|\,dy\,ds}{[d_p((x,t),(y,s))]
^{(n+p)(1-\beta)}}+
\int_{\Omega_2}\ldots
=:I(x,t)+II(x,t).
\end{align}

We first estimate $I(x,t)$. Let
$\eta:=(\frac{1}{\gamma})^{\frac{1}{p}}$ and, for any
$i\in\mathbb{N}$,
\begin{align*}
V_i^+:=Q\left(x,\eta^{-i+2}\delta\right)\times
\left(t+\gamma^i\delta^p,t+\gamma^{i-1}\delta^p\right).
\end{align*}
Then the following four statements hold obviously:
\begin{enumerate}
\item[(i)] For any $i\in\mathbb{N}$ and $(y,s)\in V_i^+$,
$d_p((x,t),(y,s))\geq(\gamma^i\delta^p)^{\frac{1}{p}}
=\eta^{-i}\delta$.

\item[(ii)] For any $i,j\in\mathbb{N}$ with $i\neq j$,
$V_i^+\cap V_j^+=\emptyset$.

\item[(iii)] For any $i\in\mathbb{N}$, the bottom of $V_i^+$
contains the top of $V_{i+1}^+$.

\item[(iv)] $\Omega_1\subset\bigcup_{i\in\mathbb{N}}V_i^+$.
\end{enumerate}
For any $i\in\mathbb{N}$, let
$R_i:=R(x,t,\eta^{-i+2}\delta)$.
Then
$V_i^+\subset R_i^+(\gamma^2)$
and
$|V_i^+|=\frac{\gamma}{1+\gamma}
|R_i^+(\gamma^2)|$.
From this, the monotone convergence theorem, (i) through (iv),
and the fact that
\begin{align*}
\left|R_i^+\left(\gamma^2\right)\right|
=2^n\left(1-\gamma^2\right)\eta^{2(n+p)}
\left(\eta^{-i}\delta\right)^{n+p},
\end{align*}
it follows that
\begin{align}\label{20240812.2223}
I(x,t)&\leq\sum_{i\in\mathbb{N}}\int_{V_i^+}
\frac{|f(y,s)|}{[d_p((x,t),(y,s))]^{(n+p)(1-\beta)}}\,dy\,ds\\
&\leq\sum_{i\in\mathbb{N}}\left(\eta^{-i}\delta\right)
^{(n+p)(\beta-1)}\int_{V_i^+}|f|\notag\\
&\leq2^{n(1+\epsilon-\beta)}\left(1-\gamma^2\right)
^{1+\epsilon-\beta}\eta^{2(n+p)(1+\epsilon-\beta)}\notag\\
&\quad\times\sum_{i\in\mathbb{N}}\left(\eta^{-i}\delta\right)
^{(n+p)\epsilon}\left|R_i^+(\gamma^2)\right|
^{\beta-\epsilon-1}\int_{R_i^+(\gamma^2)}|f|\notag\\
&=\frac{2^{n(1+\epsilon-\beta)}(1-\gamma^2)
^{1+\epsilon-\beta}\eta^{2(n+p)(1+\epsilon-\beta)}}
{\eta^{(n+p)\epsilon}-1}\delta^{(n+p)\epsilon}
M_{\beta-\epsilon}^{\gamma^2+}(f)(x,t).\notag
\end{align}

Now, we estimate $II(x,t)$ similarly.
For any
$i\in\mathbb{N}$, let
\begin{align*}
U_j^+:=Q\left(x,\eta^j\delta\right)\times
\left(t+\gamma^{-j+2}\delta^p,t+\gamma^{-j+1}\delta^p\right).
\end{align*}
Then the following other four statements hold obviously:
\begin{enumerate}
\item[(v)] For any $j\in\mathbb{N}$ and $(y,s)\in U_j^+$,
$d_p((x,t),(y,s))\geq(\gamma^{-j+2}\delta^p)^{\frac{1}{p}}
=\eta^{j-2}\delta$.

\item[(vi)] For any $j,k\in\mathbb{N}$ with $j\neq k$,
$U_j^+\cap U_k^+=\emptyset$.

\item[(vii)] For any $j\in\mathbb{N}$, the bottom of
$U_{j+1}^+$ contains the top of $U_j^+$.

\item[(viii)] $\Omega_2\subset\bigcup_{j\in\mathbb{N}}U_j^+$.
\end{enumerate}
For any $j\in\mathbb{N}$, let $P_i:=R(x,t,\eta^j\delta)$.
Then
$U_j^+\subset P_j^+(\gamma^2)$
and
$|U_j^+|=\frac{\gamma}{1+\gamma}
|P_j^+(\gamma^2)|$.
From this, the monotone convergence theorem,
(v) through (viii), and the fact that
\begin{align*}
\left|P_j^+\left(\gamma^2\right)\right|
=2^n\left(1-\gamma^2\right)\eta^{2(n+p)}
\left(\eta^{j-2}\delta\right)^{n+p},
\end{align*}
we deduce that
\begin{align*}
II(x,t)&\leq\sum_{j\in\mathbb{N}}\int_{U_j^+}
\frac{|f(y,s)|}{[d_p((x,t),(y,s))]^{(n+p)(1-\beta)}}\,dy\,ds\\
&\leq\sum_{j\in\mathbb{N}}\left(\eta^{j-2}\delta\right)
^{(n+p)(\beta-1)}\int_{U_j^+}|f|\notag\\
&\leq2^{n(1-\epsilon-\beta)}\left(1-\gamma^2\right)
^{1-\epsilon-\beta}\eta^{2(n+p)(1-\epsilon-\beta)}\notag\\
&\quad\times\sum_{j\in\mathbb{N}}\left(\eta^{j-2}\delta\right)
^{-(n+p)\epsilon}\left|P_j^+(\gamma^2)\right|
^{\beta+\epsilon-1}\int_{P_j^+(\gamma^2)}|f|\notag\\
&=\frac{2^{n(1-\epsilon-\beta)}(1-\gamma^2)
^{1-\epsilon-\beta}\eta^{2(n+p)(1-\epsilon-\beta)}}
{\eta^{(n+p)\epsilon}-1}\delta^{-(n+p)\epsilon}
M_{\beta+\epsilon}^{\gamma^2+}(f)(x,t).
\end{align*}
Combining
this, \eqref{20240812.1634}, and \eqref{20240812.2223} and
choosing
\begin{align*}
\delta:=\left[\frac{\mathcal{M}_{\beta+\epsilon}
^{\gamma^2+}(f)(x,t)}{\mathcal{M}_{\beta-\epsilon}
^{\gamma^2+}(f)(x,t)}\right]^\frac{1}{2(n+p)\epsilon},
\end{align*}
we then obtain \eqref{20240812.1627} and hence finish the proof
of Lemma \ref{Welland inequality}.
\end{proof}

Next, we are ready to present the first main result of this
section.

\begin{theorem}\label{weak type Riesz potential}
Let $\gamma,\beta\in(0,1)$, $1\leq r<q<\infty$,
$\beta=\frac{1}{r}-\frac{1}{q}$, and $(u,v)$ be a pair of
nonnegative functions on $\mathbb{R}^{n+1}$. If $u\in
A_\infty^+(\gamma)$, then $(u,v)\in TA_{r,q}^+(\gamma)$ if
and only if there exists a positive constant $C$
such that, for any $f\in L^r(\mathbb{R}^{n+1},v^r)$,
\begin{align}\label{20240818.2228}
\left\|I^{\gamma+}_\beta(f)\right\|
_{L^{q,\infty}(\mathbb{R}^{n+1},u^q)}\leq C
\|f\|_{L^r(\mathbb{R}^{n+1},v^r)}.
\end{align}
\end{theorem}

\begin{proof}
We first prove the sufficiency. Assume that
\eqref{20240818.2228} holds. By this, Lemma
\ref{pointwise control}, and Theorem
\ref{weak type inequality centered}(i),
we conclude that $(u,v)\in TA_{r,q}^+(\gamma)$, which
completes the proof of the sufficiency.

Then we show the necessity.
Assume that $(u,v)\in TA_{r,q}^+(\gamma)$.
Let $f\in
L_{\mathrm{loc}}^{1}(\mathbb{R}^{n+1})$ and
$\lambda\in(0,\infty)$. From
Corollary \ref{self-improving property},
we infer that there exists
$\delta_0\in(0,\infty)$ such that $(u,v)\in
A_{r,q+\delta}^+(\gamma)$ for any $\delta\in(0,\delta_0)$.
Choose $\epsilon\in(0,\min\{\beta,\,1-\beta\})$ such that
\begin{align}\label{20240819.1607}
\frac{1}{\frac{1}{r}-(\beta+\epsilon)}-q\in(0,\delta_0)
\end{align}
and let $q_1,q_2\in(1,\infty)$ satisfy
\begin{align}\label{20240814.1542}
\frac{1}{q_1}=\frac{1}{r}-(\beta-\epsilon)\ \ \mbox{and}\ \
\frac{1}{q_2}=\frac{1}{r}-(\beta+\epsilon).
\end{align}
Then $1<q_1<q<q_2<q+\delta_0<\infty$.
Applying Lemma \ref{Welland inequality} and
\eqref{20240819.1730}, we find that there
exists a positive constant $C$, depending only on $n$, $p$,
$\gamma$, $\beta$, and $\epsilon$, such that
\begin{align*}
\left\{\left|I^{\gamma+}_\beta(f)\right|>\lambda\right\}
&\subset\left\{I^{\gamma+}_\beta(|f|)>\lambda\right\}\subset
\left\{C\left[M^{\frac{\gamma^2}{4}+}_{\beta-\epsilon}(f)
M^{\frac{\gamma^2}{4}+}_{\beta+\epsilon}(f)\right]^
\frac{1}{2}>\lambda\right\}\\
&\subset\left\{M^{\frac{\gamma^2}{4}+}_{\beta-\epsilon}(f)
>\frac{\lambda}{C}\right\}\cup
\left\{M^{\frac{\gamma^2}{4}+}_{\beta+\epsilon}(f)
>\frac{\lambda}{C}\right\}
\end{align*}
and hence
\begin{align}\label{20240819.1619}
\left(u^q\right)\left(\left\{\left|I^{\gamma+}_\beta(f)\right|
>\lambda\right\}\right)\leq
\left(u^q\right)\left(\left\{M^{\frac{\gamma^2}{4}+}
_{\beta-\epsilon}(f)>\frac{\lambda}{C}\right\}\right)+
\left(u^q\right)\left(\left\{M^{\frac{\gamma^2}{4}+}
_{\beta+\epsilon}(f)>\frac{\lambda}{C}\right\}\right).
\end{align}

On the one hand, according to the proven conclusion that
$q_1<q$, Proposition \ref{nested property}(ii), and Corollary
\ref{independence of time lag}(i), we obtain
$(u,v)\in TA_{r,q_1}^+(\frac{\gamma^2}{4})$, which, together
with \eqref{20240814.1542} and Theorem
\ref{weak type inequality}, further implies that
there exists a positive constant $K_1$, depending only on $n$,
$p$, $r$, $q$, $\beta$, $\epsilon$, and
$[u,v]_{TA_{r,q_1}^+(\frac{\gamma^2}{4})}$, such that
\begin{align}\label{20240819.1632}
\left(u^q\right)\left(\left\{M^{\frac{\gamma^2}{4}+}
_{\beta-\epsilon}(f)>\frac{\lambda}{C}\right\}\right)
\leq\frac{K_1}{\lambda^q}\left(\int_{\mathbb{R}^{n+1}}
|f|^rv^r\right)^{\frac{q}{r}}.
\end{align}
On the other hand, from the proven conclusions that
$q<q_2<q+\delta_0$ and $(u,v)\in TA_{r,q+\delta}^+(\gamma)$
for any $\delta\in(0,\delta_0)$ and from Corollary
\ref{independence of time lag}(i), it follows that $(u,v)\in
TA_{r,q_2}^+(\frac{\gamma^2}{4})$, which, combined with
\eqref{20240814.1542} and Theorem
\ref{weak type inequality}, further implies that
there exists a positive constant $K_2$, depending only on $n$,
$p$, $r$, $q$, $\beta$, $\epsilon$, and
$[u,v]_{TA_{r,q_1}^+(\frac{\gamma^2}{4})}$, such that
\begin{align*}
\left(u^q\right)\left(\left\{M^{\frac{\gamma^2}{4}+}
_{\beta+\epsilon}(f)>\frac{\lambda}{C}\right\}\right)
\leq\frac{K_2}{\lambda^q}\left(\int_{\mathbb{R}^{n+1}}
|f|^rv^r\right)^{\frac{q}{r}}.
\end{align*}
By this, \eqref{20240819.1619}, and
\eqref{20240819.1632}, we find that
\begin{align*}
\left(u^q\right)\left(\left\{\left|I^{\gamma+}_\beta(f)\right|
>\lambda\right\}\right)\leq\frac{K_1+K_2}{\lambda^q}
\left(\int_{\mathbb{R}^{n+1}}|f|^rv^r\right)^{\frac{q}{r}}.
\end{align*}
Taking the supremum over all $\lambda\in(0,\infty)$,
we then conclude that \eqref{20240818.2228} holds.
This finishes the proof of the necessity and hence Theorem
\ref{weak type Riesz potential}.
\end{proof}

For any given $q\in[1,\infty)$, let $A_{1,q}^+(\gamma)$ be the
set of all nonnegative functions $\omega$ on
$\mathbb{R}^{n+1}$ such that
$[\omega]_{A_{1,q}^+(\gamma)}:=
[\omega^{\frac1q},\omega^{\frac1q}]
_{TA_{1,q}^+(\gamma)}<\infty$.
The following is a direct consequence  of
Theorem \ref{weak type Riesz potential} when both $u=v$ and $r=1$;
we omit the details.

\begin{corollary}\label{weak type Riesz potential 2}
Let $\gamma\in(0,1)$, $q\in(1,\infty)$,
$\beta:=1-\frac{1}{q}$, and $\omega$ be a weight on
$\mathbb{R}^{n+1}$. Then $\omega\in
A_{1,q}^+(\gamma)$ if and only $I_\beta^{\gamma+}$ is bounded
from $L^1(\mathbb{R}^{n+1},\omega)$ to
$L^{q,\infty}(\mathbb{R}^{n+1},\omega^q)$.
\end{corollary}

The second main result of this section is the following strong
type parabolic weighted inequalities for parabolic
fractional integrals with time lag.

\begin{theorem}\label{weighted inequality of Riesz potential}
Let $\gamma,\beta\in(0,1)$, $1<r<q<\infty$,
$\beta=\frac{1}{r}-\frac{1}{q}$, and $\omega$ be a
weight on $\mathbb{R}^{n+1}$. Then $\omega\in
A_{r,q}^+(\gamma)$ if and only if $I^{\gamma+}_\beta$ is
bounded from $L^r(\mathbb{R}^{n+1},\omega^r)$ to
$L^q(\mathbb{R}^{n+1},\omega^q)$.
\end{theorem}

\begin{proof}
To show the sufficiency, assume that
$I^{\gamma+}_\beta$ is bounded from
$L^r(\mathbb{R}^{n+1},\omega^r)$ to
$L^q(\mathbb{R}^{n+1},\omega^q)$. Then, by Lemma
\ref{pointwise control}, we find that
$\mathcal{M}^{\gamma+}_\beta$ is bounded from
$L^r(\mathbb{R}^{n+1},\omega^r)$ to $L^q(\mathbb{R}^{n+1},\omega^q)$,
which, together with Corollary~\ref{centered strong type inequality},
implies that
$\omega\in A_{r,q}^+(\gamma)$. This finishes the proof of
the sufficiency.

Now, we prove the necessity. Assume that $\omega\in
A_{r,q}^+(\gamma)$. Using Corollary
\ref{self-improving property}, we conclude that there
exists $\delta_0\in(0,\infty)$ such that $\omega\in
A_{r,q+\delta}^+(\gamma)$ for any $\delta\in(0,\delta_0)$.
Fix $\epsilon\in(0,\min\{\beta,1-\beta\})$ such that
\eqref{20240819.1607} holds and let $q_1,q_2\in(1,\infty)$
satisfy \eqref{20240814.1542}.
Then $1<q_1<q<q_2<q+\delta_0<\infty$ and
\begin{align*}
\frac{1}{2q_1}+\frac{1}{2q_2}=\frac{1}{q}.
\end{align*}
From this, Lemma \ref{Welland inequality}, and the H\"older
inequality, we infer that, for any $f\in
L^r(\mathbb{R}^{n+1},\omega^r)$,
\begin{align}\label{20240814.1554}
&\left[\int_{\mathbb{R}^{n+1}}\left|I^{\gamma+}_\beta(f)\right|
^q\omega^q\right]^{\frac{1}{q}}\\
&\quad\lesssim\left\{\int_{\mathbb{R}^{n+1}}
\left[\mathcal{M}_{\beta-\epsilon}^{\gamma^2+}(f)\right]
^{\frac{q}{2}}\omega^{\frac{q}{2}}\left[\mathcal{M}
_{\beta+\epsilon}^{\gamma^2+}(f)\right]^{\frac{q}{2}}
\omega^{\frac{q}{2}}\right\}^{\frac{1}{q}}\nonumber\\
&\quad\leq\left\{\int_{\mathbb{R}^{n+1}}\left[\mathcal{M}
_{\beta-\epsilon}^{\gamma^2+}(f)\right]^{q_1}
\omega^{q_1}\right\}^{\frac{1}{2q_1}}
\left\{\int_{\mathbb{R}^{n+1}}\left[\mathcal{M}
_{\beta+\epsilon}^{\gamma^2+}(f)\right]^{q_2}
\omega^{q_2}\right\}^{\frac{1}{2q_2}}
=:\mathrm{I}\times\mathrm{II}.\notag
\end{align}
To estimate $\mathrm{I}$, by the fact that
$q_1<q$ and Proposition \ref{nested property}(ii), we find
that $\omega\in A_{r,q_1}^+(\gamma)$. From this,
\eqref{20240814.1542}, and Corollary
\ref{centered strong type inequality}, we deduce that
\begin{align}\label{20240815.2218}
\mathrm{I}\lesssim\left(\int_{\mathbb{R}^{n+1}}|f|^r
\omega^r\right)^{\frac{1}{2r}}.
\end{align}
To estimate $\mathrm{II}$, by the proven conclusions that
$q_2-q\in(0,\delta_0)$ and $\omega\in A_{r,q+\delta}^+(\gamma)$
for any $\delta\in(0,\delta_0)$, we obtain $\omega\in
A_{r,q_2}^+(\gamma)$. From this, \eqref{20240814.1542}, and
Corollary \ref{centered strong type inequality}, it follows that
\begin{align*}
\mathrm{II}\lesssim\left(\int_{\mathbb{R}^{n+1}}|f|^r
\omega^r\right)^{\frac{1}{2r}}.
\end{align*}
Combining this, \eqref{20240815.2218}, and
\eqref{20240814.1554}, we find that $I^{\gamma+}_{\beta}$
is bounded from $L^r(\mathbb{R}^{n+1},\omega^r)$ to
$L^q(\mathbb{R}^{n+1},\omega^q)$,
which completes the proof of the necessity and hence
Theorem \ref{weighted inequality of Riesz potential}.
\end{proof}

The following theorem is a direct consequence of Corollary
\ref{independence of time lag}(i) and
Theorems~\ref{weak type inequality} and~\ref{weak type Riesz potential};
we omit the details.

\begin{theorem}\label{20241004.1030}
Let $\beta\in(0,1)$ and $1\leq r<q<\infty$ satisfy
$\beta=\frac{1}{r}-\frac{1}{q}$.
Let $\{\gamma_i\}_{i=1}^3$ be a sequence of $(0,1)$ and
$(u,v)$ be a pair of
nonnegative functions on $\mathbb{R}^{n+1}$.
Assume that $u\in A_\infty^+(\gamma)$.
Then
the following statements are mutually equivalent.
\begin{enumerate}
\item[\rm(i)]
$(u,v)\in TA_{r,q}^+(\gamma_1)$.
\item[\rm(ii)]
${M}^{\gamma_2+}_\beta$ is bounded from
$L^r(\mathbb{R}^{n+1},v^r)$
to $L^{q,\infty}(\mathbb{R}^{n+1},u^q)$.
\item[\rm(iii)]
$I^{\gamma_3+}_\beta$ is bounded from
$L^r(\mathbb{R}^{n+1},v^r)$
to $L^{q,\infty}(\mathbb{R}^{n+1},u^q)$.
\end{enumerate}
\end{theorem}

\begin{remark}
When $r\in(1,\infty)$, if we replace
the uncentered fractional maximal operator
${M}^{\gamma_2+}_\beta$
in Theorem~\ref{20241004.1030}(ii)
by the centered fractional maximal operator
$\mathcal{M}^{\gamma_2+}_\beta$,
then Theorem~\ref{20241004.1030} still holds.
\end{remark}

The following theorem is a direct consequence of
Theorems \ref{weak type inequality},
\ref{weak type inequality centered}(i),
\ref{weak type Riesz potential}, and
\ref{weighted inequality of Riesz potential}
and Corollaries \ref{independence of time lag}(i),
\ref{weighted inequality uncentered}, and
\ref{centered strong type inequality};
we omit the details.

\begin{theorem}\label{20241004.1125}
Let $\beta\in(0,1)$ and $1<r<q<\infty$ satisfy
$\beta=\frac{1}{r}-\frac{1}{q}$.
Let $\{\gamma_i\}_{i=1}^7$ be a sequence of $(0,1)$ and
$\omega$ be a weight on $\mathbb{R}^{n+1}$. Then
the following statements are mutually equivalent.
\begin{enumerate}
\item[\rm(i)]
$\omega\in
A_{r,q}^+(\gamma_1)$.
\item[\rm(ii)]
$M^{\gamma_2+}_\beta$ is
bounded from $L^r(\mathbb{R}^{n+1},\omega^r)$ to
$L^{q,\infty}(\mathbb{R}^{n+1},\omega^q)$.
\item[\rm(iii)]
$M^{\gamma_3+}_\beta$ is
bounded from $L^r(\mathbb{R}^{n+1},\omega^r)$ to
$L^{q}(\mathbb{R}^{n+1},\omega^q)$.
\item[\rm(iv)]
$\mathcal{M}^{\gamma_4+}_\beta$ is
bounded from $L^r(\mathbb{R}^{n+1},\omega^r)$ to
$L^{q,\infty}(\mathbb{R}^{n+1},\omega^q)$.
\item[\rm(v)]
$\mathcal{M}^{\gamma_5+}_\beta$ is
bounded from $L^r(\mathbb{R}^{n+1},\omega^r)$ to
$L^{q}(\mathbb{R}^{n+1},\omega^q)$.
\item [\rm(vi)]
$I^{\gamma_6+}_\beta$ is
bounded from $L^r(\mathbb{R}^{n+1},\omega^r)$ to
$L^{q,\infty}(\mathbb{R}^{n+1},\omega^q)$.
\item[\rm(vii)]
$I^{\gamma_7+}_\beta$ is
bounded from $L^r(\mathbb{R}^{n+1},\omega^r)$ to
$L^q(\mathbb{R}^{n+1},\omega^q)$.
\end{enumerate}
\end{theorem}

\section{Applications to Parabolic Weighted Sobolev Embeddings}
\label{section6}

In this section, we establish the weighted
boundedness of the parabolic Riesz potentials and the
parabolic Bessel potentials in \cite{ar(sm-2001),ar(rjm-2002),
g(ajm-1977)} for some special parabolic Muckenhoupt
weights. Then we apply the results to show the corresponding
parabolic Sobolev embeddings.

Let $p\in[2,\infty)$ and $\beta\in(0,n+p)$.
For any $(x,t)\in\mathbb{R}^{n+1}$, let
\begin{align*}
h_\beta(x,t):=t^{\frac{\beta-n-p}{p(p-1)}}e^{-\frac{p-1}{p}
(\frac{|x|^p}{pt})^{\frac{1}{p-1}}}
\boldsymbol{1}_{(0,\infty)}(t).
\end{align*}
As noted in \cite{kmy(ma-2023)}, $h_p$ is a solution of the
doubly nonlinear parabolic partial differential equation
\eqref{20240903.1403} in $\mathbb{R}^{n+1}_+$.
In particular, if $p=\beta=2$, then $h_2$ is the
fundamental solution of the heat equation
$\frac{\partial u}{\partial t}-\Delta u=0$ in
$\mathbb{R}^{n+1}_+$. Notice that, for any given
$\gamma\in(0,1)$ and for any
$(y,s)\in\bigcup_{L\in(0,\infty)}R(\mathbf{0},0,L)^+(\gamma)$,
\begin{align}\label{20240814.2216}
\left|h_\beta(y,s)\right|\sim\frac{1}{[d_p((y,s),(\mathbf{0},0))]
^{(n+p)(1-\widetilde{\beta})}}
\end{align}
with the positive equivalence constants depending only
on $n$, $p$, $\gamma$, and $\beta$, where
\begin{align*}
\widetilde{\beta}:=1-\frac{n+p-\beta}{(p-1)(n+p)}\in(0,1).
\end{align*}
Let $\gamma\in[0,1)$. The \emph{parabolic
Riesz potential $\mathcal{I}^{\gamma+}_\beta$ with time lag}
is defined by setting, for any $f\in
L_{\mathrm{loc}}^1(\mathbb{R}^{n+1})$ and
$(x,t)\in\mathbb{R}^{n+1}$,
\begin{align*}
\mathcal{I}^{\gamma+}_\beta(f)(x,t):=
\int_{\bigcup_{L\in(0,\infty)}
R(\mathbf{0},0,L)^+(\gamma)}f(x-y,t-s)h_\beta(y,s)\,dy\,ds.
\end{align*}
According to \eqref{20240814.2216} and Theorem
\ref{weighted inequality of Riesz potential}, we easily
obtain the following proposition.
\begin{proposition}\label{parabolic Riesz potential 1977 modified}
Let $p\in[2,\infty)$, $\gamma\in(0,1)$, $\beta\in(0,n+p)$,
$1<r<q<\infty$ with
\begin{align}\label{20240929.2116}
\frac{1}{r}-\frac{1}{q}=1-\frac{n+p-\beta}{(p-1)(n+p)},
\end{align}
and $\omega$ be a weight on
$\mathbb{R}^{n+1}$. Then $\omega\in A_{r,q}^+(\gamma)$ if and
only if $\mathcal{I}^{\gamma+}_\beta$ is bounded from
$L^r(\mathbb{R}^{n+1},\omega^r)$ to
$L^q(\mathbb{R}^{n+1},\omega^q)$.
\end{proposition}

For any measurable functions $f,g$ on
$\mathbb{R}^{n+1}$, the \emph{convolution $f*g$} of $f$ and $g$
is defined by setting, for any $(x,t)\in\mathbb{R}^{n+1}$,
\begin{align*}
(f*g)(x,t):=\int_{\mathbb{R}^{n+1}}f(x-y,t-s)g(y,s)\,dy\,ds.
\end{align*}
Recall that, for any $\beta\in(0,n+p)$ and $f\in
L_{\mathrm{loc}}^1(\mathbb{R}^{n+1})$,  the parabolic Riesz
potential $h_\beta*f$ of $f$ coincides with
$\mathcal{I}_\beta^{0+}(f)$. It was proved in
\cite{g(ajm-1977)} (see also \cite{ar(sm-2001), ar(rjm-2002)})
that, for any $1\leq r<q<\infty$ satisfying
\eqref{20240929.2116}, $\mathcal{I}_\beta^{0+}$ is bounded
from $L^r(\mathbb{R}^{n+1})$ to $L^q(\mathbb{R}^{n+1})$
if $r\in(1,\infty)$ and is bounded from
$L^r(\mathbb{R}^{n+1})$ to $L^{q,\infty}(\mathbb{R}^{n+1})$
if $r=1$.

Next, we consider the parabolic weighted boundedness of the
parabolic Riesz potential operator
$\mathcal{I}_\beta^{0+}$. Observe that, for any
$\gamma\in(0,1)$, $\beta\in(0,n+p)$, and $f\in
L_{\mathrm{loc}}^1(\mathbb{R}^{n+1})$,
$\mathcal{I}_\beta^{0+}(|f|)\geq
\mathcal{I}_\beta^{\gamma+}(|f|)$, which, together with
Proposition \ref{parabolic Riesz potential 1977 modified},
further implies the following proposition.

\begin{proposition}\label{parabolic Riesz potential 1977 modified 2}
Let $p\in[2,\infty)$, $\gamma\in(0,1)$, $\beta\in(0,n+p)$,
$1<r<q<\infty$ satisfy \eqref{20240929.2116}, and $\omega$ be
a weight on
$\mathbb{R}^{n+1}$. If the parabolic Riesz potential
operator $\mathcal{I}_\beta^{0+}$ is bounded
from $L^r(\mathbb{R}^{n+1},\omega^r)$ to
$L^q(\mathbb{R}^{n+1},\omega^q)$, then $\omega\in
A_{r,q}^+(\gamma)$.
\end{proposition}

\begin{remark}\label{20240929,2144}
Let $p$, $\gamma$, $\beta$, $r$, $q$, and $\omega$ be the same
as in Proposition \ref{parabolic Riesz potential 1977 modified 2}.
An interesting question is whether or not the converse of
Proposition \ref{parabolic Riesz potential 1977 modified 2}
holds, that is, whether or not $\omega\in A_{r,q}^+(\gamma)$
implies that $\mathcal{I}_\beta^{0+}$ is bounded
from $L^r(\mathbb{R}^{n+1},\omega^r)$ to
$L^q(\mathbb{R}^{n+1},\omega^q)$.

\end{remark}

Let $1<r\leq q<\infty$. Recall that, for any given
$E\subset\mathbb{R}$ the \emph{one-sided off-diagonal
Muckenhoupt class $A_{r,q}^+(E)$} is defined to be the set of
all nonnegative locally integrable functions $\omega$ on $E$
such that
\begin{align*}
[\omega]_{A_{r,q}^+(E)}:=\sup_{\genfrac{}{}{0pt}{}{x\in
\mathbb{R},\,h\in(0,\infty)}{[x-h,x+h]\in E}}
\frac{1}{h}\int_{x-h}^x\left[\omega(y)\right]^q\,dy
\left\{\frac{1}{h}\int_x^{x+h}\left[\omega(y)\right]
^{-r'}\,dy\right\}^{\frac{q}{r'}}<\infty;
\end{align*}
see, for instance, \cite[(1.5)]{as(tams-1988)} when
$E:=(0,\infty)$. It can be easily verified that, if $\omega\in
A_{r,q}^+(\mathbb{R})$, then, for any $\gamma\in(0,1)$,
\begin{align}\label{20240930.1032}
\sup_{x\in\mathbb{R},\,h\in(0,\infty)}
\frac{1}{1-\gamma h}
\int_{x-h}^{x-\gamma h}\left[\omega(y)\right]^q\,dy
\left\{\frac{1}{1-\gamma h}\int_{x+\gamma h}^{x+h}
\left[\omega(y)\right]^{-r'}\,dy\right\}
^{\frac{q}{r'}}\lesssim[\omega]_{A_{r,q}^+(\mathbb{R})}
\end{align}
with the implicit positive constant depending only on
$\gamma$, $r$, and $q$. Note that the
{off-diagonal Muckenhoupt class $A_{r,q}(\mathbb{R}^n)$} is
defined in \eqref{20241003.1040}. We provide a partial
answer to the question in Remark \ref{20240929,2144} as
follows.

\begin{theorem}\label{parabolic Riesz potantial weighted boundedness Gopala Rao}
Let $p\in[2,\infty)$, $\beta\in(0,n+p)$, $1<r<q<\infty$ with
\begin{align*}
\frac{1}{r}-\frac{1}{q}=1-\frac{n+p-\beta}{(p-1)(n+p)}
=:\widetilde{\beta},
\end{align*}
and $\omega$ be a weight on $\mathbb{R}^{n+1}$.
Let $u\in A_{r,q}(\mathbb{R}^n)$, $v\in
A_{r,q}^+(\mathbb{R})$, and, for any
$(x,t)\in\mathbb{R}^{n+1}$, $\omega(x,t)=u(x)v(t)$.
Then the parabolic Riesz potential $\mathcal{I}_\beta^{0+}$ is
bounded from $L^r(\mathbb{R}^{n+1},\omega^r)$ to
$L^q(\mathbb{R}^{n+1},\omega^q)$.
\end{theorem}

\begin{proof}
Fix $\gamma\in(0,1)$ and $f\in L^r(\mathbb{R}^{n+1},
\omega^r)$. From the self-improving property of
$A_{r,q}(\mathbb{R}^n)$ (see, for instance, \cite[Lemma
3.4.2]{lushanzhen-book}), we deduce that there exists
$\delta_0\in(0,\infty)$, depending only on $n$, $r$, $q$, and
$[u]_{A_{r,q}(\mathbb{R}^n)}$, such that, for
any $\delta\in(0,\delta_0)$, $u\in
A_{r,q+\delta}(\mathbb{R}^n)$. Choose
$\epsilon\in(0,\min\{\widetilde{\beta},\,
1-\widetilde{\beta}\})$ such that
\begin{align*}
\frac{1}{\frac{1}{r}-(\widetilde{\beta}+\epsilon)}-q
\in(0,\delta_0)
\end{align*}
and let $q_1,q_2\in(0,\infty)$ satisfy
\begin{align*}
\frac{1}{q_1}=\frac{1}{r}-(\widetilde{\beta}-\epsilon)\ \
\mbox{and}\ \
\frac{1}{q_2}=\frac{1}{r}-(\widetilde{\beta}+\epsilon).
\end{align*}
Then $1<q_1<q<q_2<q+\delta_0<\infty$. Therefore,
$u\in A_{r,q_2}(\mathbb{R}^n)$. On the other hand, by the
H\"older inequality, we conclude that $u\in
A_{r,q_1}(\mathbb{R}^n)$. Combining these and
\eqref{20240930.1032}, we find that, for any
$\rho\in(0,\gamma]$,
\begin{align}\label{20240923.1121}
\omega\in A_{r,q_1}^+(\rho)\cap
A_{r,q}^+(\rho)\cap A_{r,q_2}^+(\rho).
\end{align}
For any $j\in\mathbb{Z}_+$, define
\begin{align*}
\Omega_j:=\bigcup_{L\in(0,\infty)}
R(\mathbf{0},0,L)^+\left(\frac{\gamma}{2^j}\right).
\end{align*}
From Remark \ref{range of integral}, it follows that, for any
$j\in\mathbb{N}$ and $(y,s)\in\Omega_j\setminus\Omega_{j-1}$,
\begin{align*}
\frac{\gamma}{n^{\frac{p}{2}}2^j}|y|^p\leq\frac{\gamma}{2^j}
\|y\|_\infty^p<s\leq\frac{\gamma}{2^{j-1}}\|y\|_\infty^p\leq
\frac{\gamma}{2^{j-1}}|y|^p
\end{align*}
and hence
\begin{align*}
h_\beta(y,s)&=s^{\frac{\beta-n-p}{p(p-1)}}
e^{-\frac{p-1}{p}(\frac{|y|^p}{ps})^{\frac{1}{p-1}}}\\
&\leq\left[\left(\frac{2^j}{\gamma}\right)
^{\frac{1}{p}}+1\right]^{\frac{n+p-\beta}{p-1}}
\frac{1}{[d_p((y,s),(\mathbf{0},0))]^{\frac{n+p-\beta}{p-1}}}
e^{-\frac{p-1}{p}(\frac{2^{j-1}}{p\gamma})^{\frac{1}{p-1}}},
\end{align*}
which, together with the monotone convergence theorem and
\eqref{20240814.2216}, further implies that, for any
$(x,t)\in\mathbb{R}^{n+1}$,
\begin{align}\label{20240923.2105}
\mathcal{I}_\beta^{0+}(|f|)(x,t)&=
\mathcal{I}_\beta^{\gamma+}(|f|)(x,t)+\sum_{j\in\mathbb{N}}
\int_{\Omega_j\setminus\Omega_{j-1}}|f(y,s)|h_\beta(x-y,t-s)
\,dy\,ds\\
&\leq\left[\left(\frac{1}{\gamma}\right)^{\frac{1}{p}}+1\right]
^{\frac{n+p-\beta}{p-1}}I_{\widetilde{\beta}}^{\gamma+}
(|f|)(x,t)\notag\\
&\quad+\sum_{j\in\mathbb{N}}\left[\left(\frac{2^j}{\gamma}
\right)^{\frac{1}{p}}+1\right]^{\frac{n+p-\beta}{p-1}}
e^{-\frac{p-1}{p}(\frac{2^j}{2p\gamma})^{\frac{1}{p-1}}}
I_{\widetilde{\beta}}^{\frac{\gamma}{2^j}+}(|f|)(x,t).\notag
\end{align}
Combining this, \eqref{20240923.1121}, and Theorem
\ref{weighted inequality of Riesz potential}, we conclude
that, to show that $\mathcal{I}_\beta^{0+}$ is
bounded from $L^r(\mathbb{R}^{n+1},\omega^r)$ to
$L^q(\mathbb{R}^{n+1},\omega^q)$, it suffices to prove that
there exists a positive constant $C$, independent of
$f$, such that
\begin{align}\label{20240918.2142}
\sum_{j\in\mathbb{N}}\left[\left(\frac{2^j}{\gamma}
\right)^{\frac{1}{p}}+1\right]^{\frac{n+p-\beta}{p-1}}
e^{-\frac{p-1}{p}(\frac{2^j}{2p\gamma})^{\frac{1}{p-1}}}
\left\|I_{\widetilde{\beta}}^{\frac{\gamma}{2^j}+}(|f|)\right\|
_{L^q(\mathbb{R}^{n+1},\omega^q)}\leq
C\|f\|_{L^r(\mathbb{R}^{n+1},\omega^r)}.
\end{align}

To show \eqref{20240918.2142}, fix $j\in\mathbb{N}$. From an
argument similar to that used in the proof of Lemma
\ref{Welland inequality} with $\gamma$ therein replaced by
$\frac{\gamma}{2^j}$, we infer that, for any
$(x,t)\in\mathbb{R}^{n+1}$,
\begin{align*}
I_{\widetilde{\beta}}^{\frac{\gamma}{2^j}+}(|f|)(x,t)\leq
\frac{2^{n(1+\epsilon-\widetilde{\beta})+1}}
{(\frac{\gamma}{2^j})^{\frac{2(n+p)(1+\epsilon-
\widetilde{\beta})}{p}}[(\frac{\gamma}{2^j})
^{-\frac{(n+p)\epsilon}{p}}-1]}
\mathcal{M}_{\widetilde{\beta}-\epsilon}
^{(\frac{\gamma}{2^j})^2+}(f)(x,t)
\mathcal{M}_{\widetilde{\beta}+\epsilon}
^{(\frac{\gamma}{2^j})^2+}(f)(x,t),
\end{align*}
which, together with the H\"older inequality and an argument
similar to that used in the proof of \eqref{20240819.1730},
further implies that
\begin{align}\label{20240919.2110}
\left\|I_{\widetilde{\beta}}^{\frac{\gamma}{2^j}+}(|f|)
\right\|_{L^q(\mathbb{R}^{n+1},\omega^q)}&\leq
2^{n(1+\epsilon-\widetilde{\beta})+1}\frac{(\frac{2^j}{\gamma})
^\frac{2(n+p)(1+\epsilon-\widetilde{\beta})}{p}}
{(\frac{2^j}{\gamma})^{\frac{(n+p)\epsilon}{p}}-1}
\left\|\mathcal{M}_{\widetilde{\beta}-\epsilon}
^{(\frac{\gamma}{2^j})^2+}(f)
\mathcal{M}_{\widetilde{\beta}+\epsilon}
^{(\frac{\gamma}{2^j})^2+}(f)\right\|
_{L^q(\mathbb{R}^{n+1},\omega^q)}\\
&\leq2^{n(1+\epsilon-\widetilde{\beta})+1}
\frac{(\frac{2^j}{\gamma})
^\frac{2(n+p)(1+\epsilon-\widetilde{\beta})}{p}}
{(\frac{2^j}{\gamma})^{\frac{(n+p)\epsilon}{p}}-1}\notag\\
&\quad\times\left\|\mathcal{M}_{\widetilde{\beta}-\epsilon}
^{(\frac{\gamma}{2^j})^2+}(f)\right\|
_{L^{q_1}(\mathbb{R}^{n+1},\omega^{q_1})}^\frac{1}{2}
\left\|\mathcal{M}_{\widetilde{\beta}+\epsilon}
^{(\frac{\gamma}{2^j})^2+}(f)\right\|
_{L^{q_2}(\mathbb{R}^{n+1},\omega^{q_2})}^\frac{1}{2}\notag\\
&\leq2^{n(1+\epsilon-\widetilde{\beta})+1}
\left[\frac{1-\frac{(\frac{\gamma}{2^j})^2}{4}}
{1-(\frac{\gamma}{2^j})^2}\right]
^{[1-(\widetilde{\beta}-\epsilon)]+
[1-(\widetilde{\beta}+\epsilon)]}\frac{(\frac{2^j}{\gamma})
^\frac{2(n+p)(1+\epsilon-\widetilde{\beta})}{p}}
{(\frac{2^j}{\gamma})^{\frac{(n+p)\epsilon}{p}}-1}\notag\\
&\quad\times\left\|M_{\widetilde{\beta}-\epsilon}
^{\frac{(\frac{\gamma}{2^j})^2}{4}+}(f)\right\|
_{L^{q_1}(\mathbb{R}^{n+1},\omega^{q_1})}^\frac{1}{2}
\left\|M_{\widetilde{\beta}+\epsilon}
^{\frac{(\frac{\gamma}{2^j})^2}{4}+}(f)\right\|
_{L^{q_2}(\mathbb{R}^{n+1},\omega^{q_2})}^\frac{1}{2}\notag\\
&\leq2^{n(1+\epsilon-\widetilde{\beta})+1}
\left[\frac{4-\gamma^2}{4-4\gamma^2}\right]^2
\frac{(\frac{2^j}{\gamma})
^\frac{2(n+p)(1+\epsilon-\widetilde{\beta})}{p}}
{(\frac{2^j}{\gamma})^{\frac{(n+p)\epsilon}{p}}-1}\notag\\
&\quad\times\left\|M_{\widetilde{\beta}-\epsilon}
^{\frac{(\frac{\gamma}{2^j})^2}{4}+}(f)\right\|
_{L^{q_1}(\mathbb{R}^{n+1},\omega^{q_1})}^\frac{1}{2}
\left\|M_{\widetilde{\beta}+\epsilon}
^{\frac{(\frac{\gamma}{2^j})^2}{4}+}(f)\right\|
_{L^{q_2}(\mathbb{R}^{n+1},\omega^{q_2})}^\frac{1}{2}.\notag
\end{align}
Applying \eqref{20240923.1121} and an argument similar to that
used in the proof of Corollary
\ref{weighted inequality uncentered}, we find
that there exists a positive constant $C_1$, depending only on
$n$, $p$, $\gamma$, $r$, $q$, $[u]_{A_{r,q}(\mathbb{R}^n)}$,
and $[v]_{A_{r,q}^+(\mathbb{R})}$, such that
\begin{align*}
\max\left\{\left\|M_{\widetilde{\beta}-\epsilon}
^{\frac{(\frac{\gamma}{2^j})^2}{4}+}(f)\right\|
_{L^{q_1}(\mathbb{R}^{n+1},\omega^{q_1})},\
\left\|M_{\widetilde{\beta}+\epsilon}
^{\frac{(\frac{\gamma}{2^j})^2}{4}+}(f)\right\|
_{L^{q_2}(\mathbb{R}^{n+1},\omega^{q_2})}\right\}
\leq C_1\|f\|_{L^r(\mathbb{R}^{n+1},\omega^r)}.
\end{align*}
Combining this and \eqref{20240919.2110}, we obtain
\begin{align*}
\left\|I_{\widetilde{\beta}}^{\frac{\gamma}{2^j}+}(|f|)\right\|
_{L^q(\mathbb{R}^{n+1},\omega^q)}\lesssim
\frac{(\frac{2^j}{\gamma})
^\frac{2(n+p)(1+\epsilon-\widetilde{\beta})}{p}}
{(\frac{2^j}{\gamma})^{\frac{(n+p)\epsilon}{p}}-1}
\|f\|_{L^r(\mathbb{R}^{n+1},\omega^r)},
\end{align*}
where the implicit positive constant is independent of
$f$ and $j$. From this and \eqref{20240919.2110}, we deduce
that
\begin{align*}
&\sum_{j\in\mathbb{N}}\left[\left(\frac{2^j}{\gamma}
\right)^{\frac{1}{p}}+1\right]^{\frac{n+p-\beta}{p-1}}
e^{-\frac{p-1}{p}(\frac{2^j}{2p\gamma})^{\frac{1}{p-1}}}
\left\|I_{\widetilde{\beta}}^{\frac{\gamma}{2^j}+}(|f|)\right\|
_{L^q(\mathbb{R}^{n+1},\omega^q)}\\
&\quad\lesssim\sum_{j\in\mathbb{N}}
\left[\left(\frac{2^j}{\gamma}\right)^{\frac{1}{p}}+1\right]
^{\frac{n+p-\beta}{p-1}}\frac{(\frac{2^j}{\gamma})
^\frac{2(n+p)(1+\epsilon-\widetilde{\beta})}{p}}
{(\frac{2^j}{\gamma})^{\frac{(n+p)\epsilon}{p}}-1}
e^{-\frac{p-1}{p}(\frac{2^j}{2p\gamma})^{\frac{1}{p-1}}}
\|f\|_{L^r(\mathbb{R}^{n+1},\omega^r)}\\
&\quad\lesssim\|f\|_{L^r(\mathbb{R}^{n+1},\omega^r)},
\end{align*}
and hence \eqref{20240918.2142} holds, which completes the
proof of Theorem
\ref{parabolic Riesz potantial weighted boundedness Gopala Rao}.
\end{proof}

\begin{remark}\label{remark on parabolic Riesz potantial weighted boundedness Gopala Rao}
\begin{enumerate}
\item[(i)] Theorem
\ref{parabolic Riesz potantial weighted boundedness Gopala Rao}
when $\omega\equiv1$ coincides with \cite[Theorem
3.1]{g(ajm-1977)} .

\item[(ii)] For any given $E\subset\mathbb{R}^{n+1}$ and
$1<r<q<\infty$, define $\mathcal{A}_{r,q}(E)$ to be the set of
all nonnegative locally integrable functions $\omega$ on $E$
such that
\begin{align*}
[\omega]_{\mathcal{A}_{r,q}(E)}:=
\sup_{\genfrac{}{}{0pt}{}{R\in\mathcal{R}_p^{n+1}}{R\subset
E}}\fint_R\omega\left(\fint_R\omega^{-r'}\right)
^{\frac{q}{r'}}<\infty.
\end{align*}
Obviously, $\mathcal{A}_{r,q}(\mathbb{R}^{n+1})\subset
A_{r,q}^+(\mathbb{R}^{n+1})$. By \eqref{20240923.2105}, the
fact that $(\mathbb{R}^{n+1},d_p,|\cdot|)$ is a space of
homogeneous type, and  the weighted boundedness of fractional
integrals on spaces of homogeneous type (see, for instance,
\cite[Theorem 3.3]{k(ms-2014)}), we conclude that, if we
replace the condition that there exist $u\in
A_{r,q}(\mathbb{R}^n)$ and $v\in A_{r,q}^+(\mathbb{R})$
satisfying $\omega(x,t)=u(x)v(t)$ by
$\omega\in\mathcal{A}_{r,q}(\mathbb{R}^{n+1})$, then the
conclusion of Theorem
\ref{parabolic Riesz potantial weighted boundedness Gopala Rao}
still holds, that is, both $\mathcal{I}_\beta^{0+}$ and
$\mathcal{G}_\beta^{0+}$ are bounded from
$L^r(\mathbb{R}^{n+1},\omega^r)$ to
$L^q(\mathbb{R}^{n+1},\omega^q)$.
\end{enumerate}
\end{remark}

In what follows, we fix $p=2$. Let
$\beta\in(0,n+2)$, $q\in[1,\infty)$, and $\omega$ be a
weight on $\mathbb{R}^{n+1}$.
The \emph{weighted parabolic Sobolev space
$W^{\beta,q}(\mathbb{R}^{n+1},\omega)$} is defined by
setting
\begin{align*}
W^{\beta,q}(\mathbb{R}^{n+1},\omega):=\left\{h_\beta*g:\
g\in L^q(\mathbb{R}^{n+1},\omega)\right\}.
\end{align*}
For any $f\in W^{\beta,q}(\mathbb{R}^{n+1},\omega)$,
define $\|f\|_{W^{\beta,q}(\mathbb{R}^{n+1},\omega)}
:=\|g\|_{L^q(\mathbb{R}^{n+1},\omega)}$, where $g\in
L^q(\mathbb{R}^{n+1},\omega)$ satisfying that $f=h_\beta*g$.
Theorem \ref{parabolic Riesz potantial weighted boundedness Gopala Rao}
and Remark \ref{remark on parabolic Riesz potantial weighted boundedness Gopala Rao}(ii)
immediately imply the following parabolic weighted Sobolev
embedding result and we omit the details.

\begin{corollary}\label{Sobolev embedding}
Let $\beta\in(0,n+2)$, $1<r<q<\infty$ with
$\frac{1}{r}-\frac{1}{q}=\frac{\beta}{n+2}$,
and $\omega$ be a weight on
$\mathbb{R}^{n+1}$. If either of the following two conditions
holds:
\begin{enumerate}
\item[\rm(i)] there exist $u\in A_{r,q}(\mathbb{R}^n)$ and
$v\in A_{r,q}^+(\mathbb{R})$ such that $\omega(x,t)=u(x)v(t)$;

\item[\rm(ii)] $\omega\in\mathcal{A}_{r,q}(\mathbb{R}^{n+1})$,
\end{enumerate}
then $W^{\beta,r}(\mathbb{R}^{n+1},\omega^r)\subset
L^q(\mathbb{R}^{n+1},\omega^q)$. Moreover, there exists
a positive constant $C$ such that,
for any $f\in W^{\beta,r}(\mathbb{R}^{n+1},\omega^r)$,
\begin{align*}
\|f\|_{L^q(\mathbb{R}^{n+1},\omega^q)}\leq
C\|f\|_{W^{\beta,r}(\mathbb{R}^{n+1},\omega^r)}.
\end{align*}
\end{corollary}

We denote by $\mathcal{S}(\mathbb{R}^{n+1})$ the space of
all \emph{Schwartz functions} on
$\mathbb{R}^{n+1}$ equipped with a well-known topology
determined by a countable family of norms and by $\mathcal{S}'(\mathbb{R}^{n+1})$
the space of all \emph{tempered
distributions} equipped with the
weak-$*$ topology. In addition, let
$\mathcal{S}'(\mathbb{R}^{n+1}_+):=
\mathcal{S}'(\mathbb{R}^{n+1})|_{\mathbb{R}^{n+1}_+}$,
namely the restriction of $\mathcal{S}'(\mathbb{R}^{n+1})$
on $\mathbb{R}^{n+1}_+$. Using both the Fourier transform
formula of $h_\beta$ with $\beta\in(0,n+2)$ (see
\cite[(2.4)]{g(ajm-1977)}) and several elementary properties of
the Fourier transform and replacing $\omega$ and $f$,
respectively, by $\omega(x,t)\boldsymbol{1}_{(0,\infty)}(t)$
and $f(x,t)\boldsymbol{1}_{(0,\infty)}(t)$ in Theorem
\ref{parabolic Riesz potantial weighted boundedness Gopala Rao}
and Remark
\ref{remark on parabolic Riesz potantial weighted boundedness Gopala Rao}(ii),
we obtain the following application of Theorem
\ref{parabolic Riesz potantial weighted boundedness Gopala Rao},
which presents a priori estimate for the
nonhomogeneous heat equations. We omit the details.

\begin{corollary}\label{priori estimate}
Let $1<r<q<\infty$ with $\frac{1}{r}-\frac{1}{q}=
\frac{2}{n+2}$, $\omega$ be a weight on $\mathbb{R}^{n+1}_+$,
and $f\in
L^r(\mathbb{R}^{n+1}_+,\omega^r)$. If either of the
following two conditions holds:
\begin{enumerate}
\item[\rm(i)] there exist $u\in A_{r,q}(\mathbb{R}^n)$ and
$v\in A_{r,q}^+(\mathbb{R}_+)$ such that
$\omega(x,t)=u(x)v(t)$,

\item[\rm(ii)] $\omega\in\mathcal{A}_{r,q}
(\mathbb{R}^{n+1}_+)$
\end{enumerate}
and if $g\in\mathcal{S}'(\mathbb{R}^{n+1}_+)
\cap L^1_{\mathrm{loc}}(\mathbb{R}^{n+1}_+)$ satisfies
\begin{align*}
\displaystyle
\begin{cases}
\frac{\partial g}{\partial t}(x,t)-\Delta g(x,t)=f(x,t),&
(x,t)\in\mathbb{R}^{n+1}_+,\\
\displaystyle\lim_{t\to0^+}g(x,t)=0,&x\in\mathbb{R}^n,
\end{cases}
\end{align*}
then $g\in L^q(\mathbb{R}^{n+1}_+,\omega^q)$. Moreover, there
exists a positive constant $C$ such that
\begin{align*}
\|g\|_{L^q(\mathbb{R}^{n+1}_+,\omega^q)}\leq
C\|f\|_{L^r(\mathbb{R}^{n+1}_+,\omega^r)}.
\end{align*}
\end{corollary}

\noindent\textbf{Acknowledgements}\quad
After this article was finished, we found that Cruz-Uribe and
Myyryl\"ainen \cite{cm24} simultaneously also introduced the off-diagonal
parabolic Muckenhoupt class with time lag and studied
the two-weighted boundedness of the centered parabolic fractional maximal
operator with time lag. Except the definition and its
two basic properties of the off-diagonal parabolic
Muckenhoupt class with time lag,
both articles have no substantial overlap.

Weiyi Kong and Chenfeng Zhu would like to thank
Professor Sibei Yang for some helpful discussions and
suggestions on Section~\ref{section6} of this article.

\bigskip

\noindent Weiyi Kong, Dachun
Yang (Corresponding author) and Wen Yuan

\medskip

\noindent Laboratory of Mathematics
and Complex Systems (Ministry of Education of China),
School of Mathematical Sciences,
Beijing Normal University,
Beijing 100875, The People's Republic of China

\smallskip

\noindent{\it E-mails:} \texttt{weiyikong@mail.bnu.edu.cn} (W. Kong)

\noindent\phantom{{\it E-mails:} }\texttt{dcyang@bnu.edu.cn} (D. Yang)

\noindent\phantom{{\it E-mails:} }\texttt{wenyuan@bnu.edu.cn} (W. Yuan)

\bigskip

\noindent Chenfeng Zhu

\medskip

\noindent School of Mathematical Sciences, Zhejiang
University of Technology, Hangzhou 310023, The People's Republic of China

\smallskip

\noindent{\it E-mails:} \texttt{chenfengzhu@zjut.edu.cn}

\end{document}